\documentclass{amsart}

\usepackage{amsthm,amsmath,amssymb,mathtools,enumitem,tikz,pgf}

\usetikzlibrary{snakes, shapes,arrows,patterns,decorations.pathreplacing,matrix}

\newcommand{\op}{\operatorname}

\newtheorem{theorem}{Theorem}

\newtheorem{remark}{Remark}
\newtheorem{corollary}{Corollary}
\newtheorem{definition}{Definition}
\newtheorem{proposition}{Proposition}

\newcommand{\kk}{\mathbb{k}}
\newcommand{\QQ}{\mathbb{Q}}
\newcommand{\adj}{\operatorname{adj}}
\providecommand{\texorpdfstring}[2]{#1}
\newcommand{\im}{\operatorname{im}}
\newcommand{\rank}{\operatorname{rank}}
\newcommand{\tr}{\op{tr}}
\newcommand{\disc}{\op{disc}}

\pgfdeclarelayer{foreground}
\pgfsetlayers{main,foreground}

\tikzstyle{Triality}=[circle,draw,fill=orange,scale=0.5]
\tikzstyle{Triangle}=[regular polygon,regular polygon sides=3,draw,fill=orange]
\tikzstyle{Box}=[rectangle, fill=orange, draw, very thick, minimum height=0.25cm, minimum width=0.5cm]
\tikzstyle{Boxsm}=[rectangle, fill=orange, draw, thick, minimum height=0.15cm, minimum width=0.3cm]
\tikzstyle{blob}=[ellipse,minimum height=12pt, minimum width=24pt,scale=0.5,thick,draw,fill=orange]
\tikzstyle{Squiggle}=[decorate,decoration={snake}]
\tikzstyle{A}=[magenta,thick]
\tikzstyle{B}=[cyan,thick]
\tikzstyle{C}=[yellow,thick]
\tikzstyle{D}=[blue,thick]
\tikzstyle{Circ}=[circle,draw,thick,scale=0.5,fill=white]
\tikzstyle{blobA}=[ellipse,minimum height=12pt, minimum width=24pt,scale=0.6,very thick,draw,fill=magenta]
\tikzstyle{blobB}=[ellipse,minimum height=12pt, minimum width=24pt,scale=0.6,very thick,draw,fill=cyan]
\tikzstyle{blobC}=[ellipse,minimum height=12pt, minimum width=24pt,scale=0.6,very thick,draw,fill=yellow]

\newcommand{\triality}[2][]{\begin{pgfonlayer}{foreground}\node[Triality] (#1) at #2 {};\end{pgfonlayer}}
\newcommand{\Circ}[2][]{\begin{pgfonlayer}{foreground}\node[Circ] (#1) at #2 {};\end{pgfonlayer}}

\newcommand{\NS}[1]{.. controls +(north:#1) and +(south:#1) .. }
\newcommand{\WN}[1]{.. controls +(west:#1) and +(north:#1) .. }
\newcommand{\WS}[1]{.. controls +(west:#1) and +(south:#1) .. }
\newcommand{\SN}[1]{.. controls +(south:#1) and +(north:#1) .. }
\renewcommand{\SS}[1]{.. controls +(south:#1) and +(south:#1) .. }
\newcommand{\NN}[1]{.. controls +(north:#1) and +(north:#1) .. }
\newcommand{\NwSe}[1]{.. controls +(north west:#1) and +(south east:#1) .. }
\newcommand{\NeSw}[1]{.. controls +(north east:#1) and +(south west:#1) .. }
\newcommand{\SeNw}[1]{.. controls +(south east:#1) and +(north west:#1) .. }
\newcommand{\SwNe}[1]{.. controls +(south west:#1) and +(north east:#1) .. }
\newcommand{\NeSe}[1]{.. controls +(north east:#1) and +(south east:#1) .. }
\newcommand{\NwSw}[1]{.. controls +(north west:#1) and +(south west:#1) .. }

\newcommand{\SwSe}[1]{.. controls +(south west:#1) and +(south east:#1) .. }
\newcommand{\SeSw}[1]{.. controls +(south east:#1) and +(south west:#1) .. }
\newcommand{\NeNw}[1]{.. controls +(north east:#1) and +(north west:#1) .. }
\newcommand{\NwNe}[1]{.. controls +(north west:#1) and +(north east:#1) .. }
\newcommand{\SeSe}[1]{.. controls +(south east:#1) and +(south east:#1) .. }
\newcommand{\SwSw}[1]{.. controls +(south west:#1) and +(south west:#1) .. }
\newcommand{\NwNw}[1]{.. controls +(north west:#1) and +(north west:#1) .. }
\newcommand{\WNw}[1]{.. controls +(west:#1) and +(north west:#1) .. }

\newcommand{\skewBelow}[3][]{
  \path #2 ++(south:1/8) ++(west:1/8) coordinate (skewBelowA);
  \path #3 ++(south:1/8) ++(east:1/8) coordinate (skewBelowB);
  \draw[#1] (skewBelowA) -- (skewBelowB);
}

\newcommand{\skewAbove}[3][]{
  \path #2 ++(north:1/8)++(west:1/8) coordinate (skewBelowA);
  \path #3 ++(north:1/8)++(east:1/8) coordinate (skewBelowB);
  \draw[#1] (skewBelowA) -- (skewBelowB);
}

\newcommand{\symBelow}[2]{
  \path #1 ++(south:1/8)++(west:1/8) coordinate (skewBelowA);
  \path #2 ++(south:1/8)++(east:1/8) coordinate (skewBelowB);
  \draw[decorate,decoration={snake,segment length=12pt}] (skewBelowA) -- (skewBelowB);
}

\newcommand{\symAbove}[2]{
  \path #1 ++(north:1/8)++(west:1/8) coordinate (skewBelowA);
  \path #2 ++(north:1/8)++(east:1/8) coordinate (skewBelowB);
  \draw[decorate,decoration={snake,segment length=12pt}] (skewBelowA) -- (skewBelowB);
}

\usetikzlibrary{snakes, shapes,arrows,patterns,decorations.pathreplacing,matrix}

\tikzstyle{triality}=[circle, ,draw=black, fill=orange,scale=0.5]
\tikzstyle{triangle}=[regular polygon,regular polygon sides=3,draw]
\tikzstyle{Circ}=[circle,draw,thick,scale=0.5,fill=orange]
\def\bloba{\begin{tikzpicture}[baseline]
  \tikzstyle{circleentry}=[circle,thick,draw]
  \tikzstyle{el1}=[ellipse,minimum height=1/2cm, minimum width=1cm]
  \tikzstyle{triangle}=[regular polygon,regular polygon sides=3,draw]
\node[el1, thick, draw] (blob) at (0,1) {};
\coordinate (foot1) at (-1/2,-1);
\coordinate (foot2) at (1/4,-1);
\coordinate (foot3) at (1,-1);
\coordinate (triality1) at (3/4,0);
\node[triality] at (triality1) {};
\draw (node cs:name=blob,anchor=south west) -- (node cs:name=foot1);
\draw (node cs:name=blob,anchor=south east) -- (triality1);
\draw[decorate, decoration={snake}] (triality1) -- (foot2);
\draw[densely dotted, thick] (triality1) -- (foot3);
\end{tikzpicture}}

\def\blobb{\begin{tikzpicture}[baseline]
  \tikzstyle{circleentry}=[circle,thick,draw]
  \tikzstyle{el1}=[ellipse,minimum height=1/2cm, minimum width=1cm]
  \tikzstyle{triangle}=[regular polygon,regular polygon sides=3,draw]
\node[el1, thick, draw] (blob) at (0,1) {};
\coordinate (foot1) at (-1/2,-1);
\coordinate (foot2) at (1/4,-1);
\coordinate (foot3) at (1,-1);
\coordinate (triality1) at (3/4,0);
\node[triality] at (triality1) {};
\draw[decorate, decoration={snake}] (node cs:name=blob,anchor=south west) -- (foot2);
\draw[decorate, decoration={snake}] (node cs:name=blob,anchor=south east) -- (triality1);
\draw (triality1) -- (foot1);
\draw[densely dotted, thick] (triality1) -- (foot3);
\end{tikzpicture}}

\def\blobc{\begin{tikzpicture}[baseline]
  \tikzstyle{circleentry}=[circle,thick,draw]
  \tikzstyle{el1}=[ellipse,minimum height=1/2cm, minimum width=1cm]
  \tikzstyle{triangle}=[regular polygon,regular polygon sides=3,draw]
\node[el1, thick, draw] (blob) at (0,1) {};
\coordinate (foot1) at (-1/2,-1);
\coordinate (foot2) at (1/4,-1);
\coordinate (foot3) at (1,-1);
\coordinate (triality2) at (-1/4,0);
\node[triality] at (triality2) {};
\draw[densely dotted, thick] (node cs:name=blob,anchor=south east) -- (foot3);
\draw[densely dotted, thick] (node cs:name=blob,anchor=south west) -- (triality2);
\draw (triality2) -- (foot1);
\draw[decorate, decoration={snake}] (triality2) -- (foot2);
\end{tikzpicture}}

\def\triminus{\begin{tikzpicture}[baseline,scale=1.7]
\draw[yellow, thick] (-0.75,0.43) .. controls + (0.2,-0.36) and +(0.2, 0.3) .. (-0.75,-0.43);
\draw[magenta, thick] (-0.7,0.46) .. controls + (0.2,-0.36) and +(-0.4, 0) .. (0,0.02);
\draw[cyan, thick] (-0.7,-0.46) .. controls + (0.2,0.36) and +(-0.4, 0) .. (0,-0.02);
\end{tikzpicture}}

\def\triminustwo{\begin{tikzpicture}[baseline,scale=1.7]
\draw[yellow, thick] (-0.75,0.43) .. controls + (0.2,-0.36) and +(0.2, 0.3) .. (-0.75,-0.43);
\draw[magenta, thick] (-0.7,0.46) .. controls + (0.2,-0.36) and +(-0.4, 0) .. (0,0.02);
\draw[magenta, thick] (0.7,0.46) .. controls + (-0.2,-0.36) and +(0.4, 0) .. (0,0.02);
\draw[cyan, thick] (-0.7,-0.46) .. controls + (0.2,0.36) and +(-0.4, 0) .. (0,-0.02);
\draw[cyan, thick] (0.7,-0.46) .. controls + (-0.2,0.36) and +(0.4, 0) .. (0,-0.02);
\draw[yellow, thick] (0.75,0.43) .. controls + (-0.2,-0.36) and +(-0.2, 0.36) .. (0.75,-0.43);
\end{tikzpicture}}
\def\triminustwoflip{\begin{tikzpicture}[baseline,scale=1.7]
\draw[magenta, thick] (-0.7,0.46) .. controls + (0.2,-0.36) and +(0.7, 0) .. (0,0.06);
\draw[magenta, thick] (0.7,0.46) .. controls + (-0.2,-0.36) and +(-0.7, 0) .. (0,0.06);
\draw[cyan, thick] (-0.7,-0.46) .. controls + (0.2,0.36) and +(-0.4, 0) .. (0,-0.055);
\draw[cyan, thick] (0.7,-0.46) .. controls + (-0.2,0.36) and +(0.4, 0) .. (0,-0.055);
\draw[yellow, thick] (-0.75,0.43) .. controls + (0.2,-0.36) and +(-0.4, 0.04) .. (0,0);
\draw[yellow, thick] (0.75,-0.43) .. controls + (-0.2,0.36) and +(0.4, -0.04) .. (0,0);
\draw[yellow, thick] (0.75,0.43) .. controls + (-0.25,-0.45) and +(0.45, 0.05) .. (0,0);
\draw[yellow, thick] (-0.75,-0.43) .. controls + (0.25,0.45) and +(-0.45, -0.05) .. (0,0);
\end{tikzpicture}}

\def\triminustwoflipright{\begin{tikzpicture}[baseline,scale=1.7]
\draw[magenta, thick] (-0.7,0.46) .. controls + (0.2,-0.36) and +(-0.4, 0) .. (0,0.06);
\draw[magenta, thick] (0.7,0.46) .. controls + (-0.2,-0.36) and +(0.4, 0) .. (0,0.06);
\draw[yellow, thick] (-0.75,0.43) .. controls + (0.2,-0.36) and +(-0.4, 0) .. (0,0.03);
\draw[yellow, thick] (0.75,0.43) .. controls + (-0.2,-0.36) and +(0.4, 0) .. (0,0.03);
\draw[cyan, thick] (-0.7,-0.46) .. controls + (0.2,0.36) and +(-0.4, 0) .. (0,-0.04);
\draw[cyan, thick] (0.7,-0.46) .. controls + (-0.2,0.36) and +(0.4, 0) .. (0,-0.04);
\draw[yellow, thick] (-0.75,-0.43) .. controls + (0.2,0.36) and +(-0.4, 0) .. (0,-0.01);
\draw[yellow, thick] (0.75,-0.43) .. controls + (-0.2,0.36) and +(0.4, 0) .. (0,- 0.01);
\end{tikzpicture}}

\def\tri{\begin{tikzpicture}[baseline,scale=2]
  \path{ (0, 0) edge[thick, bend left=40, magenta] (-1/3,-0.58)};
  \path{ (0, 0) edge[thick, bend left=40, cyan]  (2/3,0)};
  \path{ (0, 0) edge[thick, bend left=40, yellow] (-1/3, 0.58)};
  \node[triality] at (0, 0) {};
\end{tikzpicture}}

\def\tria{
\begin{tikzpicture}[baseline,scale=2]
  \coordinate (triality) at (0,1/2);
  \coordinate (mid1) at (-1/4,0);
  \coordinate (mid2) at (0,0);
  \coordinate (mid3) at (1/4,0);
  \coordinate (foot1) at (-1/4,-1/2);
  \coordinate (foot2) at (0,-1/2);
  \coordinate (foot3) at (1/4,-1/2);
  \path{ 
    (triality) edge[bend right=25] (mid1) 
    (mid1) edge (foot1) 
  };
  \path{ 
    (triality) edge[decorate,decoration={snake}] (mid2) 
    (mid2) edge[decorate,decoration={snake}] (foot2) 
  };
  \path{ 
    (triality) edge[bend left=25,densely dotted, thick] (mid3) 
    (mid3) edge[densely dotted, thick] (foot3) 
  };
  \node[triality] at (triality) {};
  \draw (triality)  (mid2) edge[decorate, decoration={snake}] (foot2);
  \node[circle,fill=white,draw] at (mid1) {};
\end{tikzpicture}
}

\def\trib{
\begin{tikzpicture}[baseline,scale=2]
  \coordinate (triality) at (0,1/2);
  \coordinate (mid1) at (-1/4,0);
  \coordinate (mid2) at (0,0);
  \coordinate (mid3) at (1/4,0);
  \coordinate (foot1) at (-1/4,-1/2);
  \coordinate (foot2) at (0,-1/2);
  \coordinate (foot3) at (1/4,-1/2);
  \path{ 
    (triality) edge[bend right=25] (mid1) 
    (mid1) edge (foot1) 
  };
  \path{ 
    (triality) edge[decorate,decoration={snake}] (mid2) 
    (mid2) edge[decorate,decoration={snake}] (foot2) 
  };
  \path{ 
    (triality) edge[bend left=25,densely dotted, thick] (mid3) 
    (mid3) edge[densely dotted, thick] (foot3) 
  };
  \node[triality] at (triality) {};
  \draw (triality)  (mid2) edge[decorate, decoration={snake}] (foot2);
  \node[circle,fill=white,draw] at (mid2) {};
\end{tikzpicture}
}

\def\tric{
\begin{tikzpicture}[baseline,scale=2]
  \coordinate (triality) at (0,1/2);
  \coordinate (mid1) at (-1/4,0);
  \coordinate (mid2) at (0,0);
  \coordinate (mid3) at (1/4,0);
  \coordinate (foot1) at (-1/4,-1/2);
  \coordinate (foot2) at (0,-1/2);
  \coordinate (foot3) at (1/4,-1/2);
  \path{ 
    (triality) edge[bend right=25] (mid1) 
    (mid1) edge (foot1) 
  };
  \path{ 
    (triality) edge[decorate,decoration={snake}] (mid2) 
    (mid2) edge[decorate,decoration={snake}] (foot2) 
  };
  \path{ 
    (triality) edge[bend left=25,densely dotted, thick] (mid3) 
    (mid3) edge[densely dotted, thick] (foot3) 
  };
  \node[triality] at (triality) {};
  \draw (triality)  (mid2) edge[decorate, decoration={snake}] (foot2);
  \node[circle,fill=white,draw] at (mid3) {};
\end{tikzpicture}
}

\def\blackboxa{\begin{tikzpicture}[baseline]
\node[rectangle, fill=gray!50, draw, very thick, minimum height=0.25cm, minimum width=0.5cm] (box) at (0,0) {};
\draw (box.south west) -- +(south:1/2);
\draw (box.south east) -- +(south:1/2);
\draw[decorate,decoration={snake}] (box.north west) -- +(north:1/2);
\draw[decorate,decoration={snake}] (box.north east) -- +(north:1/2);
\end{tikzpicture}}

\def\blackboxls{\begin{tikzpicture}[baseline]
\node[rectangle, fill=gray!50, draw, very thick, minimum height=0.25cm, minimum width=0.5cm] (box) at (0,0) {};
\draw (box.south west) -- +(south:1/2);
\draw (box.south east) -- +(south:1/2);
\draw[decorate,decoration={snake}] (box.north west) -- +(north:1/2);
\draw[decorate,decoration={snake}] (box.north east) -- +(north:1/2);
\end{tikzpicture}}

\def\blackboxld{\begin{tikzpicture}[baseline = - 2pt]
\node[rectangle, fill=orange!50, draw, very thick, minimum height=0.25cm, minimum width=0.25cm] (box) at (0,0) {};
\draw[solid, thick, cyan] (box.south west) -- +(south west:1/2);
\draw[solid, thick, cyan] (box.south east) -- +(south east:1/2);
\draw[solid, thick, magenta] (box.north west) -- +(north west:1/2);
\draw[solid, thick, magenta] (box.north east) -- +(north east:1/2);
\end{tikzpicture}}

\def\blackboxldr{\begin{tikzpicture}[baseline = - 2pt]
\node[rectangle, fill=orange!50, draw, very thick, minimum height=0.25cm, minimum width=0.25cm] (box) at (0,0) {};
\draw[solid, thick, yellow] (box.south west) -- +(south west:1/2);
\draw[solid, thick, yellow] (box.south east) -- +(south east:1/2);
\draw[solid, thick, cyan] (box.north west) -- +(north west:1/2);
\draw[solid, thick, cyan] (box.north east) -- +(north east:1/2);
\end{tikzpicture}}

\def\blackboxlddd{\begin{tikzpicture}[baseline = - 2pt]
\node[rectangle, fill=orange!50, draw, very thick, minimum height=0.25cm, minimum width=0.25cm] (box) at (0,0) {};
\draw[solid, thick, cyan] (box.south west) -- +(south west:1/2);
\draw[solid, thick, cyan] (box.south east) -- +(south east:1/2);
\draw[solid, thick, magenta] (box.north west) -- +(north west:1/2);
\draw[solid, thick, magenta] (box.north east) -- +(north east:1/2);
\end{tikzpicture}}

\def\blackboxldrrr{\begin{tikzpicture}[baseline = - 2pt]
\node[rectangle, fill=orange!50, draw, very thick, minimum height=0.25cm, minimum width=0.25cm] (box) at (0,0) {};
\draw[solid, thick, yellow] (box.south west) -- +(south west:1/2);
\draw[solid, thick, yellow] (box.south east) -- +(south east:1/2);
\draw[solid, thick, cyan] (box.north west) -- +(north:1/4);
\draw[solid, thick, cyan] (box.north east) -- +(north:1/4);
\end{tikzpicture}}

\def\blackboxsl{\begin{tikzpicture}[baseline]
\node[rectangle, fill=gray!50, draw, very thick, minimum height=0.25cm, minimum width=0.5cm] (box) at (0,0) {};
\draw[decorate,decoration={snake}] (box.south west) -- +(south:1/2);
\draw[decorate,decoration={snake}] (box.south east) -- +(south:1/2);
\draw (box.north west) -- +(north:1/2);
\draw (box.north east) -- +(north:1/2);
\end{tikzpicture}}

\def\blackboxsd{\begin{tikzpicture}[baseline]
\node[rectangle, fill=gray!50, draw, very thick, minimum height=0.25cm, minimum width=0.5cm] (box) at (0,0) {};
\draw[decorate,decoration={snake}] (box.south west) -- +(south:1/2);
\draw[decorate,decoration={snake}] (box.south east) -- +(south:1/2);
\draw[densely dotted, thick] (box.north west) -- +(north:1/2);
\draw[densely dotted, thick] (box.north east) -- +(north:1/2);
\end{tikzpicture}}

\def\blackboxlseq{\begin{tikzpicture}[baseline]
\coordinate (triality1) at (0,0);
\coordinate (triality2) at (1/2,0);
\draw (triality1) -- +(south:1/2);
\draw (triality2) -- +(south:1/2);
\draw[decorate,decoration={snake}] (triality1) -- +(north:1/2);
\draw[decorate,decoration={snake}] (triality2) -- +(north:1/2);
\draw[densely dotted, thick] (triality1) -- (triality2);
\draw (triality1)++(-1/8,-1/4) -- +(3/4,0);
\node[triality] at (triality1) {};
\node[triality] at (triality2) {};
\end{tikzpicture}}

\def\hookedlssl{
\begin{tikzpicture}[baseline]
\node at (0,0) {\blackboxls};
\node at (0,1.25) {\blackboxsl};
\end{tikzpicture}}

\def\hookedlssd{
\begin{tikzpicture}[baseline]
\node at (0,-5/8) {\blackboxls};
\node at (0,5/8) {\blackboxsd};
\end{tikzpicture}}

\def\blackboxsdeq{\begin{tikzpicture}[baseline]
\coordinate (triality1) at (0,0);
\coordinate (triality2) at (1/2,0);
\draw[decorate,decoration={snake}] (triality1) -- +(south:1/2);
\draw[decorate,decoration={snake}] (triality2) -- +(south:1/2);
\draw[densely dotted, thick] (triality1) -- +(north:1/2);
\draw[densely dotted, thick] (triality2) -- +(north:1/2);
\draw (triality1) -- (triality2);
\node[triality] at (triality1) {};
\node[triality] at (triality2) {};
\end{tikzpicture}}

\def\blackboxlssdeq{\begin{tikzpicture}[baseline]
\node at (0,-4/8) {\blackboxlseq};
\node at (0,4/8) {\blackboxsdeq};
\end{tikzpicture}}

\def\metric{\begin{tikzpicture}[baseline]
\draw[magenta, thick] (-1/2,1/2) .. controls +(0,-1) and +(0,-1) .. (1/2,1/2);
\end{tikzpicture}}
\def\metricswitch{\begin{tikzpicture}[baseline]
\path[use as bounding box] (-1/2,-1) rectangle (1/2,1/2);
\draw[magenta, thick] (-1/2,1/2) .. controls +(2,-1) and +(-2,-1) .. (1/2,1/2);
\end{tikzpicture}}

\def\tracemetric{\begin{tikzpicture}[baseline=8pt, scale = 0.8]
\draw[magenta, thick, decorate] (-1/2,1/2) .. controls +(0,-1) and +(0,-1) .. (1/2,1/2);
\draw[magenta, thick, decorate] (-1/2,1/2) .. controls +(0,1) and +(0,1) .. (1/2,1/2);
\end{tikzpicture}}
\def\metricswitch{\begin{tikzpicture}[baseline]
\path[use as bounding box] (-1/2,-1) rectangle (1/2,1/2);
\draw[magenta, thick] (-1/2,1/2) .. controls +(2,-1) and +(-2,-1) .. (1/2,1/2);
\end{tikzpicture}}

\def\trialityLHS{
\begin{tikzpicture}[baseline = -2pt]
\coordinate (triality1) at (0,0);
\coordinate (triality2) at (1/2,0);
\draw[solid, thick, cyan] (triality1) -- +(south west:3/4);
\draw[solid, thick, cyan]  (triality2) -- +(south east:3/4);
\draw[solid, thick, magenta] (triality1) -- +(north west:3/4);
\draw[solid, thick, magenta] (triality2) -- +(north east:3/4);
\draw[thick, yellow] (triality1) -- (triality2);
\node[triality] at (triality1) {};
\node[triality] at (triality2) {};
\end{tikzpicture}
}
\def\trialityLHSr{
\begin{tikzpicture}[baseline = -2pt]
\coordinate (triality1) at (0,0);
\coordinate (triality2) at (1/2,0);
\draw[solid, thick, yellow] (triality1) -- +(south west:3/4);
\draw[solid, thick, yellow]  (triality2) -- +(south east:3/4);
\draw[solid, thick, cyan] (triality1) -- +(north west:3/4);
\draw[solid, thick, cyan] (triality2) -- +(north east:3/4);
\draw[thick, magenta] (triality1) -- (triality2);
\node[triality] at (triality1) {};
\node[triality] at (triality2) {};
\end{tikzpicture}
}

\def\trialityRHSa{
\begin{tikzpicture}[baseline = - 2pt]
\draw[solid, thick, magenta] (-1/2,1/2) .. controls +(0,-1/2) and +(0,-1/2) .. (1/2,1/2);
\draw[solid, thick, cyan] (-1/2,-1/2) .. controls +(0,1/2) and +(0,1/2) .. (1/2,-1/2);
\end{tikzpicture}
}

\def\trialityRHSar{
\begin{tikzpicture}[baseline = - 2pt]
\draw[solid, thick, cyan] (-1/2,1/2) .. controls +(0,-1/2) and +(0,-1/2) .. (1/2,1/2);
\draw[solid, thick, yellow] (-1/2,-1/2) .. controls +(0,1/2) and +(0,1/2) .. (1/2,-1/2);
\end{tikzpicture}
}

\def\trialityRHSb{
\begin{tikzpicture}[baseline = - 2pt]
\coordinate (triality1) at (0,0);
\coordinate (triality2) at (1/2,0);
\draw[solid, thick, cyan] (triality1) -- +(south west:3/4);
\draw[solid, thick, cyan]  (triality2) -- +(south east:3/4);
\draw[solid, thick, magenta] (triality1) -- +(1,0.52);
\draw[solid, thick, magenta] (triality2) -- +(-1,0.52);
\draw[thick, yellow] (triality1) -- (triality2);
\node[triality] at (triality1) {};
\node[triality] at (triality2) {};
\end{tikzpicture}
}

\def\blackboxdl{\begin{tikzpicture}[baseline]
\node[rectangle, fill=gray!50, draw, very thick, minimum height=0.25cm, minimum width=0.5cm] (box) at (0,0) {};
\draw[densely dotted, thick] (box.south west) -- +(south:1/2);
\draw[densely dotted, thick] (box.south east) -- +(south:1/2);
\draw (box.north west) -- +(north:1/2);
\draw (box.north east) -- +(north:1/2);
\end{tikzpicture}}

\def\blackboxldeq{\begin{tikzpicture}[baseline]
\trialityLHS
\coordinate (triality1) at (0,-1/4);
\coordinate (triality2) at (0,1/4);
\draw[solid, thick, orange] (triality1)++(-3/2,0) -- +(11/8,0);
\draw[solid, thick, orange] (triality2)++(-3/2,0) -- +(11/8,0);
\end{tikzpicture}}

\def\twotrialitiestraced{\begin{tikzpicture}[baseline = - 2pt]
\coordinate (triality1) at (0,5/12);
\coordinate (triality2) at (0,-5/12);
\draw[thick, magenta]  (triality1) -- +(north:5/12);
\draw[thick, magenta]  (triality2) -- +(south:5/12);
\path (triality1) edge[thick, yellow,bend left=90] (triality2);
\path (triality1) edge[thick, cyan, thick, bend right=90] (triality2);
\node[triality] at (triality1) {};
\node[triality] at (triality2) {};
\end{tikzpicture}}

\def\identityl{\begin{tikzpicture}[baseline = - 2pt]
\draw[thick, magenta] (0,-5/12)--(0,5/12);
\end{tikzpicture}
}
\def\threetrialitiestraced{\begin{tikzpicture}[baseline = - 2pt]
\coordinate (triality1) at (0,5/12);
\coordinate (triality2) at (0,-5/12);
\draw[thick, magenta]  (triality1) -- (triality2);
\path (triality1) edge[thick, yellow,bend left=90] (triality2);
\path (triality1) edge[thick, cyan, thick, bend right=90] (triality2);
\node[triality] at (triality1) {};
\node[triality] at (triality2) {};
\end{tikzpicture}}

\def\triangleaa{\begin{tikzpicture}[baseline = -5]
\coordinate (A) at (0,1/2);
\coordinate (B) at (1/2,-1/2);
\coordinate (C) at (-1/2,-1/2);
\draw[thick, yellow] (A) -- (B);
\draw[magenta, thick] (B) -- (C);
\draw[cyan, thick] (C) -- (A);
\path (A) edge[magenta, thick] +(0,1);
\path (B) edge[cyan, thick] +(0,-1);
\path (C) edge[yellow, thick] +(0,-1);
\node[triality] at (A) {};
\node[triality] at (B) {};
\node[triality] at (C) {};
\end{tikzpicture}}

\def\trianglebb{\begin{tikzpicture}[baseline = - 5]
\coordinate (A) at (0,1/2);
\coordinate (B) at (1/2,-1/2);
\coordinate (C) at (-1/2,-1/2);
\path (B) edge[bend right=90, thick, yellow] (C);
\draw[magenta, thick] (B) -- (C);
\draw[cyan, thick] (C) -- (A);
\path (A) edge[magenta, thick] +(0,1);
\path (B) edge[cyan, thick] +(0,-1);
\path{
  (C) ++(0,-1) edge[bend left=80, yellow, thick] (A)
};
\node[triality] at (A) {};
\node[triality] at (B) {};
\node[triality] at (C) {};
\end{tikzpicture}}

\vspace{1cm}

\def\trianglecc{\begin{tikzpicture}[baseline = - 5]
\coordinate (A) at (0,1/2);
\coordinate (B) at (1/2,-1/2);
\coordinate (C) at (-1/2,-1/2);
\path (A) +(0,1) edge[magenta, thick,bend right=100] (B);
\path (C) +(0,-1) edge[bend left=110, yellow] (B);
\path (B) edge[cyan, thick] +(0,-1);
\node[triality] at (B) {};
\end{tikzpicture}
}
\def\triangledd{\begin{tikzpicture}[baseline = -1]
\draw[thick, magenta] (0, 0) -- +(0,0.96);
\draw[thick, cyan] (0, 0) -- +(2/3,-2/3);
\draw[thick, yellow] (0, 0) -- +(-2/3,-2/3);
\node[triality] at (0, 0) {};
\end{tikzpicture}}

\def\diagboxa{\begin{tikzpicture}[baseline = -2]
\coordinate (A) at (-1/4,-1/4);
\coordinate (B) at (1/4,-1/4);
\coordinate (C) at (1/4,1/4);
\coordinate (D) at (-1/4,1/4);
\node[triality] at (A) {};
\node[triality] at (B) {};
\node[triality] at (C) {};
\node[triality] at (D) {};
\draw[cyan, thick] (A) -- (B);
\draw[thick, yellow] (B) -- (C);
\draw[cyan, thick] (C) -- (D);
\draw[thick, yellow] (D) -- (A);
\draw[magenta, thick]  (A) -- +(0,-1/2);
\draw[magenta, thick]  (B) -- +(0,-1/2);
\draw[magenta, thick]  (C) -- +(0,1/2);
\draw[magenta, thick]  (D) -- +(0,1/2);
\node[triality] at (A) {};
\node[triality] at (B) {};
\node[triality] at (C) {};
\node[triality] at (D) {};
\end{tikzpicture} }

\def\diagboxb{\begin{tikzpicture}[baseline = -2]
\coordinate (A) at (-1/4,-1/4);
\coordinate (B) at (1/4,-1/4);
\coordinate (C) at (1/4,1/4);
\coordinate (D) at (-1/4,1/4);
\node[triality] at (A) {};
\node[triality] at (B) {};
\node[triality] at (C) {};
\node[triality] at (D) {};
\draw[yellow, thick] (A) -- (B);
\draw[thick, cyan] (B) -- (C);
\draw[yellow, thick] (C) -- (D);
\draw[thick, cyan] (D) -- (A);
\draw[magenta, thick]  (A) -- +(0,-1/2);
\draw[magenta, thick]  (B) -- +(0,-1/2);
\draw[magenta, thick]  (C) -- +(0,1/2);
\draw[magenta, thick]  (D) -- +(0,1/2);
\node[triality] at (A) {};
\node[triality] at (B) {};
\node[triality] at (C) {};
\node[triality] at (D) {};
\end{tikzpicture} }

\def\diagboxf{\begin{tikzpicture}[baseline = -2]
\coordinate (A) at (-1/4,-1/4);
\coordinate (B) at (1/4,-1/4);
\coordinate (C) at (1/4,1/4);
\coordinate (D) at (-1/4,1/4);
\node[triality] at (A) {};
\node[triality] at (B) {};
\node[triality] at (C) {};
\node[triality] at (D) {};
\draw[yellow, thick] (A) -- (C);
\draw[thick, cyan] (B) -- (C);
\draw[yellow, thick] (B) -- (D);
\draw[thick, cyan] (D) -- (A);
\draw[magenta, thick]  (A) -- +(0,-1/2);
\draw[magenta, thick]  (B) -- +(0,-1/2);
\draw[magenta, thick]  (C) -- +(0,1/2);
\draw[magenta, thick]  (D) -- +(0,1/2);
\node[triality] at (A) {};
\node[triality] at (B) {};
\node[triality] at (C) {};
\node[triality] at (D) {};
\end{tikzpicture} }

\def\diagboxh{\begin{tikzpicture}[baseline = -2]
\coordinate (A) at (-1/4,-1/4);
\coordinate (B) at (1/4,-1/4);
\coordinate (C) at (1/4,1/4);
\coordinate (D) at (-1/4,1/4);
\node[triality] at (A) {};
\node[triality] at (B) {};
\node[triality] at (C) {};
\node[triality] at (D) {};
\draw[yellow, thick] (A) -- (C);
\draw[thick, cyan] (B) -- (A);
\draw[yellow, thick] (B) -- (D);
\draw[thick, cyan] (C) -- (D);
\draw[magenta, thick]  (A) -- +(0,-1/2);
\draw[magenta, thick]  (B) -- +(0,-1/2);
\draw[magenta, thick]  (C) -- +(0,1/2);
\draw[magenta, thick]  (D) -- +(0,1/2);
\node[triality] at (A) {};
\node[triality] at (B) {};
\node[triality] at (C) {};
\node[triality] at (D) {};
\end{tikzpicture} }

\def\diagboxl{\begin{tikzpicture}[baseline = -2]
\coordinate (A) at (-1/4,-1/4);
\coordinate (B) at (1/4,-1/4);
\coordinate (C) at (1/4,1/4);
\coordinate (D) at (-1/4,1/4);
\node[triality] at (A) {};
\node[triality] at (B) {};
\node[triality] at (C) {};
\node[triality] at (D) {};
\draw[yellow, thick] (A) -- (D);
\draw[thick, cyan] (B) -- (A);
\draw[yellow, thick] (B) -- (C);
\draw[thick, cyan] (C) -- (D);
\draw[magenta, thick]  (A) -- +(0,-1/2);
\draw[magenta, thick]  (B) -- +(0,-1/2);
\draw[magenta, thick]  (C) -- +(-1/2,1/2);
\draw[magenta, thick]  (D) -- +(1/2,1/2);
\node[triality] at (A) {};
\node[triality] at (B) {};
\node[triality] at (C) {};
\node[triality] at (D) {};
\end{tikzpicture} }

\def\diagboxc{\begin{tikzpicture}[baseline = -2,scale= 0.9]
\coordinate(A) at (-1/6,1/2);
\coordinate(B) at (-1/6,-1/2);
\coordinate(C) at (1/6,-1/2);
\coordinate(D) at (1/6,1/2);
\draw[magenta, thick]  (A) -- (B);
\draw[magenta, thick]  (C) -- (D);
\end{tikzpicture}}

\def\diagboxd{\begin{tikzpicture}[baseline= -2,scale=0.9]
\coordinate(A) at (-1/6,1/2);
\coordinate(B) at (-1/6,-1/2);
\coordinate(C) at (1/6,-1/2);
\coordinate(D) at (1/6,1/2);
\draw[magenta, thick]  (A) -- (C);
\draw[magenta, thick]  (B) -- (D);
\end{tikzpicture}}

\def\diagboxe{\begin{tikzpicture}[baseline = -2,scale=0.9]
\draw[magenta, thick]  (-1/2,1/2) .. controls +(0,-1/2) and +(0,-1/2) .. (1/2,1/2);
\draw[magenta, thick]  (-1/2,-1/2) .. controls +(0,1/2) and +(0,1/2) .. (1/2,-1/2);
\end{tikzpicture}}

\def\diagboxg{\begin{tikzpicture}[baseline = - 2]
\coordinate (triality0) at (-1/2,3/4);
\coordinate (triality3) at (-1/2,-3/4);
\coordinate (triality1) at (0,1/4);
\coordinate (triality2) at (0,-1/4);
\draw[magenta, thick] (triality1) -- +(north:1/2);
\draw[magenta, thick]  (triality2) -- +(south:1/2);
\path (triality1) edge[cyan, thick,bend left=100] (triality2);
\path (triality1) edge[yellow, thick, bend right=100] (triality2);
\path (triality0) edge[magenta, thick] (triality3);
\node[triality] at (triality1) {};
\node[triality] at (triality2) {};
\end{tikzpicture}}

\def\diagboxhfail{\begin{tikzpicture}[baseline]
\coordinate (A) at (-1/4,-1/4);
\coordinate (B) at (1/4,-1/4);
\coordinate (C) at (1/4,1/4);
\coordinate (D) at (-1/4,1/4);
\node[triality] at (A) {};
\node[triality] at (B) {};
\node[triality] at (C) {};
\node[triality] at (D) {};
\draw[decorate,decoration={snake}] (A) -- (C);
\draw[densely dotted, thick] (C) -- (D);
\draw[decorate,decoration={snake}] (B) -- (D);
\draw[densely dotted, thick] (A) -- (B);
\draw (A) -- +(0,-1/2);
\draw (B) -- +(0,-1/2);
\draw (C) -- +(0,1/2);
\draw (D) -- +(0,1/2);
\end{tikzpicture}}

\def\diagboxj{\begin{tikzpicture}[baseline = -2]
\coordinate (A) at (-1/4,-1/4);
\coordinate (B) at (1/4,-1/4);
\coordinate (C) at (1/4,1/4);
\coordinate (D) at (-1/4,1/4);
\draw[magenta, thick] (B) -- +(0,-1/2);
\draw[magenta, thick] (D) -- +(0,1/2);
\path (B) edge[cyan, thick,bend left=75] (D);
\path (B) edge[yellow,bend right=75] (D);
\draw[magenta, thick] (-1/4,-3/4)--(1/4,3/4);
\node[triality] at (B) {};
\node[triality] at (D) {};
\end{tikzpicture}}

\def\diagboxk{\begin{tikzpicture}[baseline = -2]
\coordinate(A) at (-1/8,1/2);
\coordinate(B) at (-1/8,-1/2);
\coordinate(C) at (1/8,-1/2);
\coordinate(D) at (1/8,1/2);
\draw (A) -- (B);
\draw (C) -- (D);
\end{tikzpicture}}

\def\diagboxlfail{\begin{tikzpicture}[baseline]
\coordinate (A) at (-1/4,-1/4);
\coordinate (B) at (1/4,-1/4);
\coordinate (C) at (1/4,1/4);
\coordinate (D) at (-1/4,1/4);
\node[triality] at (A) {};
\node[triality] at (B) {};
\node[triality] at (C) {};
\node[triality] at (D) {};
\draw[densely dotted, thick] (A) -- (B);
\draw[decorate,decoration={snake}] (B) -- (C);
\draw[densely dotted, thick] (C) -- (D);
\draw[decorate,decoration={snake}] (D) -- (A);
\draw (A) -- +(0,-1/2);
\draw (B) -- +(0,-1/2);
\draw (C) -- +(-1/2,1/2);
\draw (D) -- +(1/2,1/2);
\end{tikzpicture}}

\def\diagboxm{\begin{tikzpicture}[baseline = - 2]
\coordinate(A) at (-1/8,1/2);
\coordinate(B) at (-1/8,-1/2);
\coordinate(C) at (1/8,-1/2);
\coordinate(D) at (1/8,1/2);
\draw (A) -- (B);
\draw (C) -- (D);
\end{tikzpicture}}

\def\diagboxn{\begin{tikzpicture}[baseline]
\coordinate(A) at (-1/8,1/2);
\coordinate(B) at (-1/8,-1/2);
\coordinate(C) at (1/8,-1/2);
\coordinate(D) at (1/8,1/2);
\draw (A) -- (C);
\draw (B) -- (D);
\end{tikzpicture}}

\def\diagboxp{\begin{tikzpicture}[baseline]
\coordinate (A) at (-1/4,-1/4);
\coordinate (B) at (1/4,-1/4);
\coordinate (C) at (1/4,1/4);
\coordinate (D) at (-1/4,1/4);
\node[triality] at (A) {};
\node[triality] at (B) {};
\node[triality] at (C) {};
\node[triality] at (D) {};
\draw[densely dotted, thick] (A) -- (B);
\draw[decorate,decoration={snake}] (B) -- (C);
\draw[densely dotted, thick] (C) -- (D);
\draw[decorate,decoration={snake}] (D) -- (A);
\draw (A) -- +(0,-1/2);
\draw (B) -- +(0,-1/2);
\draw (C) -- +(0,1/2);
\draw (D) -- +(0,1/2);
\end{tikzpicture}}

\def\diagboxq{\begin{tikzpicture}[baseline = -2]
\draw[magenta, thick]  (-1/4,3/4) .. controls +(0,-1/2) and +(0,-1/2) .. (1/4,3/4);
\coordinate (A) at (-1/4,-1/4);
\coordinate (B) at (1/4,-1/4);
\coordinate (C) at (1/4,3/4);
\coordinate (D) at (-1/4,3/4);
\path (A) edge[bend left=90,thick, yellow] (B);
\path (A) edge[bend right=90,cyan, thick] (B);
\draw[thick, magenta] (A) -- +(0,-1/2);
\draw[thick, magenta] (B) -- +(0,-1/2);
\node[triality] at (A) {};
\node[triality] at (B) {};
\end{tikzpicture}}

\def\blackboxlda{\begin{tikzpicture}[baseline = - 10]
\node[rectangle, fill=orange!50, draw, very thick, minimum height=0.25cm, minimum width=0.25cm] (box) at (0,0) {};
\draw[solid, thick, cyan] (box.south west) -- +(south:1/4);
\draw[solid, thick, cyan] (box.south east) -- +(south:1/4);
\draw[solid, thick, magenta] (box.north west) -- +(north west:1/2);
\draw[solid, thick, magenta] (box.north east) -- +(north east:1/2);
\node[rectangle, fill=orange!50, draw, very thick, minimum height=0.25cm, minimum width=0.25cm] (box) at (0,-1/2) {};
\draw[solid, thick, magenta] (box.south west) -- +(south west:1/2);
\draw[solid, thick, magenta] (box.south east) -- +(south east:1/2);
\end{tikzpicture}}
\def\blackboxldab{\begin{tikzpicture}[baseline = - 17]
\node[rectangle, fill=orange!50, draw, very thick, minimum height=0.25cm, minimum width=0.25cm] (box) at (0,0) {};
\draw[solid, thick, cyan] (box.south west) -- +(south:1/4);
\draw[solid, thick, cyan] (box.south east) -- +(south:1/4);
\draw[solid, thick, magenta] (box.north west) -- +(north west:1/2);
\draw[solid, thick, magenta] (box.north east) -- +(north east:1/2);
\node[rectangle, fill=orange!50, draw, very thick, minimum height=0.25cm, minimum width=0.25cm] (box) at (0,-1/2) {};
\draw[solid, thick, yellow] (box.south west) -- +(south:1/4);
\draw[solid, thick, yellow] (box.south east) -- +(south:1/4);
\node[rectangle, fill=orange!50, draw, very thick, minimum height=0.25cm, minimum width=0.25cm] (box) at (0,-1) {};
\draw[solid, thick, magenta] (box.south west) -- +(south west:1/2);
\draw[solid, thick, magenta] (box.south east) -- +(south east:1/2);
\end{tikzpicture}}

\def\blackboxldaaa{\begin{tikzpicture}[baseline = - 24]
\node[rectangle, fill=orange!50, draw, very thick, minimum height=0.25cm, minimum width=0.25cm] (box) at (0,0) {};
\draw[solid, thick, cyan] (box.south west) -- +(south:1/4);
\draw[solid, thick, cyan] (box.south east) -- +(south:1/4);
\draw[solid, thick, magenta] (box.north west) -- +(north west:1/2);
\draw[solid, thick, magenta] (box.north east) -- +(north east:1/2);
\node[rectangle, fill=orange!50, draw, very thick, minimum height=0.25cm, minimum width=0.25cm] (box) at (0,-1/2) {};
\draw[solid, thick, magenta] (box.south west) -- +(south:1/2);
\draw[solid, thick, magenta] (box.south east) -- +(south:1/2);
\node[rectangle, fill=orange!50, draw, very thick, minimum height=0.25cm, minimum width=0.25cm] (box) at (0,-1) {};
\draw[solid, thick, cyan] (box.south west) -- +(south:1/4);
\draw[solid, thick, cyan] (box.south east) -- +(south:1/4);
\node[rectangle, fill=orange!50, draw, very thick, minimum height=0.25cm, minimum width=0.25cm] (box) at (0,-3/2) {};
\draw[solid, thick, magenta] (box.south west) -- +(south west:1/2);
\draw[solid, thick, magenta] (box.south east) -- +(south east:1/2);
\end{tikzpicture}}

\def\blackboxldaba{\begin{tikzpicture}[baseline = - 24]
\node[rectangle, fill=orange!50, draw, very thick, minimum height=0.25cm, minimum width=0.25cm] (box) at (0,0) {};
\draw[solid, thick, cyan] (box.south west) -- +(south:1/4);
\draw[solid, thick, cyan] (box.south east) -- +(south:1/4);
\draw[solid, thick, magenta] (box.north west) -- +(north west:1/2);
\draw[solid, thick, magenta] (box.north east) -- +(north east:1/2);
\node[rectangle, fill=orange!50, draw, very thick, minimum height=0.25cm, minimum width=0.25cm] (box) at (0,-1/2) {};
\draw[solid, thick, magenta] (box.south west) -- +(south:1/2);
\draw[solid, thick, magenta] (box.south east) -- +(south:1/2);
\node[rectangle, fill=orange!50, draw, very thick, minimum height=0.25cm, minimum width=0.25cm] (box) at (0,-1) {};
\draw[solid, thick, yellow] (box.south west) -- +(south:1/4);
\draw[solid, thick, yellow] (box.south east) -- +(south:1/4);
\node[rectangle, fill=orange!50, draw, very thick, minimum height=0.25cm, minimum width=0.25cm] (box) at (0,-3/2) {};
\draw[solid, thick, magenta] (box.south west) -- +(south west:1/2);
\draw[solid, thick, magenta] (box.south east) -- +(south east:1/2);
\end{tikzpicture}}

\def\blackboxldaa{\begin{tikzpicture}[baseline = - 10]
\node[rectangle, fill=orange!50, draw, very thick, minimum height=0.25cm, minimum width=0.25cm] (box) at (0,0) {};
\draw[solid, thick, cyan] (box.south west) -- +(south:1/4);
\draw[solid, thick, cyan] (box.south east) -- +(south:1/4);
\draw[solid, thick, magenta] (box.north west) -- +(north west:0.6);
\draw[solid, thick, magenta] (box.north east) -- +(north east:0.6);
\node[rectangle, fill=orange!50, draw, very thick, minimum height=0.25cm, minimum width=0.25cm] (box) at (0,-1/2) {};
\draw[solid, thick, yellow] (box.south west) -- +(south west:0.6);
\draw[solid, thick, yellow] (box.south east) -- +(south east:0.6);
\end{tikzpicture}}
\def\blackboxldb{\begin{tikzpicture}[baseline = - 10]
\node[rectangle, fill=orange!50, draw, very thick, minimum height=0.25cm, minimum width=0.25cm] (box) at (0,0) {};
\draw[solid, thick, yellow] (box.south west) -- +(south:1/4);
\draw[solid, thick, yellow] (box.south east) -- +(south:1/4);
\draw[solid, thick, magenta] (box.north west) -- +(north west:1/2);
\draw[solid, thick, magenta] (box.north east) -- +(north east:1/2);
\node[rectangle, fill=orange!50, draw, very thick, minimum height=0.25cm, minimum width=0.25cm] (box) at (0,-1/2) {};
\draw[solid, thick, magenta] (box.south west) -- +(south west:1/2);
\draw[solid, thick, magenta] (box.south east) -- +(south east:1/2);
\end{tikzpicture}}

\def\blackboxldc{\begin{tikzpicture}[baseline = - 2.2, scale = 0.7]
\draw[solid, thick, magenta] (1/3, 0.7)--(1/3, -0.7);
\draw[solid, thick, magenta] (-1/3, 0.7)--(-1/3, -0.7);
\end{tikzpicture}}
\def\blackboxldd{\begin{tikzpicture}[baseline = - 2.2, scale = 0.7]
\draw[solid, thick, magenta] (1/3, 0.7)--(-1/3, -0.7);
\draw[solid, thick, magenta] (-1/3, 0.7)--(1/3, -0.7);
\end{tikzpicture}}

\def\diagboxr{\begin{tikzpicture}[baseline]
\coordinate(A) at (-1/8,1/2);
\coordinate(B) at (-1/8,-1/2);
\coordinate(C) at (1/8,-1/2);
\coordinate(D) at (1/8,1/2);
\draw (A) -- (B);
\draw (C) -- (D);
\end{tikzpicture}}

\def\diagboxs{\begin{tikzpicture}[baseline]
\coordinate(A) at (-1/8,1/2);
\coordinate(B) at (-1/8,-1/2);
\coordinate(C) at (1/8,-1/2);
\coordinate(D) at (1/8,1/2);
\draw (A) -- (C);
\draw (B) -- (D);
\end{tikzpicture}}

\def\offdiaga{\begin{tikzpicture}[baseline]
\coordinate (triality1) at (0,0);
\coordinate (triality2) at (1/2,0);
\draw (triality1) -- +(south west:1/2);
\draw (triality2) -- +(south east:1/2);
\draw[densely dotted, thick] (triality1) -- +(north west:1/2);
\draw[densely dotted, thick] (triality2) -- +(north east:1/2);
\draw[decorate,decoration={snake}] (triality1) -- (triality2);
\node[triality] at (triality1) {};
\node[triality] at (triality2) {};
\end{tikzpicture}}

\def\offdiagb{\begin{tikzpicture}[baseline]
\draw[densely dotted, thick] (-1/2,1/2) .. controls +(0,-1/2) and +(0,-1/2) .. (1/2,1/2);
\draw (-1/2,-1/2) .. controls +(0,1/2) and +(0,1/2) .. (1/2,-1/2);
\end{tikzpicture}}

\def\offdiagc{\begin{tikzpicture}[baseline = - 1]
\coordinate (A) at (-1/4,-1/4);
\coordinate (B) at (1/4,-1/4);
\coordinate (C) at (1/4,1/4);
\coordinate (D) at (-1/4,1/4);
\node[triality] at (A) {};
\node[triality] at (B) {};
\node[triality] at (C) {};
\node[triality] at (D) {};
\draw[magenta, thick] (A) -- (B);
\draw[cyan, thick] (B) -- (C);
\draw[yellow, thick] (C) -- (D);
\draw[cyan, thick] (D) -- (A);
\draw[thick, yellow] (A) -- +(-1/2,-1/2);
\draw[thick, yellow] (B) -- +(1/2,-1/2);
\draw[magenta, thick] (C) -- +(1/2,1/2);
\draw[magenta, thick] (D) -- +(-1/2,1/2);
\node[triality] at (A) {};
\node[triality] at (B) {};
\node[triality] at (C) {};
\node[triality] at (D) {};
\end{tikzpicture}}
\def\offdiagf{\begin{tikzpicture}[baseline = - 2]
\coordinate (A) at (-1/4,-1/4);
\coordinate (B) at (1/4,-1/4);
\coordinate (C) at (1/4,1/4);
\coordinate (D) at (-1/4,1/4);
\node[triality] at (A) {};
\node[triality] at (B) {};
\node[triality] at (C) {};
\node[triality] at (D) {};
\draw[magenta, thick] (A) -- (B);
\draw[cyan, thick] (B) -- (C);
\draw[yellow, thick] (C) -- (A);
\draw[cyan, thick] (D) -- (A);
\path (D) edge[bend right = 40, yellow, thick] (-3/4,-3/4);
\draw[thick, yellow] (B) -- +(1/2,-1/2);
\draw[magenta, thick] (C) -- +(1/2,1/2);
\draw[magenta, thick] (D) -- +(-1/2,1/2);
\node[triality] at (A) {};
\node[triality] at (B) {};
\node[triality] at (C) {};
\node[triality] at (D) {};
\end{tikzpicture}}

\def\offdiagg{\begin{tikzpicture}[baseline = - 2]
\coordinate (A) at (-1/4,-1/4);
\coordinate (B) at (1/4,-1/4);
\coordinate (C) at (1/4,1/4);
\coordinate (D) at (-1/4,1/4);
\draw[cyan, thick] (B) -- (C);
\path (C) edge[yellow, thick, ,bend right=40] (-3/4,-3/4);
\path (B) edge[magenta, thick,bend left=40] (-3/4,3/4);
\draw[magenta, thick] (C) -- +(0,1/2);
\draw[yellow, thick] (B) -- +(0,-1/2);
\node[triality] at (B) {};
\node[triality] at (C) {};
\end{tikzpicture}}
\def\offdiagh{\begin{tikzpicture}[baseline = -2]
\coordinate (triality1) at (0,0);
\coordinate (triality2) at (1/2,0);
\draw[yellow, thick] (triality1) -- (-1/2, -3/4);
\draw[yellow, thick] (triality2) -- +(1/2, -3/4);
\draw[magenta, thick] (triality1) -- +(-1/2, 3/4);
\draw[magenta, thick] (triality2) -- +(1/2, 3/4);
\draw[cyan, thick] (triality1) -- (triality2);
\node[triality] at (triality1) {};
\node[triality] at (triality2) {};
\end{tikzpicture}}

\def\offdiagj{\begin{tikzpicture}[baseline = -2]
\coordinate (A) at (-1/4,-1/4);
\coordinate (B) at (1/4,-1/4);
\coordinate (C) at (1/4,1/4);
\coordinate (D) at (-1/4,1/4);
\draw[magenta, thick] (A) -- (B);
\draw[thick, cyan] (B) -- (C);
\draw[yellow, thick] (C) -- (D);
\draw[thick, cyan] (D) -- (A);
\draw[yellow, thick] (A) -- (-3/4,-3/4);
\draw[yellow, thick] (B) -- (3/4,-3/4);
\draw[magenta, thick] (C) -- (-3/4,3/4);
\draw[magenta, thick] (D) -- (3/4,3/4);
\node[triality] at (A) {};
\node[triality] at (B) {};
\node[triality] at (C) {};
\node[triality] at (D) {};
\end{tikzpicture}}

\def\offdiagk{\begin{tikzpicture}[baseline = - 2pt]
\node[rectangle, fill=orange!50, draw, very thick, minimum height=0.25cm, minimum width=0.25cm] (box) at (0,0) {};
\draw[solid, thick, yellow] (box.south west) -- +(south west:3/4);
\draw[solid, thick, yellow] (box.south east) -- +(south east:3/4);
\draw[solid, thick, magenta] (box.north west) -- +(north west:3/4);
\draw[solid, thick, magenta] (box.north east) -- +(north east:3/4);
\end{tikzpicture}}
\def\offdiagl{\begin{tikzpicture}[baseline = - 2pt]
\draw[magenta, thick] (-2/3,2/3) .. controls +(0,-3/4) and +(0,-3/4) .. (2/3,2/3);
\draw[yellow, thick] (-2/3,-2/3) .. controls +(0,3/4) and +(0,3/4) .. (2/3,-2/3);
\end{tikzpicture}}

\def\tripresa{\begin{tikzpicture}[baseline = - 2pt]
  \path{ (1, 1) edge[thick, bend left=40, cyan]  (2, -0.8)};
    \draw[yellow, thick] (1,1) .. controls +(0,1.3) and +(0,1.3) ..  (-2, -0.8) ;
 \node[triality] at (0, 0) {};
\node[rectangle, fill=orange!50, draw, very thick, minimum height=0.25cm, minimum width=0.25cm] (box) at (0,0) {};
 \draw[magenta, thick] (1,1) .. controls +(0,-1/2) and +(0,1/2) ..  (box.north east) ;
  \draw[magenta, thick] (-1.2,-0.8) .. controls +(0,1/2) and +(0,1/2) ..  (box.north west) ;
    \draw[yellow, thick] (-0.6,-0.8) .. controls +(0,1/2) and +(0,-1/2) ..  (box.south west) ;
       \draw[yellow, thick] (0.6,-0.8) .. controls +(0,1/2) and +(0,-1/2) ..  (box.south east) ;
        \node[triality] at (1, 1) {};
\end{tikzpicture}}

\def\tripresb{\begin{tikzpicture}[baseline = - 2pt]
  \path{ (1, 1) edge[thick, bend left=40, magenta]  (2, -0.8)};
    \draw[yellow, thick] (1,1) .. controls +(0,1.3) and +(0,1.3) ..  (-2, -0.8) ;
 \node[triality] at (0, 0) {};
\node[rectangle, fill=orange!50, draw, very thick, minimum height=0.25cm, minimum width=0.25cm] (box) at (0,0) {};
 \draw[cyan, thick] (1,1) .. controls +(0,-1/2) and +(0,1/2) ..  (box.north east) ;
  \draw[cyan, thick] (-1.2,-0.8) .. controls +(0,1/2) and +(0,1/2) ..   (box.north west) ;
    \draw[yellow, thick] (-0.6,-0.8) .. controls +(0,1/2) and +(0,-1/2) ..    (box.south west) ;
       \draw[yellow, thick] (0.6,-0.8) .. controls +(0,1/2) and +(0,-1/2) ..  (box.south east) ;
        \node[triality] at (1, 1) {};
\end{tikzpicture}}

\def\tripresc{\begin{tikzpicture}[baseline = - 2pt]
  \path{ (1, 1) edge[thick, bend left=40, cyan]  (2, -0.8)};
    \draw[magenta, thick] (1,1) .. controls +(0,1.3) and +(0,1.3) ..  (-2, -0.8);
 \draw[yellow, thick] (1,1) .. controls +(0,-1/2) and +(0,1/2) ..  (0.125, 0);
  \draw[yellow, thick] (-1.2,-0.8) .. controls +(0,1/2) and +(0,1/2) ..  (-0.125, 0) ;
    \draw[yellow, thick] (-0.6,-0.8) .. controls +(0,1/2) and +(0,-1/2) ..  (-0.125, 0) ;
       \draw[yellow, thick] (0.6,-0.8) .. controls +(0,1/2) and +(0,-1/2) ..  (0.125, 0) ;
        \node[triality] at (1, 1) {};
          \path{ (-0.25, 0) edge[thick, red]  (0.25, 0)};
\end{tikzpicture}}

\def\tripresd{\begin{tikzpicture}[baseline = - 2pt]
  \path{ (1, 1) edge[thick, bend left=40, cyan]  (2, -0.8)};
    \draw[yellow, thick] (1,1) .. controls +(0,1.3) and +(0,1.3) ..  (-2, -0.8);
 \draw[magenta, thick] (1,1) .. controls +(0,-1/2) and +(0,1/2) ..  (0.25, 0);
  \draw[magenta, thick] (-1.2,-0.8) .. controls +(0,1/2) and +(0,1/2) ..  (-0.25, 0) ;
    \draw[yellow, thick] (-0.6,-0.8) .. controls +(0,1/2) and +(0,-1/2) ..  (-0.25, 0) ;
       \draw[yellow, thick] (0.6,-0.8) .. controls +(0,1/2) and +(0,-1/2) ..  (0.25, 0) ;
        \node[triality] at (1, 1) {};
          \path{ (-0.25, 0) edge[thick, cyan]  (0.25, 0)};        
           \node[triality] at (-0.25, 0) {};
            \node[triality] at (0.25, 0) {};
          \path{ (-0.5, -0.3) edge[thick, red]  (0.5, -0.3)};
\end{tikzpicture}}

\def\triprese{\begin{tikzpicture}[baseline = - 2pt]
  \path{ (0.25, 0) edge[thick, bend left=40, cyan]  (2, -0.8)};
    \draw[yellow, thick] (1,1) .. controls +(0,1.3) and +(0,1.3) ..  (-2, -0.8);
 \draw[magenta, thick] (1,1) .. controls +(0,-1/2) and +(0,1/2) ..  (0.25, 0);
  \draw[magenta, thick] (-1.2,-0.8) .. controls +(0,1/2) and +(0,1/2) ..  (-0.25, 0) ;
    \draw[yellow, thick] (-0.6,-0.8) .. controls +(0,1/2) and +(0,-1/2) ..  (-0.25, 0) ;
       \draw[yellow, thick] (0.6,-0.8) .. controls +(0,1/2) and +(0,-1/2) ..  (0.25, 0) ;
          \path{ (-0.25, 0) edge[thick, cyan]  (1, 1)};        
           \node[triality] at (-0.25, 0) {};
            \node[triality] at (0.25, 0) {};
                    \node[triality] at (1, 1) {};
          \path{ (-0.5, -0.3) edge[thick, red]  (0.5, -0.3)};
\end{tikzpicture}}

\def\tripresf{\begin{tikzpicture}[baseline = - 2pt]
  \path{ (-0.25, 0) edge[thick, bend left=60, cyan]  (2, -0.8)};
    \draw[yellow, thick] (1,1) .. controls +(0,1.3) and +(0,1.3) ..  (-2, -0.8);
        \draw[yellow, thick] (1,1) .. controls +(0,-1/2) and +(0,1/2) ..  (0.25, 0);
  \draw[magenta, thick] (-1.2,-0.8) .. controls +(0,1/2) and +(0,1/2) ..  (-0.25, 0) ;
    \draw[yellow, thick] (-0.6,-0.8) .. controls +(0,1/2) and +(0,-1/2) ..  (-0.25, 0) ;
       \draw[yellow, thick] (0.6,-0.8) .. controls +(0,1/2) and +(0,-1/2) ..  (0.25, 0) ;
           \node[triality] at (-0.25, 0) {};
          \path{ (-0.5, -0.3) edge[thick, red]  (0.5, -0.3)};
\end{tikzpicture}}

\def\tripresg{\begin{tikzpicture}[baseline = - 2pt]
  \path{ (0.25, 0) edge[thick, bend left=40, cyan]  (2, -0.8)};
    \draw[yellow, thick] (1,1) .. controls +(0,1.3) and +(0,1.3) ..  (-2, -0.8);
 \draw[magenta, thick] (1,1) .. controls +(-1,0) and +(0,1/2) ..  (-1.2, -0.8);
    \draw[yellow, thick] (-0.6,-0.8) .. controls +(0,1/2) and +(0,-1/2) ..  (-0.25, 0) ;
       \draw[yellow, thick] (0.6,-0.8) .. controls +(0,1/2) and +(0,-1/2) ..  (0.25, 0) ;
          \path{ (-0.25, 0) edge[thick, cyan]  (1, 1)};     
               \path{ (-0.25, 0) edge[thick, magenta]  (0.25, 0)};       
           \node[triality] at (-0.25, 0) {};
            \node[triality] at (0.25, 0) {};
                    \node[triality] at (1, 1) {};
          \path{ (-0.5, -0.3) edge[thick, red]  (0.5, -0.3)};
\end{tikzpicture}}

\def\tripresh{\begin{tikzpicture}[baseline = - 2pt]
  \path{ (-0.25, 0) edge[thick, bend left=40, cyan]  (2, -0.8)};
    \draw[yellow, thick] (1,1) .. controls +(0,1.3) and +(0,1.3) ..  (-2, -0.8);
 \draw[yellow, thick] (1,1) .. controls +(0,-1/2) and +(0,1/2) ..  (0.25, 0);
  \draw[magenta, thick] (-1.2,-0.8) .. controls +(0,1/2) and +(0,1/2) ..  (-0.25, 0) ;
    \draw[yellow, thick] (-0.6,-0.8) .. controls +(0,1/2) and +(0,-1/2) ..  (-0.25, 0) ;
       \draw[yellow, thick] (0.6,-0.8) .. controls +(0,1/2) and +(0,-1/2) ..  (0.25, 0) ;
           \node[triality] at (-0.25, 0) {};
          \path{ (-0.5, -0.3) edge[thick, red]  (0.5, -0.3)};
\end{tikzpicture}}

\def\tricona{\begin{tikzpicture}[baseline = - 2pt]
  \path{ (1, 1) edge[thick, bend left=40, yellow]  (1.2, -0.4)};
    \draw[magenta, thick] (1,1) .. controls +(0,1.3) and +(0,1.3) ..  (-1.6, -0.4) ;
 \node[triality] at (0, 0) {};
\node[rectangle, fill=cyan, draw, very thick, minimum height=0.30cm, minimum width=0.30cm] (box) at (0,0) {};
 \draw[cyan, thick] (1,1) .. controls +(0,-1/2) and +(0,1/2) ..  (box.north east) ;
  \draw[cyan, thick] (-1.2,-0.4) .. controls +(0,1/2) and +(0,1/2) ..  (box.north west) ;
        \node[triality] at (1, 1) {};
\end{tikzpicture}}

\def\triconas{\begin{tikzpicture}[baseline = - 2pt]
  \path{ (1, 1) edge[thick, bend left=40, yellow]  (1.2, -0.4)};
    \draw[magenta, thick] (1,1) .. controls +(0,1.3) and +(0,1.3) ..  (-1.6, -0.4) ;
 \node[triality] at (0, 0) {};
\node[blobB, thick, draw] (blob1) at (0,0) {};
 \draw[cyan, thick] (1,1) .. controls +(0,-1/2) and +(0,1/2) ..  (blob1.north east) ;
  \draw[cyan, thick] (-1.2,-0.4) .. controls +(0,1/2) and +(0,1/2) ..  (blob1.north west) ;
        \node[triality] at (1, 1) {};
\end{tikzpicture}}

\def\triconb{\begin{tikzpicture}[baseline = - 2pt]
  \path{ (1, 1) edge[thick, bend left=40, magenta]  (1.2, -0.4)};
    \draw[cyan, thick] (1,1) .. controls +(0,1.3) and +(0,1.3) ..  (-1.6, -0.4) ;
 \node[triality] at (0, 0) {};
\node[rectangle, fill=yellow, draw, very thick, minimum height=0.30cm, minimum width=0.30cm] (box) at (0,0) {};
 \draw[yellow, thick] (1,1) .. controls +(0,-1/2) and +(0,1/2) ..  (box.north east) ;
  \draw[yellow, thick] (-1.2,-0.4) .. controls +(0,1/2) and +(0,1/2) ..  (box.north west) ;
        \node[triality] at (1, 1) {};
\end{tikzpicture}}
\def\triconbs{\begin{tikzpicture}[baseline = - 2pt]
  \path{ (1, 1) edge[thick, bend left=40, magenta]  (1.2, -0.4)};
    \draw[cyan, thick] (1,1) .. controls +(0,1.3) and +(0,1.3) ..  (-1.6, -0.4) ;
 \node[triality] at (0, 0) {};
\node[blobC, thick, draw] (blob2) at (0,0) {};
 \draw[yellow, thick] (1,1) .. controls +(0,-1/2) and +(0,1/2) ..  (blob2.north east) ;
  \draw[yellow, thick] (-1.2,-0.4) .. controls +(0,1/2) and +(0,1/2) ..  (blob2.north west) ;
        \node[triality] at (1, 1) {};
\end{tikzpicture}}

\def\triconc{\begin{tikzpicture}[baseline = - 2pt]
  \path{ (1, 1) edge[thick, bend left=40, cyan]  (1.2, -0.4)};
    \draw[yellow, thick] (1,1) .. controls +(0,1.3) and +(0,1.3) ..  (-1.6, -0.4) ;
 \node[triality] at (0, 0) {};
\node[rectangle, fill=magenta, draw, very thick, minimum height=0.30cm, minimum width=0.30cm] (box) at (0,0) {};
 \draw[magenta, thick] (1,1) .. controls +(0,-1/2) and +(0,1/2) ..  (box.north east) ;
  \draw[magenta, thick] (-1.2,-0.4) .. controls +(0,1/2) and +(0,1/2) ..  (box.north west) ;
        \node[triality] at (1, 1) {};
\end{tikzpicture}}

\def\triconcs{\begin{tikzpicture}[baseline = - 2pt]
  \path{ (1, 1) edge[thick, bend left=40, cyan]  (1.2, -0.4)};
    \draw[yellow, thick] (1,1) .. controls +(0,1.3) and +(0,1.3) ..  (-1.6, -0.4) ;
 \node[triality] at (0, 0) {};
\node[blobA, thick, draw] (blob3) at (0,0) {};
 \draw[magenta, thick] (1,1) .. controls +(0,-1/2) and +(0,1/2) ..  (blob3.north east) ;
  \draw[magenta, thick] (-1.2,-0.4) .. controls +(0,1/2) and +(0,1/2) ..  (blob3.north west) ;
        \node[triality] at (1, 1) {};
\end{tikzpicture}}

\def\triconca{\begin{tikzpicture}[baseline = - 3pt]
\node[rectangle, fill=magenta, draw, very thick, minimum height=0.30cm, minimum width=0.30cm] (box) at (0,0) {};
 \draw[magenta, thick] (0.5,0.7) .. controls +(-1/8,-1/8) and +(1/8,1/4) ..  (box.north east) ;
 \draw[magenta, thick] (-0.5,0.7) .. controls +(1/8,-1/8) and +(-1/8,1/4) ..  (box.north west) ;
\end{tikzpicture}}

\def\triconcb{\begin{tikzpicture}[baseline = - 3pt]
\node[rectangle, fill=magenta, draw, very thick, minimum height=0.30cm, minimum width=0.30cm] (box) at (0,0) {};
 \draw[magenta, thick] (0.5,0.7) .. controls +(-1/8,-1/8) and +(1/8,1/4) ..  (box.north west) ;
 \draw[magenta, thick] (-0.5,0.7) .. controls +(1/8,-1/8) and +(-1/8,1/4) ..  (box.north east) ;
\end{tikzpicture}}

\def\triconcc{\begin{tikzpicture}[baseline = - 3pt]
\node[rectangle, fill=cyan, draw, very thick, minimum height=0.30cm, minimum width=0.30cm] (box) at (0,0) {};
 \draw[cyan, thick] (0.5,0.7) .. controls +(-1/8,-1/8) and +(1/8,1/4) ..  (box.north east) ;
 \draw[cyan, thick] (-0.5,0.7) .. controls +(1/8,-1/8) and +(-1/8,1/4) ..  (box.north west) ;
\end{tikzpicture}}

\def\triconcd{\begin{tikzpicture}[baseline = - 3pt]
\node[rectangle, fill=cyan, draw, very thick, minimum height=0.30cm, minimum width=0.30cm] (box) at (0,0) {};
 \draw[cyan, thick] (0.5,0.7) .. controls +(-1/8,-1/8) and +(1/8,1/4) ..  (box.north west) ;
 \draw[cyan, thick] (-0.5,0.7) .. controls +(1/8,-1/8) and +(-1/8,1/4) ..  (box.north east) ;
\end{tikzpicture}}

\def\triconce{\begin{tikzpicture}[baseline = - 3pt]
\node[rectangle, fill=yellow, draw, very thick, minimum height=0.30cm, minimum width=0.30cm] (box) at (0,0) {};
 \draw[yellow, thick] (0.5,0.7) .. controls +(-1/8,-1/8) and +(1/8,1/4) ..  (box.north east) ;
 \draw[yellow, thick] (-0.5,0.7) .. controls +(1/8,-1/8) and +(-1/8,1/4) ..  (box.north west) ;
\end{tikzpicture}}

\def\triconcf{\begin{tikzpicture}[baseline = - 3pt]
\node[rectangle, fill=yellow, draw, very thick, minimum height=0.30cm, minimum width=0.30cm] (box) at (0,0) {};
 \draw[yellow, thick] (0.5,0.7) .. controls +(-1/8,-1/8) and +(1/8,1/4) ..  (box.north west) ;
 \draw[yellow, thick] (-0.5,0.7) .. controls +(1/8,-1/8) and +(-1/8,1/4) ..  (box.north east) ;
\end{tikzpicture}}

\def\triconcebox{\begin{tikzpicture}[baseline = - 3pt]
\node[rectangle, fill=orange, draw, very thick, minimum height=0.3, minimum width=0.3] (box) at (0,3/4) {};
 \draw[magenta, thick] (0.5,1.45) .. controls +(-1/8,-1/8) and +(1/8,1/4) ..  (box.north east) ;
 \draw[magenta, thick] (-0.5,1.45) .. controls +(1/8,-1/8) and +(-1/8,1/4) ..  (box.north west) ;
\node[rectangle, fill=yellow, draw, very thick, minimum height=0.30cm, minimum width=0.30cm] (box) at (0,0) {};
  \path{(-0.125, 0.625) edge[thick, bend right = 30, yellow]  (box.north west) };
    \path{(0.125, 0.625) edge[thick, bend left = 30, yellow]  (box.north east) };
\end{tikzpicture}}

\def\triconceboxa{\begin{tikzpicture}[baseline = - 3pt]
\node[rectangle, fill=orange, draw, very thick, minimum height=0.3, minimum width=0.3] (box) at (0,3/4) {};
 \draw[cyan, thick] (0.5,1.45) .. controls +(-1/8,-1/8) and +(1/8,1/4) ..  (box.north east) ;
 \draw[cyan, thick] (-0.5,1.45) .. controls +(1/8,-1/8) and +(-1/8,1/4) ..  (box.north west) ;
\node[rectangle, fill=yellow, draw, very thick, minimum height=0.30cm, minimum width=0.30cm] (box) at (0,0) {};
  \path{(-0.125, 0.625) edge[thick, bend right = 30, yellow]  (box.north west) };
    \path{(0.125, 0.625) edge[thick, bend left = 30, yellow]  (box.north east) };
\end{tikzpicture}}

\def\tricontra{\begin{tikzpicture}[baseline = - 2pt]
  \path{ (1, 1) edge[thick, bend left=40, yellow]  (0, 1)};
    \path{ (1, 1) edge[thick, bend right=40, magenta]  (0, 1)};
    \draw[cyan, thick] (0,1) .. controls +(0,1.3) and +(0,1.3) ..  (-1.6, -0.4) ;
 \node[triality] at (0, 0) {};
\node[rectangle, fill=cyan, draw, very thick, minimum height=0.30cm, minimum width=0.30cm] (box) at (0,0) {};
 \draw[cyan, thick] (1,1) .. controls +(0,-1/2) and +(0,1/2) ..  (box.north east) ;
  \draw[cyan, thick] (-1.2,-0.4) .. controls +(0,1/2) and +(0,1/2) ..  (box.north west) ;
        \node[triality] at (1, 1) {};
         \node[triality] at (0, 1) {};
\end{tikzpicture}}

\def\tricontrc{\begin{tikzpicture}[baseline = - 2pt]
  \path{ (1, 1) edge[thick, bend left=40, cyan]  (1.2, -0.4)};
    \path{ (1, 1) edge[thick, bend right = 40, yellow]  (0, 1)};
    \draw[cyan, thick] (0,1) .. controls +(0,1.3) and +(0,1.3) ..  (-1.6, -0.4) ;
 \node[triality] at (0, 0) {};
\node[rectangle, fill=magenta, draw, very thick, minimum height=0.30cm, minimum width=0.30cm] (box) at (0,0) {};
 \draw[magenta, thick] (1,1) .. controls +(0,-1/2) and +(0,1/2) ..  (box.north east) ;
  \draw[magenta, thick] (0,1) .. controls +(-1/8,-1/4) and +(-1/8,1/4) ..  (box.north west) ;
        \node[triality] at (1, 1) {};
            \node[triality] at (0, 1) {};
\end{tikzpicture}}

\def\triconcage{\begin{tikzpicture}[baseline = - 3pt]
\node[rectangle, fill=magenta, draw, very thick, minimum height=0.30cm, minimum width=0.30cm] (box) at (0,0) {};
 \draw[magenta, thick] (0.4,0.6) .. controls +(-1/8,-1/8) and +(1/8,1/4) ..  (box.north east) ;
 \draw[magenta, thick] (-0.4,0.6) .. controls +(1/8,-1/8) and +(-1/8,1/4) ..  (box.north west) ;
   \path{ (-0.4, 0.6) edge[thick, yellow]  (0.4, 0.6)};
    \path{ (0.4, 0.6) edge[thick, bend left = 20, cyan]  (0.7, 1)};
        \path{ (-0.4, 0.6) edge[thick, bend right = 20, cyan]  (-0.7, 1)};
   \node[triality] at (0.4, 0.6) {};
            \node[triality] at (-0.4, 0.6) {};
\end{tikzpicture}}

\def\triconage{\begin{tikzpicture}[baseline = - 2pt]
  \path{ (1, 1) edge[thick, bend left=40, yellow]  (1.2, -0.4)};
    \draw[magenta, thick] (1,1) .. controls +(0,1.3) and +(0,1.3) ..  (-1.6, -0.4) ;
 \draw[cyan, thick] (1,1) .. controls +(0,-1/2) and +(0,1/2) ..  (0.4, 0.3) ;
  \draw[cyan, thick] (-1.2,-0.4) .. controls +(0,1/2) and +(0,1/2) ..  (-0.4, 0.3) ;
     \node[triality] at (1, 1) {};
 \node[rectangle, fill=magenta, draw, very thick, minimum height=0.30cm, minimum width=0.30cm] (box) at (0,-0.3) {};
 \draw[magenta, thick] (0.4,0.3) .. controls +(-1/8,-1/8) and +(1/8,1/4) ..  (box.north east) ;
 \draw[magenta, thick] (-0.4,0.3) .. controls +(1/8,-1/8) and +(-1/8,1/4) ..  (box.north west) ;
   \path{ (-0.4, 0.3) edge[thick, yellow]  (0.4, 0.3)};
   \node[triality] at (0.4, 0.3) {};
            \node[triality] at (-0.4, 0.3) {};       
        
\end{tikzpicture}}

\def\triconagea{\begin{tikzpicture}[baseline = - 2pt]
  \path{ (0.4, 0.3) edge[thick, bend left=40, yellow]  (1.2, -0.4)};
    \draw[magenta, thick] (1,1) .. controls +(0,1.3) and +(0,1.3) ..  (-1.6, -0.4) ;
 \draw[cyan, thick] (1,1) .. controls +(0,-1/2) and +(0,1/2) ..  (0.4, 0.3) ;
  \draw[cyan, thick] (-1.2,-0.4) .. controls +(0,1/2) and +(0,1/2) ..  (-0.4, 0.3) ;
     \node[triality] at (1, 1) {};
 \node[rectangle, fill=magenta, draw, very thick, minimum height=0.30cm, minimum width=0.30cm] (box) at (0,-0.3) {};
 \draw[magenta, thick] (0.4,0.3) .. controls +(-1/8,-1/8) and +(1/8,1/4) ..  (box.north east) ;
 \draw[magenta, thick] (-0.4,0.3) .. controls +(1/8,-1/8) and +(-1/8,1/4) ..  (box.north west) ;
   \path{ (-0.4, 0.3) edge[thick, yellow]  (1, 1)};
   \node[triality] at (0.4, 0.3) {};
            \node[triality] at (-0.4, 0.3) {};       
        
\end{tikzpicture}}

\def\triconagec{\begin{tikzpicture}[baseline = - 2pt]
 \draw[magenta, thick] (1,1) .. controls +(0,1.3) and +(0,1.3) ..  (-1.6, -0.4) ; 
 \draw[magenta, thick] (0.2,1) .. controls +(0,1.3) and +(0,1.3) ..  (-2, -0.4) ;
 \draw[cyan, thick] (1,1) .. controls +(0,-1/2) and +(0,1/2) ..  (0.4, 0.3) ;
  \draw[cyan, thick] (0.2,1) .. controls +(0,-1/2) and +(0,1/2) ..  (-0.4, 0.3) ;
 \node[rectangle, fill=magenta, draw, very thick, minimum height=0.30cm, minimum width=0.30cm] (box) at (0,-0.3) {};
 \draw[magenta, thick] (0.4,0.3) .. controls +(-1/8,-1/8) and +(1/8,1/4) ..  (box.north east) ;
 \draw[magenta, thick] (-0.4,0.3) .. controls +(1/8,-1/8) and +(-1/8,1/4) ..  (box.north west) ;
   \path{ (-0.4, 0.3) edge[thick, yellow]  (0.4, 0.3)};
      \path{ (0.2, 1) edge[thick, yellow]  (1, 1)};
      \node[triality] at (1, 1) {};
   \node[triality] at (0.4, 0.3) {};
            \node[triality] at (-0.4, 0.3) {};       
             \node[triality] at (0.2, 1) {};
\end{tikzpicture}}

\def\triconccc{\begin{tikzpicture}[baseline = - 2pt]
  \path{ (1, 1) edge[thick, bend left = 30, cyan]  (0.2, 1)};
    \path{ (1, 1) edge[thick, bend right = 30, yellow]  (0.2, 1)};
    \draw[magenta, thick] (0.2,1) .. controls +(0,1.3) and +(0,1.3) ..  (-1.6, -0.4) ;
 \node[triality] at (0, 0) {};
\node[rectangle, fill=magenta, draw, very thick, minimum height=0.30cm, minimum width=0.30cm] (box) at (0,0) {};
 \draw[magenta, thick] (1,1) .. controls +(0,-1/2) and +(0,1/2) ..  (box.north east) ;
  \draw[magenta, thick] (-1.2,-0.4) .. controls +(0,1/2) and +(0,1/2) ..  (box.north west) ;
        \node[triality] at (1, 1) {};
         \node[triality] at (0.2, 1) {};
\end{tikzpicture}}

\def\triconageccc{\begin{tikzpicture}[baseline = - 2pt]
 \node[rectangle, fill=magenta, draw, very thick, minimum height=0.30cm, minimum width=0.30cm] (box) at (0,-0.3) {};
 \draw[magenta, thick] (box.north west) .. controls +(0,1) and +(0,1.1) ..  (-1.6, -0.4) ; 
 \draw[magenta, thick] (box.north east) .. controls +(0,1.3) and +(0,1.3) ..  (-2, -0.4) ;
\end{tikzpicture}}

\def\triconagecc{\begin{tikzpicture}[baseline = - 2pt]
 \node[rectangle, fill=magenta, draw, very thick, minimum height=0.30cm, minimum width=0.30cm] (box) at (0,-0.3) {};
 \path{ (box.north west) edge[thick, bend left = 30, magenta]  (-0.1, 0.4)};
  \path{ (box.north east) edge[thick, bend right = 30, magenta]  (0.12, 0.4)};
\node[rectangle, fill=orange!50, draw, very thick, minimum height=0.25cm, minimum width=0.25cm] (box) at (0,1) {};
\draw[solid, thick, cyan] (box.south west) -- +(south:1/4);
\draw[solid, thick, cyan] (box.south east) -- +(south:1/4);
 \draw[magenta, thick] (box.north east) .. controls +(0,1.3) and +(0,1.3) ..  (-1.6, -0.4) ; 
 \draw[magenta, thick] (box.north west) .. controls +(0,1.3) and +(0,1.3) ..  (-2, -0.4) ;
\node[rectangle, fill=orange!50, draw, very thick, minimum height=0.25cm, minimum width=0.25cm] (box) at (0,1/2) {};
\end{tikzpicture}}

\def\offdiagkepsa{\begin{tikzpicture}[baseline = - 2pt]
\node[rectangle, fill=orange!50, draw, very thick, minimum height=0.25cm, minimum width=0.25cm] (box) at (0,0) {};
 \path{ (box.north west) edge[thick, bend left = 30, magenta]  (-0.25, 0.45)};
  \path{ (box.north east) edge[thick, bend right = 30, magenta]  (0.25, 0.45)};
   \path{ (box.south west) edge[thick, bend right = 30, yellow]  (-0.25, -0.45)};
  \path{ (box.south east) edge[thick, bend left = 30, yellow]  (0.25, -0.45)};
   \path{ (-0.25, 0.45) edge[thick, magenta]  (0.25, 0.45)};
      \path{ (-0.25, -0.45) edge[thick, yellow]  (0.25, -0.45)};
\end{tikzpicture}}

\def\offdiagkepsb{\begin{tikzpicture}[baseline = - 2pt]
\node[rectangle, fill=orange!50, draw, very thick, minimum height=0.25cm, minimum width=0.25cm] (box) at (0,0) {};
 \path{ (box.north west) edge[thick, bend left = 30, yellow]  (-0.25, 0.45)};
  \path{ (box.north east) edge[thick, bend right = 30, yellow]  (0.25, 0.45)};
    \path{ (-0.25, 0.45) edge[thick, yellow]  (0.25, 0.45)};
   \path{ (box.south west) edge[thick, bend right = 30, magenta]  (-0.25, -0.45)};
  \path{ (box.south east) edge[thick, bend left = 30, magenta]  (0.25, -0.45)};
      \path{ (-0.25, -0.45) edge[thick, magenta]  (0.25, -0.45)};
\end{tikzpicture}}

\def\offdiagkepsc{\begin{tikzpicture}[baseline = - 2pt]
\node[rectangle, fill=orange!50, draw, very thick, minimum height=0.25cm, minimum width=0.25cm] (box) at (0,0) {};
 \path{ (box.north west) edge[thick, bend left = 30, cyan]  (-0.25, 0.45)};
  \path{ (box.north east) edge[thick, bend right = 30, cyan]  (0.25, 0.45)};
     \path{ (-0.25, 0.45) edge[thick, cyan]  (0.25, 0.45)};
   \path{ (box.south west) edge[thick, bend right = 30, magenta]  (-0.25, -0.45)};
  \path{ (box.south east) edge[thick, bend left = 30, magenta]  (0.25, -0.45)};
      \path{ (-0.25, -0.45) edge[thick, magenta]  (0.25, -0.45)};
\end{tikzpicture}}

\def\offdiagkepsd{\begin{tikzpicture}[baseline = - 2pt]
\node[rectangle, fill=orange!50, draw, very thick, minimum height=0.25cm, minimum width=0.25cm] (box) at (0,0) {};
 \path{ (box.north west) edge[thick, bend left = 30, magenta]  (-0.25, 0.45)};
  \path{ (box.north east) edge[thick, bend right = 30, magenta]  (0.25, 0.45)};
    \path{ (-0.25, 0.45) edge[thick, magenta]  (0.25, 0.45)};
   \path{ (box.south west) edge[thick, bend right = 30, cyan]  (-0.25, -0.45)};
  \path{ (box.south east) edge[thick, bend left = 30, cyan]  (0.25, -0.45)};
      \path{ (-0.25, -0.45) edge[thick, cyan]  (0.25, -0.45)};
\end{tikzpicture}}

\def\offdiagkepse{\begin{tikzpicture}[baseline = - 2pt]
\node[rectangle, fill=orange!50, draw, very thick, minimum height=0.25cm, minimum width=0.25cm] (box) at (0,0) {};
 \path{ (box.north west) edge[thick, bend left = 30, yellow]  (-0.25, 0.45)};
  \path{ (box.north east) edge[thick, bend right = 30, yellow]  (0.25, 0.45)};
     \path{ (-0.25, 0.45) edge[thick, yellow]  (0.25, 0.45)};
   \path{ (box.south west) edge[thick, bend right = 30, cyan]  (-0.25, -0.45)};
  \path{ (box.south east) edge[thick, bend left = 30, cyan]  (0.25, -0.45)};
      \path{ (-0.25, -0.45) edge[thick, cyan]  (0.25, -0.45)};
\end{tikzpicture}}

\def\offdiagkepsf{\begin{tikzpicture}[baseline = - 2pt]
\node[rectangle, fill=orange!50, draw, very thick, minimum height=0.25cm, minimum width=0.25cm] (box) at (0,0) {};
 \path{ (box.north west) edge[thick, bend left = 30, cyan]  (-0.25, 0.45)};
  \path{ (box.north east) edge[thick, bend right = 30, cyan]  (0.25, 0.45)};
    \path{ (-0.25, 0.45) edge[thick, cyan]  (0.25, 0.45)};
   \path{ (box.south west) edge[thick, bend right = 30, yellow]  (-0.25, -0.45)};
  \path{ (box.south east) edge[thick, bend left = 30, yellow]  (0.25, -0.45)};
      \path{ (-0.25, -0.45) edge[thick, yellow]  (0.25, -0.45)};
\end{tikzpicture}}

\def\offdiagl{\begin{tikzpicture}[baseline = - 2pt]
\draw[magenta, thick] (-2/3,2/3) .. controls +(0,-3/4) and +(0,-3/4) .. (2/3,2/3);
\draw[yellow, thick] (-2/3,-2/3) .. controls +(0,3/4) and +(0,3/4) .. (2/3,-2/3);
\node[triality] at (0, 0) {};
\end{tikzpicture}}

\def\offdiagd{\begin{tikzpicture}[baseline = - 1]
\draw[magenta, thick] (-1/2,0.72) .. controls +(0.2,-0.6) and +(-0.2,-0.6) .. (1/2,0.72);
\draw[yellow, thick] (-1/2,-0.72) .. controls +(0.2, 0.6) and +(-0.2,0.6) .. (1/2,-0.72);
\end{tikzpicture}}

\def\offdiage{\begin{tikzpicture}[baseline]
\node[rectangle, fill=gray!50, draw, very thick, minimum height=0.25cm, minimum width=0.5cm] (box1) at (0,-1/4) {};
\node[rectangle, fill=gray!50, draw, very thick, minimum height=0.25cm, minimum width=0.5cm] (box2) at (0,1/4) {};
\draw[decorate,decoration={snake}] (box1.south west) -- +(south:1/2);
\draw[decorate,decoration={snake}] (box1.south east) -- +(south:1/2);
\draw (box1.north west) -- +(north:1/2);
\draw (box1.north east) -- +(north:1/2);
\draw (box2.south west) -- +(south:1/2);
\draw (box2.south east) -- +(south:1/2);
\draw[densely dotted, thick] (box2.north west) -- +(north:1/2);
\draw[densely dotted, thick] (box2.north east) -- +(north:1/2);
\end{tikzpicture}}

\def\offdiagffail{\begin{tikzpicture}[baseline]
\coordinate (A) at (-1/4,-1/4);
\coordinate (B) at (1/4,-1/4);
\coordinate (C) at (1/4,1/4);
\coordinate (D) at (-1/4,1/4);
\node[triality] at (A) {};
\node[triality] at (B) {};
\node[triality] at (C) {};
\node[triality] at (D) {};
\draw[densely dotted, thick] (A) -- (B);
\draw (B) -- (C);
\draw[decorate,decoration={snake}] (C) -- (A);
\draw (D) -- (A);
\path (D) edge[decorate,decoration={snake},bend right=30] (-1/2,-3/4);
\draw[decorate,decoration={snake}] (B) -- +(0,-1/2);
\draw[densely dotted, thick] (C) -- +(0,1/2);
\draw[densely dotted, thick] (D) -- +(0,1/2);
\end{tikzpicture}}

\def\offdiaggfail{\begin{tikzpicture}[baseline]
\coordinate (A) at (-1/4,-1/4);
\coordinate (B) at (1/4,-1/4);
\coordinate (C) at (1/4,1/4);
\coordinate (D) at (-1/4,1/4);
\node[triality] at (B) {};
\node[triality] at (C) {};
\draw (B) -- (C);
\path (C) edge[decorate,decoration={snake},bend right=30] (-1/2,-3/4);
\path (B) edge[densely dotted, thick,bend left=30] (-1/2,3/4);
\draw[densely dotted, thick] (C) -- +(0,1/2);
\draw[decorate,decoration={snake}] (B) -- +(0,-1/2);
\end{tikzpicture}}

\def\offdiaghfail{\begin{tikzpicture}[baseline]
\coordinate (triality1) at (0,0);
\coordinate (triality2) at (1/2,0);
\draw[decorate,decoration={snake}] (triality1) -- +(south west:1/2);
\draw[decorate,decoration={snake}] (triality2) -- +(south east:1/2);
\draw[densely dotted, thick] (triality1) -- +(north west:1/2);
\draw[densely dotted, thick] (triality2) -- +(north east:1/2);
\draw (triality1) -- (triality2);
\node[triality] at (triality1) {};
\node[triality] at (triality2) {};
\end{tikzpicture}}

\def\offdiagjfail{\begin{tikzpicture}[baseline]
\coordinate (A) at (-1/4,-1/4);
\coordinate (B) at (1/4,-1/4);
\coordinate (C) at (1/4,1/4);
\coordinate (D) at (-1/4,1/4);
\node[triality] at (A) {};
\node[triality] at (B) {};
\node[triality] at (C) {};
\node[triality] at (D) {};
\draw[densely dotted, thick] (A) -- (B);
\draw (B) -- (C);
\draw[decorate,decoration={snake}] (C) -- (D);
\draw (D) -- (A);
\draw[decorate,decoration={snake}] (A) -- +(0,-1/2);
\draw[decorate,decoration={snake}] (B) -- +(0,-1/2);
\draw[densely dotted, thick] (C) -- +(-1/2,1/2);
\draw[densely dotted, thick] (D) -- +(1/2,1/2);
\end{tikzpicture}}

\def\offdiagkfail{\begin{tikzpicture}[baseline]
\node[rectangle, fill=gray!50, draw, very thick, minimum height=0.25cm, minimum width=0.5cm] (box) at (0,0) {};
\draw[decorate,decoration={snake}] (box.south west) -- +(south:1/2);
\draw[decorate,decoration={snake}] (box.south east) -- +(south:1/2);
\draw[densely dotted, thick] (box.north west) -- +(north:1/2);
\draw[densely dotted, thick] (box.north east) -- +(north:1/2);
\end{tikzpicture}}

\def\offdiaglfail{\begin{tikzpicture}[baseline]
\draw[densely dotted, thick] (-1/2,1/2) .. controls +(0,-1/2) and +(0,-1/2) .. (1/2,1/2);
\draw[decorate,decoration={snake}] (-1/2,-1/2) .. controls +(0,1/2) and +(0,1/2) .. (1/2,-1/2);
\end{tikzpicture}}

\def\offdiagm{\begin{tikzpicture}[baseline]
\coordinate (triality1) at (0,0);
\coordinate (triality2) at (1/2,0);
\draw[decorate,decoration={snake}] (triality1) -- +(south west:1/2);
\draw[decorate,decoration={snake}] (triality2) -- +(south east:1/2);
\draw[densely dotted, thick] (triality1) -- +(north west:1/2);
\draw[densely dotted, thick] (triality2) -- +(north east:1/2);
\draw (triality1) -- (triality2);
\node[triality] at (triality1) {};
\node[triality] at (triality2) {};
\end{tikzpicture} }

\def\offdiagn{\begin{tikzpicture}[baseline]
\draw[densely dotted, thick] (-1/2,1/2) .. controls +(0,-1/2) and +(0,-1/2) .. (1/2,1/2);
\draw[decorate,decoration={snake}] (-1/2,-1/2) .. controls +(0,1/2) and +(0,1/2) .. (1/2,-1/2);
\end{tikzpicture}}

\def\twoboxa{\begin{tikzpicture}[baseline]
\node[rectangle, fill=gray!50, draw, very thick, minimum height=0.25cm, minimum width=0.5cm] (box1) at (0,-1/4) {};
\node[rectangle, fill=gray!50, draw, very thick, minimum height=0.25cm, minimum width=0.5cm] (box2) at (0,1/4) {};
\draw[decorate,decoration={snake}] (box1.south west) -- +(south:1/2);
\draw[decorate,decoration={snake}] (box1.south east) -- +(south:1/2);
\draw (box1.north west) -- +(north:1/2);
\draw (box1.north east) -- +(north:1/2);
\draw (box2.south west) -- +(south:1/2);
\draw (box2.south east) -- +(south:1/2);
\draw[densely dotted, thick] (box2.north west) -- +(north:1/2);
\draw[densely dotted, thick] (box2.north east) -- +(north:1/2);
\end{tikzpicture}}

\def\twoboxb{\begin{tikzpicture}[baseline]
\node[rectangle, fill=gray!50, draw, very thick, minimum height=0.25cm, minimum width=0.5cm] (box) at (0,0) {};
\draw[decorate,decoration={snake}] (box.south west) -- +(south:1/2);
\draw[decorate,decoration={snake}] (box.south east) -- +(south:1/2);
\draw[densely dotted, thick] (box.north west) -- +(north:1/2);
\draw[densely dotted, thick] (box.north east) -- +(north:1/2);
\end{tikzpicture}}

\def\twoboxc{\begin{tikzpicture}[baseline]
\node[rectangle, fill=gray!50, draw, very thick, minimum height=0.25cm, minimum width=0.5cm] (box) at (0,0) {};
\draw (box.south west) -- +(south:1/2);
\draw (box.south east) -- +(south:1/2);
\draw (box.north west) -- +(north:1/2);
\draw (box.north east) -- +(north:1/2);
\end{tikzpicture}}

\def\twoboxd{\begin{tikzpicture}[baseline]
\node[rectangle, fill=gray!50, draw, very thick, minimum height=0.25cm, minimum width=0.5cm] (box1) at (0,-1/4) {};
\node[rectangle, fill=gray!50, draw, very thick, minimum height=0.25cm, minimum width=0.5cm] (box2) at (0,1/4) {};
\draw (box1.south west) -- +(south:1/2);
\draw (box1.south east) -- +(south:1/2);
\draw[decorate,decoration={snake}] (box1.north west) -- (box2.south west);
\draw[decorate,decoration={snake}] (box1.north east) -- (box2.south east);
\draw (box2.north west) -- +(north:1/2);
\draw (box2.north east) -- +(north:1/2);
\end{tikzpicture}}

\def\twoboxe{\begin{tikzpicture}[baseline]
\node[rectangle, fill=gray!50, draw, very thick, minimum height=0.25cm, minimum width=0.5cm] (box1) at (0,-1/4) {};
\node[rectangle, fill=gray!50, draw, very thick, minimum height=0.25cm, minimum width=0.5cm] (box2) at (0,1/4) {};
\draw (box1.south west) -- +(south:1/2);
\draw (box1.south east) -- +(south:1/2);
\draw (box1.north west) -- (box2.south west);
\draw (box1.north east) -- (box2.south east);
\draw (box2.north west) -- +(north:1/2);
\draw (box2.north east) -- +(north:1/2);
\end{tikzpicture} }

\def\twoboxf{ \begin{tikzpicture}[baseline]
\node[rectangle, fill=gray!50, draw, very thick, minimum height=0.25cm, minimum width=0.5cm] (box1) at (0,-3/4) {};
\node[rectangle, fill=gray!50, draw, very thick, minimum height=0.25cm, minimum width=0.5cm] (box2) at (0,-1/4) {};
\node[rectangle, fill=gray!50, draw, very thick, minimum height=0.25cm, minimum width=0.5cm] (box3) at (0,1/4) {};
\node[rectangle, fill=gray!50, draw, very thick, minimum height=0.25cm, minimum width=0.5cm] (box4) at (0,3/4) {};
\draw (box1.south west) -- +(south:1/2);
\draw (box1.south east) -- +(south:1/2);
\draw[decorate,decoration={snake}] (box1.north west) -- (box2.south west);
\draw[decorate,decoration={snake}] (box1.north east) -- (box2.south east);
\draw[densely dotted,thick] (box2.north west) -- (box3.south west);
\draw[densely dotted, thick] (box2.north east) -- (box3.south east);
\draw[decorate,decoration={snake}] (box3.north west) -- (box4.south west);
\draw[decorate,decoration={snake}] (box3.north east) -- (box4.south east);
\draw (box4.north west) -- +(north:1/2);
\draw (box4.north east) -- +(north:1/2);
\end{tikzpicture}}

\def\twoboxf{\begin{tikzpicture}[baseline]
\node[rectangle, fill=gray!50, draw, very thick, minimum height=0.25cm, minimum width=0.5cm] (box2) at (0,-1/2) {};
\node[rectangle, fill=gray!50, draw, very thick, minimum height=0.25cm, minimum width=0.5cm] (box3) at (0,0) {};
\node[rectangle, fill=gray!50, draw, very thick, minimum height=0.25cm, minimum width=0.5cm] (box4) at (0,1/2) {};
\draw (box2.south west) -- +(south:1/2);
\draw (box2.south east) -- +(south:1/2);
\draw[densely dotted,thick] (box2.north west) -- (box3.south west);
\draw[densely dotted,thick] (box2.north east) -- (box3.south east);
\draw[decorate,decoration={snake}] (box3.north west) -- (box4.south west);
\draw[decorate,decoration={snake}] (box3.north east) -- (box4.south east);
\draw (box4.north west) -- +(north:1/2);
\draw (box4.north east) -- +(north:1/2);
\end{tikzpicture}}

\def\twoboxg{\begin{tikzpicture}[baseline]
\node[rectangle, fill=gray!50, draw, very thick, minimum height=0.25cm, minimum width=0.5cm] (box1) at (0,-1/4) {};
\node[rectangle, fill=gray!50, draw, very thick, minimum height=0.25cm, minimum width=0.5cm] (box2) at (0,1/4) {};
\draw (box1.south west) -- +(south:1/2);
\draw (box1.south east) -- +(south:1/2);
\draw[decorate,decoration={snake}] (box1.north west) -- (box2.south west);
\draw[decorate,decoration={snake}] (box1.north east) -- (box2.south east);
\draw (box2.north west) -- +(north:1/2);
\draw (box2.north east) -- +(north:1/2);
\end{tikzpicture}}

\def\twoboxh{\begin{tikzpicture}[baseline]
\node[rectangle, fill=gray!50, draw, very thick, minimum height=0.25cm, minimum width=0.5cm] (box) at (0,0) {};
\draw (box.south west) -- +(south:1/2);
\draw (box.south east) -- +(south:1/2);
\draw (box.north west) -- +(north:1/2);
\draw (box.north east) -- +(north:1/2);
\end{tikzpicture}}

\def\twoboxz{\begin{tikzpicture}[baseline]
\node[rectangle, fill=gray!50, draw, very thick, minimum height=0.25cm, minimum width=0.5cm] (box1) at (0,-3/4) {};
\node[rectangle, fill=gray!50, draw, very thick, minimum height=0.25cm, minimum width=0.5cm] (box2) at (0,-1/4) {};
\node[rectangle, fill=gray!50, draw, very thick, minimum height=0.25cm, minimum width=0.5cm] (box3) at (0,1/4) {};
\node[rectangle, fill=gray!50, draw, very thick, minimum height=0.25cm, minimum width=0.5cm] (box4) at (0,3/4) {};
\draw (box1.south west) -- +(south:1/2);
\draw (box1.south east) -- +(south:1/2);
\draw[decorate,decoration={snake}] (box1.north west) -- (box2.south west);
\draw[decorate,decoration={snake}] (box1.north east) -- (box2.south east);
\draw[densely dotted,thick] (box2.north west) -- (box3.south west);
\draw[densely dotted, thick] (box2.north east) -- (box3.south east);
\draw[decorate,decoration={snake}] (box3.north west) -- (box4.south west);
\draw[decorate,decoration={snake}] (box3.north east) -- (box4.south east);
\draw (box4.north west) -- +(north:1/2);
\draw (box4.north east) -- +(north:1/2);
\end{tikzpicture} }

 \def\symsym{\begin{tikzpicture}[baseline,scale=1]
 \path (-0.3, -0.3) edge[magenta, thick] (-0.3,0.5);
  \path (0, -0.3) edge[magenta, thick] (0,0.5);
  \end{tikzpicture}}
   \def\symsyma{\begin{tikzpicture}[baseline,scale=1]
 \path (-0.3, -0.3) edge[magenta, thick] (0,0.5);
  \path (0, -0.3) edge[magenta, thick] (-0.3,0.5);
\end{tikzpicture}}

 \def\symsymb{\begin{tikzpicture}[baseline,scale=1]
 \path (-0.3, -0.3) edge[magenta, thick] (-0.3,0.5);
  \path (0, -0.3) edge[magenta, thick] (0,0.5);
    \path (0.3, -0.3) edge[magenta, thick] (0.3,0.5);
  \end{tikzpicture}}
   \def\symsymc{\begin{tikzpicture}[baseline,scale=1]
 \path (-0.3, -0.3) edge[magenta, thick] (0,0.5);
  \path (0, -0.3) edge[magenta, thick] (0.3,0.5);
     \path (0.3, -0.3) edge[magenta, thick] (-0.3,0.5);
    \end{tikzpicture}} 
        \def\symsymd{\begin{tikzpicture}[baseline,scale=1]
 \path (-0.3, -0.3) edge[magenta, thick] (0.3,0.5);
  \path (0, -0.3) edge[magenta, thick] (-0.3,0.5);
     \path (0.3, -0.3) edge[magenta, thick] (0,0.5);
\end{tikzpicture}}
 \def\symsyme{\begin{tikzpicture}[baseline,scale=1]
 \path (-0.3, -0.3) edge[magenta, thick] (-0.3,0.5);
  \path (0, -0.3) edge[magenta, thick] (0.3,0.5);
    \path (0.3, -0.3) edge[magenta, thick] (0,0.5);
  \end{tikzpicture}}
   \def\symsymf{\begin{tikzpicture}[baseline,scale=1]
 \path (-0.3, -0.3) edge[magenta, thick] (0.3,0.5);
  \path (0, -0.3) edge[magenta, thick] (0,0.5);
     \path (0.3, -0.3) edge[magenta, thick] (-0.3,0.5);
    \end{tikzpicture}} 
        \def\symsymg{\begin{tikzpicture}[baseline,scale=1]
 \path (-0.3, -0.3) edge[magenta, thick] (0,0.5);
  \path (0, -0.3) edge[magenta, thick] (-0.3,0.5);
     \path (0.3, -0.3) edge[magenta, thick] (0.3,0.5);
\end{tikzpicture}}
        \def\metinvmet{\begin{tikzpicture}[baseline,scale=1]
 \path (-0.3, -0.3) edge[bend left = 80, magenta, thick] (0.3,- 0.3);
 \path (-0.3, 0.5) edge[bend right = 80, magenta, thick] (0.3,0.5);
\end{tikzpicture}}
   \def\mettr{\begin{tikzpicture}[baseline,scale=1]
 \path (0, -0.3) edge[bend left = 90, magenta, thick] (0,0.5);
 \path (0, -0.3) edge[bend right = 90, magenta, thick] (0,0.5);
\end{tikzpicture}}
   \def\delt{\begin{tikzpicture}[baseline,scale=1]
 \path (0, -0.3) edge[magenta, thick] (0,0.5);
\end{tikzpicture}}

   \def\eps{\begin{tikzpicture}[baseline,scale=1]
 \path (-0.2, 0.366) edge[magenta, thick] (-0.2,-0.166);
  \path (-0.2, -0.166) edge[magenta, thick] (0.2,-0.166);
    \path (0.2, -0.166) edge[magenta, thick] (0.2,0.366);
\end{tikzpicture}}

   \def\epsin{\begin{tikzpicture}[baseline,scale=1]
 \path (-0.2, 0.366) edge[magenta, thick] (-0.2,-0.166);
  \path (-0.2, 0.366) edge[magenta, thick] (0.2, 0.366);
    \path (0.2, -0.166) edge[magenta, thick] (0.2,0.366);
\end{tikzpicture}}

\def\ids{\begin{tikzpicture}[baseline,scale=1]
\path (-0.2, 0.366) edge[magenta, thick] (0.2,-0.166);
\end{tikzpicture}} 
   \def\epsepsin{\begin{tikzpicture}[baseline,scale=1]
 \path (-0.2, 0.366) edge[magenta, thick] (-0.2,-0.166);
  \path (-0.2, -0.166) edge[magenta, thick] (0.2,-0.166);
    \path (0.2, -0.166) edge[magenta, thick] (0.2,0.366);
 \path (0.2, 0.366) edge[magenta, thick] (0.2,-0.166);
  \path (0.2, 0.366) edge[magenta, thick] (0.6,0.366);
    \path (0.6, -0.166) edge[magenta, thick] (0.6,0.366);
\end{tikzpicture}}

  \def\epsepsc{\begin{tikzpicture}[baseline,scale=1]
 \path (-0.2, 0.366) edge[magenta, thick] (-0.2,-0.16);
  \path (-0.2, -0.166) edge[magenta, thick] (0.2,-0.166);
    \path (0.2, -0.166) edge[magenta, thick] (0.2,0.366);
  \path (-0.2, 0.366) edge[magenta, thick] (0.2,0.366);
\end{tikzpicture}}

 \def\epsinr{\begin{tikzpicture}[baseline,scale=1]
 \path (-0.2, 0.366) edge[magenta, thick] (0.2,-0.166);
  \path (-0.2, 0.366) edge[magenta, thick] (0.2,0.366);
    \path (-0.2, -0.166) edge[magenta, thick] (0.2,0.366);
\end{tikzpicture}}

   \def\epsr{\begin{tikzpicture}[baseline,scale=1]
 \path (-0.2, 0.366) edge[magenta, thick] (0.2,-0.166);
  \path (-0.2, -0.166) edge[magenta, thick] (0.2,-0.166);
    \path (-0.2, -0.166) edge[magenta, thick] (0.2,0.366);
\end{tikzpicture}}

  \def\epsinveps{\begin{tikzpicture}[baseline,scale=1]
 \path (-0.2, 0.5) edge[magenta, thick] (-0.2,0.15);
  \path (-0.2, 0.15) edge[magenta, thick] (0.2,0.15);
    \path (0.2, 0.15) edge[magenta, thick] (0.2,0.5);
     \path (-0.2, -0.3) edge[magenta, thick] (-0.2,0.05);
  \path (-0.2, 0.05) edge[magenta, thick] (0.2,0.05);
    \path (0.2, 0.05) edge[magenta, thick] (0.2,-0.3);
\end{tikzpicture}}

  \def\epstri{\begin{tikzpicture}[baseline,scale=2]
 \path (-0.2, 0.47) edge[magenta, thick] (-0.2,-0.07);
  \path (-0.125, 0.515) edge[yellow, thick] (0.325,0.245);
    \path (0.325, 0.155) edge[cyan, thick] (-0.125,-0.115);
      \path (-0.2, 0.47) edge[magenta, thick, bend left = 40] (-0.418,0.909);
        \path (-0.125, 0.515) edge[yellow, thick, bend right = 40] (-0.382,0.931);
                \path (-0.2, -0.07) edge[magenta, thick, bend right = 30] (-0.418,- 0.809);
                        \path (0.2, -0.07) edge[magenta, thick, bend left = 30] (0.418,- 0.809);
                 \path (-0.125, -0.115) edge[cyan, thick, bend left = 30] (-0.32,- 0.809);
        \path (0.325, 0.155) edge[cyan, thick, bend right = 40] (0.8, 0.178);       
         \path (0.325, 0.245) edge[yellow, thick, bend left = 40] (0.8, 0.221);            
 \path (0.2, 0.47) edge[magenta, thick] (0.2,-0.07);
  \path (0.125, 0.515) edge[yellow, thick] (-0.325,0.245);
    \path (-0.325, 0.155) edge[cyan, thick] (0.125,-0.115);
      \path (0.2, 0.47) edge[magenta, thick, bend right = 40] (0.418,0.909);
        \path (0.125, 0.515) edge[yellow, thick, bend left = 40] (0.382,0.931);
                 \path (0.125, -0.115) edge[cyan, thick, bend right = 30] (0.32,- 0.809);
        \path (-0.325, 0.155) edge[cyan, thick, bend left = 40] (-0.8, 0.178);       
         \path (-0.325, 0.245) edge[yellow, thick, bend right = 40] (-0.8, 0.221);  
         \path (0.382, 0.931) edge[yellow, thick] (0.382, 1.231);
           \path (-0.382, 0.931) edge[yellow, thick] (-0.382, 1.231);
             \path (-0.382, 1.231) edge[yellow, thick] (0.382, 1.231);
                      \path (0.418, 0.909) edge[magenta, thick] (0.418, 1.275);
           \path (-0.418, 0.909) edge[magenta, thick] (-0.418, 1.275);
             \path (-0.418, 1.275) edge[magenta, thick] (0.418, 1.275);
               \path (-0.418, 0.909) edge[magenta, thick] (-0.418, 1.275);
             \end{tikzpicture}}

\def\epstrib{\begin{tikzpicture}[baseline,scale=2]
 \path (-0.2, 0.47) edge[magenta, thick] (-0.2,-0.07);
  \path (-0.125, 0.515) edge[yellow, thick] (0.325,0.245);
    \path (0.325, 0.155) edge[cyan, thick] (-0.125,-0.115);
      \path (-0.2, 0.47) edge[magenta, thick, bend left = 40] (-0.418,0.909);
        \path (-0.125, 0.515) edge[yellow, thick, bend right = 40] (-0.382,0.931);
                \path (-0.2, -0.07) edge[magenta, thick, bend right = 30] (0.418,- 0.809);
                        \path (0.2, -0.07) edge[magenta, thick, bend left = 30] (-0.418,- 0.809);
                 \path (-0.125, -0.115) edge[cyan, thick, bend left = 30] (0.32,- 0.809);
        \path (0.325, 0.155) edge[cyan, thick, bend right = 40] (0.8, 0.178);       
         \path (0.325, 0.245) edge[yellow, thick, bend left = 40] (0.8, 0.221);            
 \path (0.2, 0.47) edge[magenta, thick] (0.2,-0.07);
  \path (0.125, 0.515) edge[yellow, thick] (-0.325,0.245);
    \path (-0.325, 0.155) edge[cyan, thick] (0.125,-0.115);
      \path (0.2, 0.47) edge[magenta, thick, bend right = 40] (0.418,0.909);
        \path (0.125, 0.515) edge[yellow, thick, bend left = 40] (0.382,0.931);
                 \path (0.125, -0.115) edge[cyan, thick, bend right = 30] (-0.32,- 0.809);
        \path (-0.325, 0.155) edge[cyan, thick, bend left = 40] (-0.8, 0.178);       
         \path (-0.325, 0.245) edge[yellow, thick, bend right = 40] (-0.8, 0.221);  
         \path (0.382, 0.931) edge[yellow, thick] (0.382, 1.231);
           \path (-0.382, 0.931) edge[yellow, thick] (-0.382, 1.231);
             \path (-0.382, 1.231) edge[yellow, thick] (0.382, 1.231);
                      \path (0.418, 0.909) edge[magenta, thick] (0.418, 1.275);
           \path (-0.418, 0.909) edge[magenta, thick] (-0.418, 1.275);
             \path (-0.418, 1.275) edge[magenta, thick] (0.418, 1.275);
               \path (-0.418, 0.909) edge[magenta, thick] (-0.418, 1.275);
             \end{tikzpicture}}

               \def\epstric{\begin{tikzpicture}[baseline,scale=2]
 \path (-0.2, 0.47) edge[magenta, thick] (-0.2,-0.07);
  \path (-0.125, 0.515) edge[yellow, thick] (0.325,0.245);
    \path (0.325, 0.155) edge[cyan, thick] (-0.125,-0.115);
      \path (-0.2, 0.47) edge[magenta, thick, bend left = 40] (-0.418,0.909);
        \path (-0.125, 0.515) edge[yellow, thick, bend right = 40] (-0.382,0.931);
                \path (-0.2, -0.07) edge[magenta, thick, bend right = 30] (-0.418,- 0.809);
                        \path (0.2, -0.07) edge[magenta, thick, bend left = 30] (0.418,- 0.809);
        \path (0.325, 0.155) edge[cyan, thick, bend right = 40] (0.8, 0.178);       
         \path (0.325, 0.245) edge[yellow, thick, bend left = 40] (0.8, 0.221);            
 \path (0.2, 0.47) edge[magenta, thick] (0.2,-0.07);
  \path (0.125, 0.515) edge[yellow, thick] (-0.325,0.245);
    \path (-0.325, 0.155) edge[cyan, thick] (0.125,-0.115);
      \path (0.2, 0.47) edge[magenta, thick, bend right = 40] (0.418,0.909);
        \path (0.125, 0.515) edge[yellow, thick, bend left = 40] (0.382,0.931);
                 \path (0.125, -0.115) edge[cyan, thick, bend right = 22] (0.17,- 0.609);
                    \path (-0.125, -0.115) edge[cyan, thick, bend left = 22] (-0.17,- 0.609);
                      \path (0.18, -0.660) edge[cyan, thick, bend right = 16] (0.32,- 0.809);
                          \path (-0.18, -0.660) edge[cyan, thick, bend left = 16] (-0.32,- 0.809);
                   \path (-0.17, -0.609) edge[cyan, thick] (0.17,- 0.609);
                    \path (-0.18, -0.660) edge[cyan, thick] (0.18,- 0.660);
        \path (-0.325, 0.155) edge[cyan, thick, bend left = 40] (-0.8, 0.178);       
         \path (-0.325, 0.245) edge[yellow, thick, bend right = 40] (-0.8, 0.221);  
         \path (0.382, 0.931) edge[yellow, thick] (0.382, 1.231);
           \path (-0.382, 0.931) edge[yellow, thick] (-0.382, 1.231);
             \path (-0.382, 1.231) edge[yellow, thick] (0.382, 1.231);
                      \path (0.418, 0.909) edge[magenta, thick] (0.418, 1.275);
           \path (-0.418, 0.909) edge[magenta, thick] (-0.418, 1.275);
             \path (-0.418, 1.275) edge[magenta, thick] (0.418, 1.275);
               \path (-0.418, 0.909) edge[magenta, thick] (-0.418, 1.275);
             \end{tikzpicture}}

                         \def\epstrid{\begin{tikzpicture}[baseline,scale=2]
              \path (0.8, 0.221) edge[yellow, thick, bend left = 25] (0.125, 0.514);          
                  \path (-0.8, 0.221) edge[yellow, thick, bend right = 25] (-0.125, 0.514);   
                    \path (-0.125, 0.514) edge[yellow, thick] (0.125, 0.514);    
                \path (0.8, 0.178) edge[cyan, thick, bend right = 25] (0.125, -0.115);          
                  \path (-0.8, 0.178) edge[cyan, thick, bend left= 25] (-0.125, -0.115);   
                    \path (-0.125, -0.115) edge[cyan, thick] (0.125, -0.115);       
                      \path (0.18, -0.660) edge[cyan, thick, bend right = 16] (0.32,- 0.809);
                          \path (-0.18, -0.660) edge[cyan, thick, bend left = 16] (-0.32,- 0.809);
                    \path (-0.18, -0.660) edge[cyan, thick] (0.18,- 0.660);
                      \path (-0.25, -0.615) edge[magenta, thick] (0.25,- 0.615);
                                \path (0.25, -0.615) edge[magenta, thick, bend right = 16] (0.40,- 0.809);
                          \path (-0.25, -0.615) edge[magenta, thick, bend left = 16] (-0.40,- 0.809);    
             \end{tikzpicture}}

              \def\epstriz{\begin{tikzpicture}[baseline,scale=2]
 \path (-0.2, 0.47) edge[magenta, thick] (-0.2,-0.07);
  \path (-0.125, 0.515) edge[yellow, thick] (0.325,0.245);
   \path (-0.2, 0.47) edge[magenta, thick, bend left = 40] (-0.418,0.909);
     \path (-0.125, 0.515) edge[yellow, thick, bend right = 40] (-0.382,0.931);
         \path (0.325, 0.245) edge[yellow, thick, bend left = 40] (0.8, 0.221);   
             \path (0.325, 0.155) edge[cyan, thick, bend right = 40] (0.8, 0.179);   
           \path (-0.125, -0.115) edge[cyan, thick] (0.325,0.155);  
                \path (-0.125, -0.115) edge[cyan, thick, bend left = 40] (-0.382,-0.531); 
                   \path (-0.2, -0.07) edge[magenta, thick, bend right = 40] (-0.418,-0.509);
             \end{tikzpicture}}

             \def\tripresx{\begin{tikzpicture}[baseline = - 2pt]
\path{ (1, 1) edge[thick, bend left=40, cyan]  (2, -0.8)};
\draw[yellow, thick] (1,1) .. controls +(0,1.3) and +(0,1.3) ..  (-2, -0.8) ;
 \draw[magenta, thick] (1,1) .. controls +(0,-1/2) and +(0,1/2) ..  (0.3, -0.3) ;
  \draw[magenta, thick] (-1.2,-0.8) .. controls +(0,1/2) and +(0,1/2) ..  (-0.3, -0.3) ;
  \path{ (-0.3, -0.3) edge[thick, magenta]  (0.3, -0.3)};
\node[triality] at (1, 1) {};
\end{tikzpicture}}

         \def\tripresy{\begin{tikzpicture}[baseline = - 2pt]
\path{ (1, 1) edge[thick, bend left=40, magenta]  (2, -0.8)};
\draw[cyan, thick] (1,1) .. controls +(0,1.3) and +(0,1.3) ..  (-2, -0.8) ;
 \draw[yellow, thick] (1,1) .. controls +(0,-1/2) and +(0,1/2) ..  (0.3, -0.3) ;
  \draw[yellow, thick] (-1.2,-0.8) .. controls +(0,1/2) and +(0,1/2) ..  (-0.3, -0.3) ;
  \path{ (-0.3, -0.3) edge[thick, yellow]  (0.3, -0.3)};
\node[triality] at (1, 1) {};
\end{tikzpicture}}
         \def\tripresz{\begin{tikzpicture}[baseline = - 2pt]
\path{ (1, 1) edge[thick, bend left=40, yellow]  (2, -0.8)};
\draw[magenta, thick] (1,1) .. controls +(0,1.3) and +(0,1.3) ..  (-2, -0.8) ;
 \draw[cyan, thick] (1,1) .. controls +(0,-1/2) and +(0,1/2) ..  (0.3, -0.3) ;
  \draw[cyan, thick] (-1.2,-0.8) .. controls +(0,1/2) and +(0,1/2) ..  (-0.3, -0.3) ;
  \path{ (-0.3, -0.3) edge[thick, cyan]  (0.3, -0.3)};
\node[triality] at (1, 1) {};
\end{tikzpicture}}

\def\epsa{\begin{tikzpicture}[baseline = - 3pt]
 \draw[magenta, thick] (0.5,0.7) .. controls +(-1/8,-1/8) and +(1/8,1/4) ..  (0.2, -0.1) ;
 \draw[magenta, thick] (-0.5,0.7) .. controls +(1/8,-1/8) and +(-1/8,1/4) ..  (-0.2, -0.1) ;
 \path{ (-0.2, -0.1) edge[thick, magenta]  (0.2, -0.1)};
\end{tikzpicture}}

\def\epsb{\begin{tikzpicture}[baseline = - 3pt]
 \draw[cyan, thick] (0.5,0.7) .. controls +(-1/8,-1/8) and +(1/8,1/4) ..  (0.2, -0.1) ;
 \draw[cyan, thick] (-0.5,0.7) .. controls +(1/8,-1/8) and +(-1/8,1/4) ..  (-0.2, -0.1) ;
 \path{ (-0.2, -0.1) edge[thick, cyan]  (0.2, -0.1)};
\end{tikzpicture}}

\def\epsc{\begin{tikzpicture}[baseline = - 3pt]
 \draw[yellow, thick] (0.5,0.7) .. controls +(-1/8,-1/8) and +(1/8,1/4) ..  (0.2, -0.1) ;
 \draw[yellow, thick] (-0.5,0.7) .. controls +(1/8,-1/8) and +(-1/8,1/4) ..  (-0.2, -0.1) ;
 \path{ (-0.2, -0.1) edge[thick, yellow]  (0.2, -0.1)};
\end{tikzpicture}}

\def\triconagedd{\begin{tikzpicture}[baseline = - 2pt]
 \node[rectangle, fill=magenta, draw, very thick, minimum height=0.30cm, minimum width=0.30cm] (box) at (0,-0.3) {};
 \path{ (box.north west) edge[thick, bend left = 30, magenta]  (-0.1, 0.4)};
  \path{ (box.north east) edge[thick, bend right = 30, magenta]  (0.12, 0.4)};
\node[rectangle, fill=orange!50, draw, very thick, minimum height=0.25cm, minimum width=0.25cm] (box) at (0,1) {};
\draw[solid, thick, cyan] (box.south west) -- +(south:1/4);
\draw[solid, thick, cyan] (box.south east) -- +(south:1/4);
 \path{(box.north east) edge[magenta, thick, bend left = 20]  (0.5, 2.1)} ; 
  \path{(box.north west) edge[magenta, thick, bend right = 20]  (-0.5, 2.1)} ; 
\node[rectangle, fill=orange!50, draw, very thick, minimum height=0.25cm, minimum width=0.25cm] (box) at (0,1/2) {};
\node at (2, 0) {,} ;
\end{tikzpicture}}

\def\sympboxa{\begin{tikzpicture}[baseline = - 2pt]
\node[rectangle, fill=orange!50, draw, very thick, minimum height=0.25cm, minimum width=0.25cm] (box) at (0,0) {};
\draw[solid, thick, magenta] (box.north west) -- +(north west:1/2);
\draw[solid, thick, magenta] (box.north east) -- +(north east:1/2);
\end{tikzpicture}}

\def\sympboxb{\begin{tikzpicture}[baseline = - 2pt]
\node[rectangle, fill=orange!50, draw, very thick, minimum height=0.25cm, minimum width=0.25cm] (box) at (0,0) {};
\draw[solid, thick, yellow] (box.north west) -- +(north west:1/2);
\draw[solid, thick, yellow] (box.north east) -- +(north east:1/2);
\end{tikzpicture}}

\def\sympboxc{\begin{tikzpicture}[baseline = - 2pt]
\node[rectangle, fill=orange!50, draw, very thick, minimum height=0.25cm, minimum width=0.25cm] (box) at (0,0) {};
\draw[solid, thick, cyan] (box.north west) -- +(north west:1/2);
\draw[solid, thick, cyan] (box.north east) -- +(north east:1/2);
\end{tikzpicture}}

\def\sympboxad{\begin{tikzpicture}[baseline = - 2pt]
\node[rectangle, fill=orange!50, draw, very thick, minimum height=0.25cm, minimum width=0.25cm] (box) at (0,0) {};
\draw[solid, thick, magenta] (box.south west) -- +(south west:1/2);
\draw[solid, thick, magenta] (box.south east) -- +(south east:1/2);
\end{tikzpicture}}

\def\sympboxbd{\begin{tikzpicture}[baseline = - 2pt]
\node[rectangle, fill=orange!50, draw, very thick, minimum height=0.25cm, minimum width=0.25cm] (box) at (0,0) {};
\draw[solid, thick, yellow] (box.south west) -- +(south west:1/2);
\draw[solid, thick, yellow] (box.south east) -- +(south east:1/2);
\end{tikzpicture}}

\def\sympboxcd{\begin{tikzpicture}[baseline = - 2pt]
\node[rectangle, fill=orange!50, draw, very thick, minimum height=0.25cm, minimum width=0.25cm] (box) at (0,0) {};
\draw[solid, thick, cyan] (box.south west) -- +(south west:1/2);
\draw[solid, thick, cyan] (box.south east) -- +(south east:1/2);
\end{tikzpicture}}

\def\blackboxaa{\begin{tikzpicture}[baseline = - 2pt]
\node[rectangle, fill=orange!50, draw, very thick, minimum height=0.25cm, minimum width=0.25cm] (box) at (0,0) {};
\draw[solid, thick, magenta] (box.south west) -- +(south west:1/2);
\draw[solid, thick, magenta] (box.south east) -- +(south east:1/2);
\draw[solid, thick, magenta] (box.north west) -- +(north west:1/2);
\draw[solid, thick, magenta] (box.north east) -- +(north east:1/2);
\end{tikzpicture}}

\def\blackboxbb{\begin{tikzpicture}[baseline = - 2pt]
\node[rectangle, fill=orange!50, draw, very thick, minimum height=0.25cm, minimum width=0.25cm] (box) at (0,0) {};
\draw[solid, thick, yellow] (box.south west) -- +(south west:1/2);
\draw[solid, thick, yellow] (box.south east) -- +(south east:1/2);
\draw[solid, thick, yellow] (box.north west) -- +(north west:1/2);
\draw[solid, thick, yellow] (box.north east) -- +(north east:1/2);
\end{tikzpicture}}

\def\blackboxcc{\begin{tikzpicture}[baseline = - 2pt]
\node[rectangle, fill=orange!50, draw, very thick, minimum height=0.25cm, minimum width=0.25cm] (box) at (0,0) {};
\draw[solid, thick, cyan] (box.south west) -- +(south west:1/2);
\draw[solid, thick, cyan] (box.south east) -- +(south east:1/2);
\draw[solid, thick, cyan] (box.north west) -- +(north west:1/2);
\draw[solid, thick, cyan] (box.north east) -- +(north east:1/2);
\end{tikzpicture}}

\def\uboxaa{\begin{tikzpicture}[baseline = - 2pt]
\node[rectangle, fill=orange!50, draw, very thick, minimum height=0.25cm, minimum width=0.25cm] (boxA) at (0,0) {};

\node[rectangle, fill=orange!50, draw, very thick, minimum height=0.25cm, minimum width=0.25cm] (boxB) at (1,0) {};
\draw[solid, thick, magenta] (boxA.south east) .. controls +(0,-1/2) and +(0,-1/2) .. (boxB.south west);
\draw[solid, thick, magenta] (boxA.south west) .. controls +(0,-1) and +(0,-1) .. (boxB.south east);
\end{tikzpicture}}

\def\uboxbb{\begin{tikzpicture}[baseline = - 2pt]
\node[rectangle, fill=orange!50, draw, very thick, minimum height=0.25cm, minimum width=0.25cm] (boxA) at (0,0) {};

\node[rectangle, fill=orange!50, draw, very thick, minimum height=0.25cm, minimum width=0.25cm] (boxB) at (1,0) {};
\draw[solid, thick, yellow] (boxA.south east) .. controls +(0,-1/2) and +(0,-1/2) .. (boxB.south west);
\draw[solid, thick, yellow] (boxA.south west) .. controls +(0,-1) and +(0,-1) .. (boxB.south east);
\end{tikzpicture}}

\def\uboxcc{\begin{tikzpicture}[baseline = - 2pt]
\node[rectangle, fill=orange!50, draw, very thick, minimum height=0.25cm, minimum width=0.25cm] (boxA) at (0,0) {};

\node[rectangle, fill=orange!50, draw, very thick, minimum height=0.25cm, minimum width=0.25cm] (boxB) at (1,0) {};
\draw[solid, thick, cyan] (boxA.south east) .. controls +(0,-1/2) and +(0,-1/2) .. (boxB.south west);
\draw[solid, thick, cyan] (boxA.south west) .. controls +(0,-1) and +(0,-1) .. (boxB.south east);
\end{tikzpicture}}

\newcommand{\Sym}{\mathrm{Sym}}
\newcommand{\Spin}{\mathrm{Spin}}

\title{Triality and the Magic Square of Hans Freudenthal}

\author{Jonathan Holland}
\author{George Sparling}
\address{Laboratory of Axiomatics, Mathematics Department,\\ University of Pittsburgh}
\date{\today}
\begin{document}

\begin{abstract}
We study real triality structures through their intrinsic tensor algebra. Starting from a single triality symbol, we construct the associated Lie algebra of two-triality operators, prove the Jacobi identity, and identify the resulting algebra uniformly with the corresponding entry of the magic square. We then examine the natural invariant bilinear forms and the Clifford-theoretic structures arising from this construction. In low dimension, the triality formalism also recovers classical arithmetic data: in the \(2\times2\times2\) case, the associated binary quadratic forms have a common discriminant and fit naturally into the Bhargava cube picture.
\end{abstract}

\maketitle

\section{Introduction}

The Freudenthal magic square gives a uniform construction of a family of Lie algebras that includes the exceptional series. \cite{Freudenthal1954I,Freudenthal1954II,Tits1966,Vinberg1966,BartonSudbery2003,LandsbergManivel2001,Elduque2006}  In the standard accounts one starts from the real composition algebras and their split forms.  The present paper isolates the tensorial datum that is actually used in the construction: triality.  Once one works with that datum directly, the Lie algebras in the square can be recovered from pairs of trialities by explicit diagrammatic formulas.

We fix a field $\mathbb F$ of characteristic different from two.  The basic data are three vector spaces with non-degenerate quadratic forms and a trilinear form relating them.  The contraction identities imposed on this trilinear form are the same identities that arise in the classical cases from multiplicativity of the norm in the real, complex, quaternionic, and octonionic composition algebras, together with their split analogues.  Triality is therefore the common tensorial structure behind the dimensions $1$, $2$, $4$, and $8$.

The first part of the paper treats a single triality, in the sense going back to {\'E}lie Cartan. \cite{Cartan1925}  We work throughout in abstract tensor notation and use diagrams to control the contractions.  From the basic identities one obtains the triality constraint and hence the familiar Adolf Hurwitz obstruction: if a triality exists on three mutually isometric spaces of dimension $n$, then \cite{Hurwitz1898,Hurwitz1922}
\[
n^2(n-1)(n-2)(n-4)(n-8)=0.
\]
Over a field of characteristic zero this leaves only the dimensions
\[
n=1,2,4,8.
\]
The same calculation determines the dimensions of the triality-preserving orthogonal algebras:
\[
\tau_1=0,\qquad \tau_2=2,\qquad \tau_4=9,\qquad \tau_8=28.
\]
These integers are small, but they control the later construction.

The second part of the paper uses two trialities.  From the two triality-preserving orthogonal algebras and three mixed tensor pieces we define an anti-symmetric bracket by diagrammatic formulas.  Writing the bracket down is straightforward.  The work is to prove that it is well defined, that it is compatible with the triality relations, and that it satisfies the Jacobi identity.  Once this is done, the resulting Lie algebra has dimension
\[
3nn' + \tau_n + \tau_{n'},
\]
for trialities of dimensions $n$ and $n'$.  For $(n,n')\in\{1,2,4,8\}^2$ one obtains
\[
\begin{array}{|c|cccc|}
\hline
&1&2&4&8\\
\hline
1&3&8&21&52\\
2&8&16&35&78\\
4&21&35&66&133\\
8&52&78&133&248\\
\hline
\end{array}
\]
which are the dimensions of the entries in the magic square.  The identifications carried out later in the paper recover the expected complex Lie algebras, including $F_4$, $E_6$, $E_7$, and $E_8$.

This way of organizing the argument keeps the tensor identities in view. \cite{LandsbergManivel2002,Elduque2004,Elduque2006,Evans2009}  Nothing essential is hidden in auxiliary algebraic structure.  The diagrams are used as bookkeeping; they keep the contractions legible and reduce the amount of coordinate computation.

The final part of the paper turns to arithmetic.  We first examine the two-dimensional symplectic case.  A triality symbol on three symplectic $2$-spaces gives three binary quadratic forms with the same discriminant.  This is the point where Carl Friedrich Gauss composition and Manjul Bhargava's reinterpretation of higher composition laws enter the discussion. \cite{Gauss1966,Bhargava2004I,Bhargava2004II,Krutelevich2007,BhargavaGross2014}  The purpose of that section is limited and concrete.  We isolate the part of triality that survives in this arithmetic setting and determine the normalization forced by the tensor calculus.

The paper is organized as follows.  Section~\ref{sec:trialities} defines trialities and recalls the basic examples in dimensions $1$, $2$, $4$, and $8$.  The next sections develop the diagrammatic identities for a single triality, derive the triality constraint, and obtain the corresponding restriction on the possible dimensions.  We then study the triality-preserving orthogonal algebra and use a pair of trialities to construct the Lie algebra associated with the magic square.  After verifying skew-symmetry, consistency, and the Jacobi identity, we examine the Killing form and the relation with a universal Clifford algebra.  The resulting complex Lie algebras are then identified with the expected entries of the Freudenthal square.  The final section concerns the arithmetic material just described.

Throughout, all arguments are written in abstract tensor notation.  Coordinates are introduced only for models and identifications.  Once the triality tensors have been isolated, no appeal is made to the internal multiplication of a composition algebra.  The point is to keep the construction at the level where it is actually proved.

\section{Trialities and basic examples}\label{sec:trialities}

We begin by fixing the abstract notion of triality used throughout the paper and by recording the basic examples that will be needed later.  The examples of primary interest occur for $n=8$ and their restrictions to $n=4$, $n=2$, and $n=1$.  These are the cases that ultimately enter the construction of the triality Lie algebras and the recovery of the Freudenthal magic square.

\subsection{The abstract definition}

Let $\mathbb{F}$ be a field of characteristic different from two.  Let $\mathbb{V}$ be a vector space over $\mathbb{F}$ equipped with a metric $g$, that is, a non-degenerate symmetric bilinear form on $\mathbb{V}$.

Let $\mathbb{S}$ be a set with three elements.  To each $s\in\mathbb{S}$ we assign a subspace $\mathbb{V}_s\subset \mathbb{V}$ of non-negative integer dimension $n_s$ such that:
\begin{itemize}
\item $\mathbb{V}_s\cap \mathbb{V}_t=\{0\}$ whenever $s$ and $t$ are distinct,
\item $\sum_{s\in\mathbb{S}}\mathbb{V}_s=\mathbb{V}$,
\item the subspaces $\{\mathbb{V}_s:s\in\mathbb{S}\}$ are pairwise orthogonal with respect to $g$,
\item the numbers $\{n_s:s\in\mathbb{S}\}$ are all equal in $\mathbb{F}$.
\end{itemize}
We denote this common value in $\mathbb{F}$ by $n$.  In characteristic zero this simply says that each $\mathbb{V}_s$ has dimension $n$, so that $\mathbb{V}$ has dimension $3n$.

A \emph{triality} on $\mathbb{V}$ is a symmetric trilinear form $\tau$ on $\mathbb{V}$ with the property that $\tau(x,y,z)$ vanishes unless $x\in\mathbb{V}_s$, $y\in\mathbb{V}_t$, and $z\in\mathbb{V}_u$ for pairwise distinct $s,t,u\in\mathbb{S}$.  For fixed $x,y\in\mathbb{V}$ there is then a unique element of $\mathbb{V}$, denoted $xy$, such that
\[
\tau(x,y,z)=g(xy,z)
\]
for all $z\in\mathbb{V}$.  In this way $\mathbb{V}$ acquires a bilinear commutative product.  If $s,t,u$ are pairwise distinct and $x\in\mathbb{V}_s$, $y\in\mathbb{V}_t$, then
\begin{itemize}
\item $x^2=0$,
\item $xy\in\mathbb{V}_u$.
\end{itemize}
Finally, we require that the triality be of Clifford algebra type: for any distinct $s,t,u\in\mathbb{S}$, any $y\in\mathbb{V}_t$, and any $x\in\mathbb{V}_s+\mathbb{V}_u$, one has
\[
 x(xy)=g(x,x)y.
\]

\subsection{The standard $n=8$ trialities}

There are two basic trialities with $n=8$, corresponding to the compact and split orthogonal groups.  These are the fundamental models from which the lower-dimensional examples will be obtained by restriction.

\begin{itemize}
\item \textbf{The $\mathbb{O}(8)$ triality.}
\end{itemize}
This triality is determined by the sixty-four-term trinomial in the vector variables $x,y,z\in\mathbb{R}^8$:
\[
F=\tau(x,y,z)
\]
\[
\hspace{-60pt}  = z_1(x_1y_1  - x_2y_2 + x_3y_3+x_4y_4-x_5y_5-x_6y_6-x_7y_7-x_8y_8) +z_2(x_1y_2 + x_2y_1 - x_3y_4 + x_4y_3 - x_5 y_6 + x_6 y_5 - x_7y_8+x_8y_7)
\]
\[
\hspace{-60pt}  +z_3(-x_1y_3 + x_2y_4 + x_3y_1+x_4y_2 - x_5 y_7+x_6 y_8 + x_7y_5-x_8y_6)+z_4( - x_1 y_4- x_2y_3 - x_3 y_2 + x_4y_1 - x_5y_8 - x_6y_7+x_7y_6+x_8y_5)
\]
\[
\hspace{-60pt}  +z_5(-x_1y_5- x_2y_6 - x_3y_7 - x_4y_8 - x_5y_1 +x_6y_2-x_7y_3-x_8y_4) + z_6(-x_1y_6 + x_2y_5 + x_3y_8 - x_4y_7 - x_5y_2 - x_6 y_1  - x_7 y_4 + x_8y_3)
\]
\[
\hspace{-60pt} +z_7(x_1y_7+x_2y_8 -x_3y_5-x_4y_6-x_5y_3-x_6y_4+x_7y_1-x_8y_2) +z_8(-x_1y_8+x_2y_7-x_3y_6+x_4y_5+x_5y_4-x_6y_3-x_7y_2-x_8y_1).
\]
The trinomial $F$ obeys the differential relations of $\mathbb{O}(8)$ triality:
\[
\nabla_x(F)\cdot\nabla_x(F)=(y\cdot y)(z\cdot z),\qquad
\nabla_y(F)\cdot\nabla_y(F)=(z\cdot z)(x\cdot x),\qquad
\nabla_z(F)\cdot\nabla_z(F)=(x\cdot x)(y\cdot y).
\]
Here the dot products are the standard ones on Euclidean eight-space and $\nabla_x$, $\nabla_y$, and $\nabla_z$ are the gradient vectors with respect to $x$, $y$, and $z$.

\begin{itemize}
\item \textbf{The $\mathbb{O}(4,4)$ triality.}
\end{itemize}
This triality is determined by the sixty-four-term trinomial in the vector variables $x,y,z\in\mathbb{R}^8$:
\[
G=\tau(x,y,z)
\]
\[
\hspace{-60pt}  = z_1(x_1y_1  - x_2y_2 + x_3y_3+x_4y_4+x_5y_5+x_6y_6+x_7y_7+x_8y_8) +z_2(x_1y_2 + x_2y_1 - x_3y_4 + x_4y_3 + x_5 y_6 - x_6 y_5 +x_7y_8-x_8y_7)
\]
\[
\hspace{-60pt}  +z_3(-x_1y_3 + x_2y_4 + x_3y_1+x_4y_2 +x_5 y_7-x_6 y_8 - x_7y_5+x_8y_6)+z_4( - x_1 y_4- x_2y_3 - x_3 y_2 + x_4y_1 + x_5y_8 +x_6y_7-x_7y_6-x_8y_5)
\]
\[
\hspace{-60pt}  +z_5(x_1y_5 + x_2y_6 + x_3y_7 + x_4y_8 + x_5y_1 -x_6y_2+x_7y_3+x_8y_4) + z_6(x_1y_6 - x_2y_5 - x_3y_8 + x_4y_7 + x_5y_2 + x_6 y_1  + x_7 y_4 - x_8y_3)
\]
\[
\hspace{-60pt} +z_7(-x_1y_7-x_2y_8 +x_3y_5+x_4y_6+x_5y_3+x_6y_4-x_7y_1+x_8y_2) +z_8(x_1y_8-x_2y_7+x_3y_6-x_4y_5-x_5y_4+x_6y_3+x_7y_2+x_8y_1).
\]
Formally, $G$ is obtained from $F$ by multiplying the variables $x_k$, $y_k$, and $z_k$ by a commuting quantity $j$ for $k=5,6,7,8$, expanding, and quotienting by the relation $j^2+1=0$.  The trinomial $G$ obeys the differential relations of $\mathbb{O}(4,4)$ triality:
\[
\hspace{-85pt} g^{-1}(\nabla_x(G),\nabla_x(G))=g(y,y)g(z,z),\qquad g^{-1}(\nabla_y(G),\nabla_y(G))=g(z,z)g(x,x),\qquad g^{-1}(\nabla_z(G),\nabla_z(G))=g(x,x)g(y,y).
\]
Here the metric $g$ of signature $(4,4)$ is given by
\[
 g(x,x)=x_1^2+x_2^2+x_3^2+x_4^2-x_5^2-x_6^2-x_7^2-x_8^2.
\]
Although the specific coordinate formulas are not unique, they are unique up to isomorphism, since they are governed by the corresponding Clifford algebra structures.

\subsection{The derived $n=4$ trialities}

By restricting the split $n=8$ model one obtains two basic trialities with $n=4$.

\begin{itemize}
\item \textbf{The $\mathbb{O}(4)$ triality.}
\end{itemize}
Suppressing the variables $x_5,x_6,x_7,x_8$, $y_5,y_6,y_7,y_8$, and $z_5,z_6,z_7,z_8$ in $G$ gives the sixteen-term trinomial $H$ in the variables $x,y,z\in\mathbb{R}^4$:
\[
\hspace{-85pt}H = z_1(x_1y_1  - x_2y_2 + x_3y_3+x_4y_4) + z_2(x_1y_2 + x_2y_1 - x_3y_4 + x_4y_3) + z_3(-x_1y_3 + x_2y_4 + x_3y_1+x_4y_2) + z_4( - x_1 y_4- x_2y_3 - x_3 y_2 + x_4y_1).
\]
The trinomial $H$ obeys the differential relations of $\mathbb{O}(4)$ triality:
\[
\nabla_x(H)\cdot\nabla_x(H)=(y\cdot y)(z\cdot z),\qquad
\nabla_y(H)\cdot\nabla_y(H)=(z\cdot z)(x\cdot x),\qquad
\nabla_z(H)\cdot\nabla_z(H)=(x\cdot x)(y\cdot y).
\]
Here the dot products are the standard ones on Euclidean four-space.  If we introduce the quaternions
\[
 x=x_1+ix_2-jx_3-kx_4,\qquad y=y_1+iy_2+jy_3+ky_4,\qquad z=z_1-iz_2+jz_3+kz_4,
\]
where $i$, $j$, and $k$ pairwise anti-commute and square to $-1$, then $H=\Re(xyz)$.

\begin{itemize}
\item \textbf{The $\mathbb{O}(2,2)$ triality.}
\end{itemize}
Suppressing instead the variables $x_3,x_4,x_5,x_6$, $y_3,y_4,y_5,y_6$, and $z_3,z_4,z_5,z_6$ in $G$ gives the sixteen-term trinomial $K$:
\[
\hspace{-85pt} K = z_1(x_1y_1  - x_2y_2+x_7y_7+x_8y_8) +z_2(x_1y_2 + x_2y_1 +x_7y_8-x_8y_7) +z_7(-x_1y_7-x_2y_8 -x_7y_1+x_8y_2) +z_8(x_1y_8-x_2y_7+x_7y_2+x_8y_1).
\]
The trinomial $K$ obeys the differential relations of $\mathbb{O}(2,2)$ triality:
\[
 K_{x_1}^2 + K_{x_2}^2 - K_{x_7}^2 - K_{x_8}^2 = (y_1^2 + y_2^2 - y_7^2 - y_8^2)( z_1^2 + z_2^2 - z_7^2 - z_8^2),
\]
\[
 K_{y_1}^2 + K_{y_2}^2 - K_{y_7}^2 - K_{y_8}^2 = (z_1^2 + z_2^2 - z_7^2 - z_8^2)(x_1^2 + x_2^2 - x_7^2 - x_8^2),
\]
\[
 K_{z_1}^2 + K_{z_2}^2 - K_{z_7}^2 - K_{z_8}^2 = (x_1^2 + x_2^2 - x_7^2 - x_8^2)( y_1^2 + y_2^2 - y_7^2 - y_8^2).
\]

\subsection{The derived $n=2$ and $n=1$ trialities}

By further restriction one obtains the remaining cases needed later.

\begin{itemize}
\item \textbf{The $\mathbb{O}(2)$ triality.}
\end{itemize}
Suppressing the variables $x_3,x_4$, $y_3,y_4$, and $z_3,z_4$ in $H$ gives the trinomial $J$:
\[
J = z_1(x_1y_1  - x_2y_2) + z_2(x_1y_2 + x_2y_1).
\]
The trinomial $J$ obeys the differential relations of $\mathbb{O}(2)$ triality:
\[
 J_{x_1}^2 + J_{x_2}^2  = (y_1^2 + y_2^2)( z_1^2 + z_2^2),\qquad  J_{y_1}^2 + J_{y_2}^2 = (z_1^2 + z_2^2)(x_1^2 + x_2^2),\qquad  J_{z_1}^2 + J_{z_2}^2 = (x_1^2 + x_2^2)( y_1^2 + y_2^2).
\]
If we introduce the complex numbers
\[
 x=x_1+ix_2,\qquad y=y_1+iy_2,\qquad z=z_1-iz_2,
\]
where $i^2=-1$, then $J=\Re(xyz)$.

\begin{itemize}
\item \textbf{The $\mathbb{O}(1,1)$ triality.}
\end{itemize}
Suppressing the variables $x_2,x_7$, $y_2,y_7$, and $z_2,z_7$ in $K$ gives the trinomial $L$:
\[
L = z_1(x_1y_1 +x_8y_8) +z_8(x_1y_8+ x_8y_1).
\]
The trinomial $L$ obeys the differential relations of $\mathbb{O}(1,1)$ triality:
\[
L_{x_1}^2 - L_{x_8}^2  = (y_1^2 - y_8^2)( z_1^2 - z_8^2),
\qquad
L_{y_1}^2 - L_{y_8}^2 = (z_1^2 - z_8^2)(x_1^2 - x_8^2),
\qquad
L_{z_1}^2 - L_{z_8}^2 = (x_1^2 - x_8^2)( y_1^2 - y_8^2).
\]

\begin{itemize}
\item \textbf{The $\mathbb{O}(1)$ triality.}
\end{itemize}
Finally, suppressing the variables $x_2$, $y_2$, and $z_2$ in $J$ gives the trinomial $M$:
\[
M=x_1y_1z_1.
\]
This obeys the differential relations of the $\mathbb{O}(1)$ triality:
\[
 M_{x_1}^2=y_1^2 z_1^2,\qquad M_{y_1}^2=z_1^2 x_1^2,\qquad M_{z_1}^2=x_1^2 y_1^2.
\]

\medskip

These models provide the concrete trialities used throughout the sequel. \cite{SpringerVeldkamp2000,Baez2002,Elduque2000}  Later arguments will depend only on the abstract identities satisfied by a triality, but it is helpful to keep in mind that all of the cases relevant to the magic square arise from these explicit compact and split models and their restrictions.

\section{Tensor diagrams and triality}
 We shall do our calculations using the abstract tensor diagrams of Blissard, James Joseph Sylvester, William Kingdon Clifford and Roger Penrose. \cite{Sylvester1878,Clifford1878,Penrose1971,Cvitanovic2008}  For us this method has a number of crucial advantages:
\begin{itemize}\item   The nature of a tensor can be seen at a glance.
\item For each calculation, it is trivial to see how to do the equivalent calculation with the vector spaces of the triality appropriately permuted.
\item The calculations do not use elements of the vector spaces of the triality, so do not require that there really be a vector space interpretation (or even a background field, although we do multiply and divide by powers of two).  \item So, for example, one could investigate the triality in other contexts, such as symmetric monoidal tensor categories, or in non-standard dimensions.  The calculations would go through.
\end{itemize}  A vector is equipped with a line representing its "index" or "argument"; a tensor with multiple lines, one for each index.  Contraction between "indices" is represented by joining the relevant lines: in particular a tensor with no non-contracted lines is a scalar.  The metric is simply represented by the line itself, as is the Kronecker delta.  Since the metric is an intrinsic part of the structure, we do not need to distinguish between  covariant and contravariant indices.   We make the convention that index ordering is only important amongst tensors of the same kind.   To distinguish the three vector space genres, we will need three kinds of lines:  here we use three different colors for the lines.

\section{Relations obeyed by a triality}
We prove various formulas obeyed by the triality operator.
\begin{itemize}\item 
The metric is symmetric and has trace the dimension $n$ (regarded as an element of the underlying field, not as an integer):
\[  \metric = \metricswitch,    \hspace{20pt} \tracemetric=\hspace{2pt} n.\]
\item The triality operator may be written as follows:
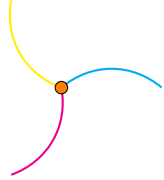
\[  \tri\]
\item The basic Clifford algebra identity obeyed by the triality is:
\[  \trialityLHS \hspace{3pt} + \hspace{3pt} \trialityRHSb\hspace{3pt}= \hspace{3pt} 2\hspace{3pt}\trialityRHSa\]
Note that this identity is written in a way does not explicitly require the existence of elements of the vector spaces.    Also, it is understood here and in the following that any such formula is valid if the colors of the edges are appropriately permuted.  
\item Equivalently, we may rewrite this identity as follows:
\[ \trialityLHS\hspace{3pt} = \hspace{3pt} \trialityRHSa\hspace{3pt} +\hspace{3pt} \hspace{3pt}\blackboxld \]
Here the box tensor is skew in its upper and lower arguments and is given in terms of the triality operator by:
\[  2 \hspace{3pt} \blackboxld \hspace{3pt}= \hspace{3pt} \trialityLHS \hspace{3pt} - \hspace{3pt} \trialityRHSb\hspace{3pt}\]
\item
Taking a trace of the basic triality identity gives the relation:
\[ \twotrialitiestraced  \hspace{3pt} = \hspace{3pt}  n \hspace{3pt}  \identityl\]
\item Taking the trace of this relation gives the complete contraction of the triality operator with itself:
\[ \threetrialitiestraced \quad  = \quad  n^2\]
\item Next we have the triangle relation:
\[ \triangleaa\  =\   - \   \trianglebb\   + \  2\   \trianglecc\qquad   =\   (2 - n) \qquad \triangledd \]
\item If we glue four trialities together, we arrive at the first double box identity, which we call a diagonal double box identity, since the free upper and lower tensor arguments belong to the same vector space:
\[ \diagboxa \quad  + \quad  \diagboxb\quad   = \quad  2n \quad  \Bigg( \quad   \diagboxc \quad  - \quad  \diagboxd \quad   + \quad   \diagboxe  \  \Bigg)  \]
\begin{proof}
	\[   \quad  \diagboxb\quad   = \quad   - \quad  \diagboxf \quad   + 2\quad   \diagboxg    =  \quad  \diagboxh \quad   - \quad  2\quad  \diagboxj \quad   + \quad  2n\quad  \diagboxc \quad   \]
\[  = \quad  \diagboxl \quad  +  \quad  2n \quad \diagboxc \quad  - \quad  2n\quad  \diagboxd  \quad  =\quad  -\quad  \diagboxb  + \quad  2\quad  \diagboxq \quad +  \quad  2n \quad \diagboxc \quad  - \quad  2n\quad  \diagboxd  \quad   \]  
\[  =\quad  -\quad  \diagboxb \quad  + \quad  2n \quad  \Bigg( \quad   \diagboxc \quad  - \quad  \diagboxd \quad   + \quad   \diagboxe  \  \Bigg)  \]
\end{proof}
\item Taking the skew parts of both sides of the diagonal double box identity, we deduce the relation:
\[  \quad \blackboxlda \quad  + \quad \quad  \blackboxldb \   = \quad  2n \quad \Bigg( \  \blackboxldc  \   - \  \blackboxldd\   \Bigg) \]
\item Next we deduce the corresponding off-diagonal double box identities.
We start with two copies of the basic triality identity:
\[ \trialityLHS\hspace{3pt} = \hspace{3pt} \trialityRHSa\hspace{3pt} +\hspace{3pt} \hspace{3pt}\blackboxld\]
\[ \trialityLHSr\hspace{3pt} = \hspace{3pt} \trialityRHSar\hspace{3pt} +\hspace{3pt} \hspace{3pt}\blackboxldr\]
We glue these together to give the following relation:
\[  \offdiagc \quad   =\quad   n\quad   \offdiagd\quad   +\quad   \quad   \blackboxldaa  \]
\[ = \quad   - \quad   \offdiagf \quad  +\quad   2\quad   \offdiagg \quad   = \quad   (n -2) \quad   \offdiagh \quad  + \quad 2\quad   \begin{tikzpicture}[baseline = -2pt]
\draw[solid, thick, yellow] (0,0) -- +(-0.5, -0.75);
\draw[solid, thick, yellow]  (1/2,0) -- +(0.5, -0.75);
\draw[solid, thick, magenta] (0,0) -- +(1, 0.75);
\draw[solid, thick, magenta] (1/2,0) -- +(-1,0.75);
\draw[thick, cyan] (0,0) -- (1/2,0);
\node[triality] at (0,0) {};
\node[triality] at (1/2,0) {};
\end{tikzpicture}\]
\[ =  \quad   n \quad   \offdiagd \quad  +   \quad  (n -4)  \quad  \offdiagk  \quad  \]
In particular we get the reproducing formula for the off-diagonal boxes:
\[  \quad  \blackboxldaa\quad  = \quad  (n - 4) \quad \offdiagk \]
\item Next we use the off-diagonal double box identity (twice) to prove that generically (specifically, when the quantity $(n - 4)$ is invertible) the two diagonal double boxes are equal and therefore (since their sum is known and given above), each is determined:
\[ (n - 4) \quad  \blackboxlda \quad  = \quad  \blackboxldab \quad  = \quad (n - 4) \quad  \blackboxldb \quad  = \quad  n(n -4) \quad \Bigg( \  \blackboxldc  \quad  - \quad \blackboxldd\  \Bigg) \]
\item Finally, we have the important triality preservation formula:
\[\hspace{-50pt}\tripresa\qquad +\qquad \tripresb\qquad +\qquad  2\qquad \tripresc \qquad   = \qquad 0\]
This is proved as follows:
\[ 2\qquad  \tripresa\qquad  =\qquad   \tripresd\qquad \]
\[  = \qquad  - \qquad  \triprese \qquad  + 2\qquad   \tripresf \]
\[ =   \qquad  \tripresg \qquad  - 4\qquad   \tripresc \]
\[ = \qquad  - 2 \qquad  \tripresb \qquad  - 4\qquad   \tripresc \]
\end{itemize}

\section{Solving the triality constraint}
The triality constraint is, by definition, the tensor equation:
\[ 0  \qquad   =  \qquad  \tricona \qquad   + \qquad   \triconb  \qquad  +  \qquad  \triconc\]
Here the unknowns are skew:
\[ \quad  \triconca  \quad  = \quad   - \quad   \triconcb \quad , \quad    \triconcc  \quad  = \quad   - \quad   \triconcd \quad , \quad    \triconce  \quad  = \quad   - \quad   \triconcf \quad . \quad    \]
By the triality preservation formula proved above, one solution is as follows:
\[ 4 \triconca \quad  = \quad  \triconcebox\quad , \quad    4\triconcc \quad  = \quad  \triconceboxa\quad ,  \quad \textrm{with $\triconce$ an arbitrary skew tensor}.\] 
So to analyze the solution space it suffices to solve the triality constraint, with $\triconce$ taken to be zero, so it remains to understand the following equation:
\[    \qquad  \tricona \qquad   = -   \qquad  \triconc\]
Contract both sides of this equation with a triality operator, yielding the relation:
\[  \tricontra\quad  = \quad  -\quad  \tricontrc \quad , \quad  n \quad  \triconcc \quad  = \quad \triconcage \quad . \]
Now back substitute, to eliminate the tensor $\triconcc$, giving:
\[ - n \quad  \triconc  \quad  =   \quad \triconage \quad \]
Now multiply both sides by $(n -4)$ and contract again with a triality operator.   Then using the formula for the diagonal double box times $(n -4)$ we get:
\[   - n(n - 4) \quad  \triconccc = \quad  (n - 4) \quad  \triconagec \quad =  \quad  (n - 4) \quad  \triconagecc\quad , \]
\[ - n^2(n - 4) \quad \triconageccc =\quad   - 2n(n - 4)\quad  \triconageccc\quad , \quad n(n -2)(n - 4) \quad \triconca =\quad 0\hspace{3pt}.\]
We conclude that when the polynomial $n(n - 2)(n - 4)$ is invertible in $\mathbb{F}$, the triality constraint is solved by picking one skew tensor, say $ \triconca\hspace{2pt}, $ arbitrarily and then the other two skew tensors are uniquely determined.  Note, in particular, that this is the case when $n = 8$, provided that the characteristic of $\mathbb{F}$ is neither two nor three.

For the purposes of the Magic Square considered over the real or complex number fields, we need to analyze further  the cases $n = 1$, $2$ and $4$ in characteristic zero. 
\begin{itemize}\item 
In dimension one, there are no non-zero skew tensors, so no solutions.  
\item  In dimension two, we have the factorization of the skew symmetrizer:
\end{itemize}
\[   0 \   \ne \   \symsym \   - \   \symsyma \   =   \   \epsinveps \  , \]
\[ \eps\   =\   -\   \epsr\qquad , \qquad  \epsin\   =\   -\   \epsinr\  , \   \epsepsin \   =  \   - \   \ids \  , \   \epsepsc \   = 2. \vspace{5pt}\]
Define scalars $\alpha, \beta$ and $\gamma$ as follows:
\[ \alpha \quad  = \quad \offdiagkepsa\quad  =  \quad \offdiagkepsb\quad   , \quad \beta \quad  = \quad \offdiagkepsc\quad  =  \quad \offdiagkepsd\quad   ,  \quad \gamma \quad  = \quad \offdiagkepse\quad  =  \quad \offdiagkepsf\quad   . \]
Then the off-diagonal double box identities  give the relations:
\[ \alpha\beta = - 4\gamma, \hspace{3pt} \gamma\alpha = - 4\beta, \hspace{3pt} \beta\gamma = - 4\alpha.\]
Also the diagonal double box identities give the relations:
\[ \alpha^2 = \beta^2 = \gamma^2 = 16.\]
So we have:
\[ \alpha = \beta = \gamma = - 4.\]
Then the three permutations of the triality preservation formula immediately reduce to the following three relations:
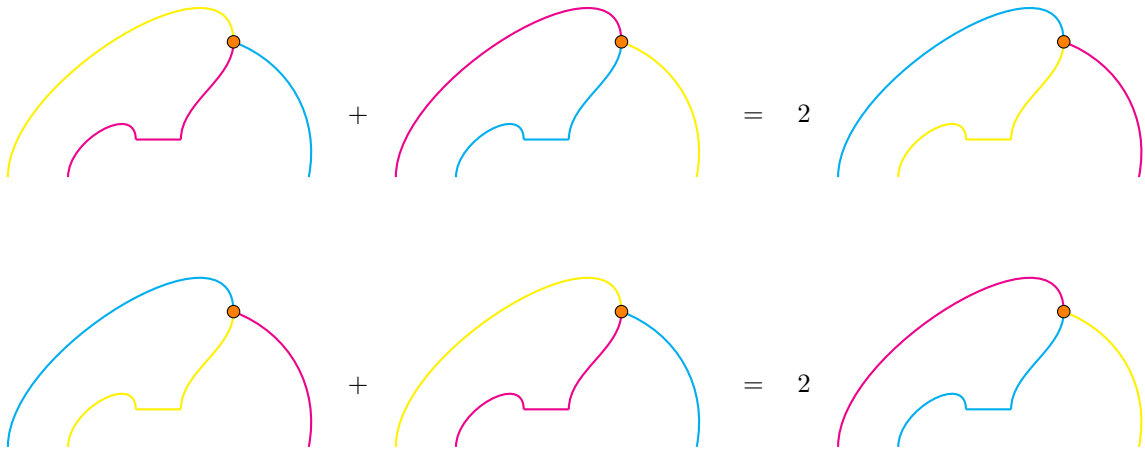
\[ \hspace{-30pt}\tripresx  \quad +  \quad \tripresz \quad   =  \quad 2  \quad \tripresy  \]
\[ \hspace{-30pt}\tripresy \quad  + \quad  \tripresx \quad   = \quad  2 \quad  \tripresz  \]
\[\hspace{-30pt} \tripresz  \quad +  \quad \tripresy \quad   = \quad  2 \quad  \tripresx  \]
Thus we have
\[ \hspace{-30pt} \tripresx  \quad = \quad   \tripresz \quad   = \quad   \tripresy  \hspace{3pt}.\]
Using this relation the triality constraint equation just becomes:
\[   \quad  \triconca  \quad  = \quad   a \quad    \epsa \quad , \quad    \triconcc  \quad  = \quad   b \quad    \epsb \quad , \quad    \triconce  \quad  = \quad   c \quad    \epsc\quad ,   \]
\[ a + b + c  = 0.\]
When $a  + b + c = 0$,  the triality constraint equation is satisfied, as required, so in the two dimensional case, two of the three skew tensors may be chosen  arbitrarily and the third is then determined, so the solution space has two degrees of freedom.
\begin{itemize}\item In the case $n = 4$,   we need to solve the equation given above:
\end{itemize}
\[ - 4 \quad  \triconc  \quad  =   \quad \triconage \quad \]
Contract both sides with a triality operator, giving the formula:
\[   - 4 \quad  \triconccc = \quad   \quad  \triconagec \quad =  \quad   \quad  \triconagecc\quad . \]
This gives the relation:
\[ 16\quad   \triconca \quad  = \quad \triconagedd\quad .\]

When $n = 4$, different off-diagonal boxes annihilate each other:
\[  \quad  \blackboxldaa\quad  = \quad  0. \]
Then the pairs of diagonal double boxes form mutually orthogonal projection operators on the space of skew tensors:
\[  \quad \blackboxlda \quad  + \quad \quad  \blackboxldb \   = \quad  8 \quad \Bigg( \  \blackboxldc  \   - \  \blackboxldd\   \Bigg) \]
\[  \quad  \blackboxldaaa \quad  = \quad  16 \quad  \blackboxlda\quad , \quad  \blackboxldaba \quad   = \quad  0\quad .\]
The images of these projection operators split the space of skew tensors canonically into two three-dimensional sub-spaces.
So $\triconca$ lies in one of these spaces and when it does so, via the properties of the projection operators, all conditions are satisfied.  So the solution space of the triality constraint has nine degrees of freedom, six from one of the skew tensors, which is arbitrary and three from one of the other two skew tensors which is required to be an eigenstate of one of the projection operators.  

Summarizing the case  of characteristic zero, we have:
\begin{itemize} \item If $n = 0$, or $n =1$, all three skew tensors vanish, so there is nothing to do and there are no degrees of freedom in the solutions.
\item In the case $n = 2$, two of the three skew tensors (each with one degree of freedom) may be prescribed and then the other one is determined, so there are two degrees of freedom in all.
\item In the case $n = 4$, one skew tensor is arbitrary, for six degrees of freedom.  Then modulo that tensor, one other is self-dual and then the third is completely determined, giving three additional degrees of freedom and a total of nine degrees of freedom.   
\item  In the case $n = 8$, the solution space is that of one skew tensor, which is arbitrary and then the other two are determined, giving twenty-eight degrees of freedom.  
\end{itemize}
So, if we denote by $\tau_n$ the number of degrees of freedom in solving the triality relation, we have the following data, in characteristic zero:
\begin{itemize}\item $\tau_0 = \tau_1 = 0, \qquad  \tau_2 = 2,  \qquad  \tau_4 = 9,  \qquad  \tau_8 = 28$.
\end{itemize}

\section{The Theorem of Diophantus, Brahmagupta, Euler, Degen and Hurwitz}
 In this section we give a direct proof of the theorem, due particularly to Carl Ferdinand Degen and Adolf Hurwitz, building on the work of Diophantus of Alexandria, Brahmagupta and Leonhard Euler. \cite{Diophantus1670,Brahmagupta1817,Euler1760,Degen1822,Hurwitz1898,Hurwitz1922} On the one hand trialities do exist, over any field, if the vector space dimension $n$ is the integer $0, 1, 2, 4$ or $8$ and on the other hand no trialities exist unless the polynomial $n^2(n - 1)(n - 2)(n  -4)(n - 8)$ vanishes in $\mathbb{F}$ (assumed to be not of characteristic two). \cite{Hurwitz1898,Hurwitz1922}  The case of dimension zero is trivial: the vector spaces of the triality each consist only of a zero vector and the metric and triality tensors are identically zero.  The other cases of integer dimensions $1$, $2$, $4$ and $8$ are realized by the trinomials given explicitly above (and valid over any field).   It remains to prove the non-existence statement of the theorem.   

First we recall that, as proved above, when they are multiplied by $(n -4)$, the two diagonal double boxes are equal to each other and to a multiple of the identity operator on skew tensors:
\[ (n - 4) \quad  \blackboxlda \quad  =  \quad (n - 4) \quad  \blackboxldb \quad  = \quad  n(n -4) \quad \Bigg( \  \blackboxldc  \quad  - \quad \blackboxldd\  \Bigg)\hspace{3pt}. \]
So now we have, using this relation and the off-diagonal double box identity (twice), the following formula for the square of the diagonal double box, multiplied by $(n -4)$:
\[ (n - 4)\quad  \blackboxldaaa \quad  = \quad (n-4)\quad  \blackboxldaba \quad  = \quad  (n - 4)^2 \quad  \blackboxldab \quad  = \quad (n - 4)^3 \quad  \blackboxlda \]
Now we simplify both sides of this equation to multiples of the identity operator on skew tensors:
\[ 2n^2(n - 4) \quad \Bigg( \  \blackboxldc  \quad  - \quad \blackboxldd\  \Bigg) \quad  = \quad  n(n - 4)^3 \quad \Bigg( \  \blackboxldc  \quad  - \quad \blackboxldd\  \Bigg) \]
Collecting terms, we get:
\[  \quad n(n -2)(n - 4)(n - 8)  \quad \Bigg( \  \blackboxldc  \quad  - \quad \blackboxldd\  \Bigg)  \quad  =\quad   0\]
Taking the double trace, we get the required result, since the double trace of the Kronecker delta terms gives a factor of $n(n -1)$:
\[ 0 = n^2(n -1)(n - 2)(n - 4)(n - 8).\]

\section{The embedding of the orthogonal algebra; the triality sub-algebras}

Let
\[
\mathbb V=\mathbb V_s\oplus \mathbb V_t\oplus \mathbb V_u
\]
be a triality of dimension $n$, with $s,t,u\in \mathbb S$ distinct.  Each space $\mathbb V_r$ carries its given non-degenerate quadratic form, and hence its orthogonal group and orthogonal Lie algebra.  The direct sum
\[
\mathfrak{so}(\mathbb V_s)\oplus \mathfrak{so}(\mathbb V_t)\oplus \mathfrak{so}(\mathbb V_u)
\]
acts naturally on $\mathbb V$ by endomorphisms preserving the three summands.  We now isolate the subalgebra of this direct sum that preserves the triality tensor.

\begin{definition}
The \emph{triality group} $\mathcal T$ is the group of orthogonal transformations of $\mathbb V$ that preserve each of the three spaces $\mathbb V_s$, $\mathbb V_t$ and $\mathbb V_u$, and preserve the triality form.  Thus an element of $\mathcal T$ is a triple
\[
(g_s,g_t,g_u)\in O(\mathbb V_s)\times O(\mathbb V_t)\times O(\mathbb V_u)
\]
with the property that
\[
\tau(g_sx,g_ty,g_uz)=\tau(x,y,z)
\]
for all $x\in \mathbb V_s$, $y\in \mathbb V_t$ and $z\in \mathbb V_u$.

The \emph{triality algebra} $\mathfrak T$ is the Lie algebra of $\mathcal T$.  Equivalently, $\mathfrak T$ consists of all triples
\[
(p,q,r)\in \mathfrak{so}(\mathbb V_s)\oplus \mathfrak{so}(\mathbb V_t)\oplus \mathfrak{so}(\mathbb V_u)
\]
such that
\[
\tau(px,y,z)+\tau(x,qy,z)+\tau(x,y,rz)=0
\]
for all $x\in \mathbb V_s$, $y\in \mathbb V_t$ and $z\in \mathbb V_u$.
\end{definition}

The defining linear relation for $\mathfrak T$ is exactly the triality constraint solved in the previous section.  The point of the present section is therefore not to introduce a new calculation, but to reinterpret that solution space as a Lie subalgebra of the ambient orthogonal algebra.

\begin{proposition}
The triality algebra $\mathfrak T$ is canonically identified with the solution space of the triality constraint considered in the previous section.
\end{proposition}

\begin{proof}
If $(p,q,r)$ is a triple of skew endomorphisms preserving the three summands, then differentiating the identity
\[
\tau(g_sx,g_ty,g_uz)=\tau(x,y,z)
\]
at the identity gives precisely the linear relation
\[
\tau(px,y,z)+\tau(x,qy,z)+\tau(x,y,rz)=0.
\]
Conversely, any triple of skew endomorphisms satisfying this relation is, by definition, an infinitesimal symmetry of the triality tensor.  Thus the Lie algebra of the triality group is exactly the solution space of the triality constraint.
\end{proof}

The previous section therefore gives a complete description of $\mathfrak T$ in the dimensions relevant to the magic square.  We now record the result in Lie-theoretic form.

\begin{theorem}
Assume that the ground field has characteristic zero.  For a triality of dimension $n\in\{1,2,4,8\}$, the triality algebra $\mathfrak T$ is as follows.
\begin{itemize}
\item If $n=1$, then $\mathfrak T=0$.

\item If $n=2$, then each algebra $\mathfrak{so}(\mathbb V_r)$ is one-dimensional.  After identifying these three lines with $\mathbb F$, the triality algebra is
\[
\mathfrak T\cong \{(a,b,c)\in \mathbb F^3: a+b+c=0\}.
\]
In particular $\mathfrak T$ is abelian of dimension $2$.

\item If $n=4$, then each algebra $\mathfrak{so}(\mathbb V_r)$ splits into two simple three-dimensional summands,
\[
\mathfrak{so}(4)\cong \mathfrak{so}(3)_+\oplus \mathfrak{so}(3)_-.
\]
The triality constraint annihilates one chirality and leaves the other free; equivalently, the triality algebra is the direct sum of one three-dimensional summand from each factor.  In particular,
\[
\dim \mathfrak T=9.
\]

\item If $n=8$, then projection onto any one factor gives an isomorphism
\[
\mathfrak T\xrightarrow{\ \sim\ }\mathfrak{so}(\mathbb V_r),
\]
so that $\mathfrak T\cong \mathfrak{so}(8)$ and therefore
\[
\dim \mathfrak T=28.
\]
\end{itemize}
\end{theorem}

\begin{proof}
The case $n=1$ is immediate, since there are no non-zero skew endomorphisms.

For $n=2$, the previous section showed that, once a generator of each one-dimensional orthogonal algebra is chosen, an arbitrary solution of the triality constraint has the form
\[
p=a\,\varepsilon_s,\qquad q=b\,\varepsilon_t,\qquad r=c\,\varepsilon_u,
\]
with the single linear relation $a+b+c=0$.  This proves the stated description.

For $n=4$, the previous section showed that the space of solutions of the triality constraint has dimension $9$, and that this solution space is obtained by the canonical splitting of the skew tensors into the two three-dimensional eigenspaces associated with the diagonal double-box operators.  In Lie-algebraic language this is precisely the decomposition
\[
\mathfrak{so}(4)\cong \mathfrak{so}(3)_+\oplus \mathfrak{so}(3)_-,
\]
and the triality condition selects one of the two summands in each factor.  Hence the triality algebra has dimension $9$.

For $n=8$, the generic part of the solution of the triality constraint applies: one skew endomorphism may be prescribed arbitrarily, and the other two are then uniquely determined.  Projection onto any one of the three factors is therefore an isomorphism from $\mathfrak T$ to the corresponding orthogonal algebra.  Since $\dim \mathfrak{so}(8)=28$, the conclusion follows.
\end{proof}

It is convenient to denote by $\tau_n$ the dimension of the triality algebra in dimension $n$.  Thus, in characteristic zero, we have
\[
\tau_0=0,\qquad \tau_1=0,\qquad \tau_2=2,\qquad \tau_4=9,\qquad \tau_8=28.
\]
These are the numerical data required for the construction of the Lie algebras associated to pairs of trialities in the next section.

\section{The Lie algebra associated to a pair of trialities}

Let us now pass from a single triality to a pair of trialities.  We assume given \emph{two} trialities, one governing the upper indices and one the lower.  Since these two sets of indices will always be kept separate, we do not need different notations for the two index triples.  Let the vector spaces of the upper triality be $\mathbb V_1,\mathbb V_2,\mathbb V_3$, each of dimension $n$, and let those of the lower triality be $\mathbb V'_1,\mathbb V'_2,\mathbb V'_3$, each of dimension $n'$.  The Lie algebra will contain three families of tensor generators, lying in
\[
\mathbb V_1\otimes \mathbb V'_2,\qquad \mathbb V_2\otimes \mathbb V'_3,\qquad \mathbb V_3\otimes \mathbb V'_1,
\]
so these contribute a total dimension of $3nn'$.  These generators are supplemented by the triality-preserving orthogonal algebras attached to the two trialities.  By the previous section, these have dimensions $\tau_n$ and $\tau_{n'}$, so the total dimension is
\[
\dim \mathfrak g(\gamma,\gamma') = 3nn' + \tau_n + \tau_{n'}.
\]
For the classical triality dimensions $n,n'\in\{1,2,4,8\}$ this gives the symmetric table
\[
\begin{array}{|c|cccc|}
 n/n'&\underline{1}&\underline{2}&\underline{4}&\underline{8}\\
 \underline{1}&3&8&21&52\\
 \underline{2}&8&16&35&78\\
 \underline{4}&21&35&66&133\\
 \underline{8}&52&78&133&248
\end{array}\hspace{3pt},
\]
which already anticipates the Freudenthal magic square.  Interchanging $n$ and $n'$ simply interchanges the two trialities and yields an isomorphic Lie algebra.

We now describe the bracket explicitly in the diagrammatic notation.  The remainder of this section gives the construction itself: first the triality-preserving orthogonal generators and tensor generators, then the bracket relations between them, and finally the triality constraint required for consistency.

To describe the Lie algebra, we slightly modify our notation, in order that the description of triality invariance be symmetrically handled:  henceforth, we never use  the diagonal box and the off-diagonal boxes are  rescaled and rewritten as follows:
\[    
 \hspace{3pt}.\] 
We instead define the idempotent skew-symmetrizer as a "diagonal box":
\[  %
\hspace{4pt}. \]
The triality preservation formula is now the symmetrical expression:
\[ \hspace{-30pt} %
\]
 We begin by listing the Lie algebra generators here regarded as  forming a Poisson algebra and we write out the defining Poisson brackets.
\begin{itemize}\item  There are six orthogonal group generators,  two sets of three skew tensors:
\end{itemize}
\begin{equation*}
%
=0.\]
\begin{itemize} \item Then there are three sets of tensor generators:
\end{itemize}
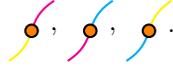
\[  %
\]
\begin{itemize}\item The Poisson brackets of the upper and lower orthogonal generators vanish:
\end{itemize} 
\[ 
\Bigg[%
\Bigg] = 0.\qquad \]
\begin{itemize} \item The diagonal brackets of the upper and lower orthogonal generators amongst themselves are straightforward, giving the standard orthogonal algebras:
\end{itemize}
\[ \hspace{-20pt}
\Bigg[%
\hspace{3pt}.
\]
Here the quantity $\kappa$ is a non-zero constant, to be determined later.  

\begin{itemize} \item The off-diagonal commutators between the orthogonal generators may be summarized with just one Poisson bracket:
\end{itemize}
\[ 
\Bigg[%
.
\]
Here and in the following, once one commutator is written,  all appropriate colour permutations are to be understood to be written also.  Similarly, once one commutator has been written, involving one or more orthogonal generators, then its reflection in the horizontal is understood also. 
\begin{itemize}\item Next we have the commutator of an orthogonal group generator with a tensor generator:\end{itemize}
\[ \Bigg[%
     \  .
 \]
Note that the lower index of the tensor generator is inert with respect to the upper orthogonal generator (and vice-versa).  Note also that the second line is determined from the first by the requirement  of skew symmetry of the algebra.
\begin{itemize}\item Next we have the commutator of a tensor generator with itself:
\end{itemize}
\[
\Bigg[\hspace{2pt}%
\]
\begin{itemize}\item Next we have the commutator of two different tensor generators:
\end{itemize}
\[
\Bigg[\hspace{2pt}%
\]
This commutator is \emph{sign sensitive} and is invariant under cyclic permutations (only) of the colours, in the order magenta, cyan, yellow, so we have also:
\[
\Bigg[\hspace{2pt}%
\  .
\]
\begin{itemize}\item Finally, we have the triality constraint:
\end{itemize}
\[ 0 \qquad   =  \qquad  \triconas \qquad   + \qquad   \triconbs  \qquad  +  \qquad  \triconcs\]

\section{The skew-symmetry of the algebra}
Our first tasks will be to establish the skewness of the algebra and that the triality constraint is consistent with the various commutators i.e. that the commutator of the left hand-side of the triality constraint  with any of the generators is zero.  The Poisson algebra is manifestly skew-symmetric by its very definition, except for the off-diagonal commutators between two orthogonal group generators.  For these we need to prove the identity:
\[ 
0 = \Bigg[\begin{tikzpicture}[baseline = - 2pt]
  \begin{pgfonlayer}{foreground}
\node[blobA, thick, draw] (blob1) at (0,-1/4) {};
\end{pgfonlayer}
  \coordinate (foot1) at (-1/4,1/4);
  \coordinate (foot2) at (1/4,1/4);
  \draw[A](node cs:name=blob1,anchor=north west) -- (node cs:name=foot1);
  \draw[A](node cs:name=blob1,anchor=north east) -- (node cs:name=foot2);
\end{tikzpicture}\quad,\quad
\begin{tikzpicture}[baseline = - 2pt]
  \begin{pgfonlayer}{foreground}
\node[blobB, thick, draw] (blob2) at (0,-1/4) {};
\end{pgfonlayer}
  \coordinate (foot1) at (-1/4,1/4);
  \coordinate (foot2) at (1/4,1/4);
  \draw[B](node cs:name=blob2,anchor=north west) -- (node cs:name=foot1);
  \draw[B](node cs:name=blob2,anchor=north east) -- (node cs:name=foot2);
\end{tikzpicture}\Bigg] \  + \   
\Bigg[\begin{tikzpicture}[baseline = - 2pt]
  \begin{pgfonlayer}{foreground}
\node[blobB, thick, draw] (blob1) at (0,-1/4) {};
\end{pgfonlayer}
  \coordinate (foot1) at (-1/4,1/4);
  \coordinate (foot2) at (1/4,1/4);
  \draw[B](node cs:name=blob1,anchor=north west) -- (node cs:name=foot1);
  \draw[B](node cs:name=blob1,anchor=north east) -- (node cs:name=foot2);
\end{tikzpicture}\quad,\quad
\begin{tikzpicture}[baseline = - 2pt]
  \begin{pgfonlayer}{foreground}
\node[blobA, thick, draw] (blob2) at (0,-1/4) {};
\end{pgfonlayer}
  \coordinate (foot1) at (-1/4,1/4);
  \coordinate (foot2) at (1/4,1/4);
  \draw[A](node cs:name=blob2,anchor=north west) -- (node cs:name=foot1);
  \draw[A](node cs:name=blob2,anchor=north east) -- (node cs:name=foot2);
\end{tikzpicture}
\Bigg] \   =  \   
\begin{tikzpicture}[baseline = - 2pt]
\node at (-0.8, 0) {$2\kappa$}; 
  \begin{pgfonlayer}{foreground}
\node[blobB, thick, draw] (blob3) at (0,-1/4) {};
\end{pgfonlayer}
  \coordinate (A) at (-1/4,1/2);
  \coordinate (B) at (1/4,1/2);
    \coordinate (C) at (-0.4,0.25);
      \coordinate (D) at (0.4,0.25);
     \coordinate (E) at (-0.2,0.25);
      \coordinate (F) at (0.2,0.25);     
  \coordinate (G) at (-1/16,1/2);
  \coordinate (H) at (1/16,1/2);
  \coordinate (I) at (-0.2,0.15);
  \coordinate (J) at (0.2,0.15);  
  \path (I) edge[B, thick, bend right = 90] (J);
    \path (C) edge[B,bend right = 20] (blob3.north west);
  \path (D) edge[B,bend left = 20] (blob3.north east);
\path (C) edge[A, thick, bend left = 20] (A);
\path (D) edge[B, thick, bend right = 20] (B);
\path (E) edge[A, thick, bend left = 20] (G);
\path (F) edge[B, thick, bend right = 20] (H);
    \node[blob, thick, draw] (blobl) at (-0.3,0.25) {};
  \node[blob, thick, draw] (blobr) at (0.3,0.25) {};
\end{tikzpicture}
\   +  \  
\begin{tikzpicture}[baseline = - 2pt]
\node at (-0.8, 0) {$2\kappa$}; 
  \begin{pgfonlayer}{foreground}
\node[blobA, thick, draw] (blob3) at (0,-1/4) {};
\end{pgfonlayer}
  \coordinate (A) at (-1/4,1/2);
  \coordinate (B) at (1/4,1/2);
    \coordinate (C) at (-0.4,0.25);
      \coordinate (D) at (0.4,0.25);
     \coordinate (E) at (-0.2,0.25);
      \coordinate (F) at (0.2,0.25);     
  \coordinate (G) at (-1/16,1/2);
  \coordinate (H) at (1/16,1/2);
  \coordinate (I) at (-0.2,0.15);
  \coordinate (J) at (0.2,0.15);  
  \path (I) edge[A, thick, bend right = 90] (J);
    \path (C) edge[A,bend right = 20] (blob3.north west);
  \path (D) edge[A,bend left = 20] (blob3.north east);
\path (C) edge[B, thick, bend left = 20] (A);
\path (D) edge[A, thick, bend right = 20] (B);
\path (E) edge[B, thick, bend left = 20] (G);
\path (F) edge[A, thick, bend right = 20] (H);
    \node[blob, thick, draw] (blobl) at (-0.3,0.25) {};
  \node[blob, thick, draw] (blobr) at (0.3,0.25) {};
\end{tikzpicture}  \hspace{3pt}.
\]
We use triality invariance to simplify:
\[ 16 \   \begin{tikzpicture}[baseline = - 2pt]
  \begin{pgfonlayer}{foreground}
\node[blobB, thick, draw] (blob3) at (0,-1/4) {};
\end{pgfonlayer}
  \coordinate (A) at (-1/4,1/2);
  \coordinate (B) at (1/4,1/2);
    \coordinate (C) at (-0.4,0.25);
      \coordinate (D) at (0.4,0.25);
     \coordinate (E) at (-0.2,0.25);
      \coordinate (F) at (0.2,0.25);     
  \coordinate (G) at (-1/16,1/2);
  \coordinate (H) at (1/16,1/2);
  \coordinate (I) at (-0.2,0.15);
  \coordinate (J) at (0.2,0.15);  
  \path (I) edge[B, thick, bend right = 90] (J);
    \path (C) edge[B,bend right = 20] (blob3.north west);
  \path (D) edge[B,bend left = 20] (blob3.north east);
\path (C) edge[A, thick, bend left = 20] (A);
\path (D) edge[B, thick, bend right = 20] (B);
\path (E) edge[A, thick, bend left = 20] (G);
\path (F) edge[B, thick, bend right = 20] (H);
    \node[blob, thick, draw] (blobl) at (-0.3,0.25) {};
  \node[blob, thick, draw] (blobr) at (0.3,0.25) {};
\end{tikzpicture}
\   +  \  16 \   
\begin{tikzpicture}[baseline = - 2pt]
  \begin{pgfonlayer}{foreground}
\node[blobA, thick, draw] (blob3) at (0,-1/4) {};
\end{pgfonlayer}
  \coordinate (A) at (-1/4,1/2);
  \coordinate (B) at (1/4,1/2);
    \coordinate (C) at (-0.4,0.25);
      \coordinate (D) at (0.4,0.25);
     \coordinate (E) at (-0.2,0.25);
      \coordinate (F) at (0.2,0.25);     
  \coordinate (G) at (-1/16,1/2);
  \coordinate (H) at (1/16,1/2);
  \coordinate (I) at (-0.2,0.15);
  \coordinate (J) at (0.2,0.15);  
  \path (I) edge[A, thick, bend right = 90] (J);
    \path (C) edge[A,bend right = 20] (blob3.north west);
  \path (D) edge[A,bend left = 20] (blob3.north east);
\path (C) edge[B, thick, bend left = 20] (A);
\path (D) edge[A, thick, bend right = 20] (B);
\path (E) edge[B, thick, bend left = 20] (G);
\path (F) edge[A, thick, bend right = 20] (H);
    \node[blob, thick, draw] (blobl) at (-0.3,0.25) {};
  \node[blob, thick, draw] (blobr) at (0.3,0.25) {};
\end{tikzpicture}   \   =\  
 \begin{tikzpicture}[baseline = - 2pt]
  \begin{pgfonlayer}{foreground}
\node[blobB, thick, draw] (blob3) at (0,-1/4) {};
\node[Triality] at (-0.45, 0.1){};
\node[Triality] at (-0.15, 0.1){};
\end{pgfonlayer}
    \path (-0.45, 0.1) edge[C] (-0.15, 0.1);
  \coordinate (A) at (0.25,1/2);
    \coordinate (A') at (0.125,1/2);
      \coordinate (A'') at (-0.25,1/2);
        \coordinate (A''') at (-1/8,1/2);
  \coordinate (B) at (-0.25,1/4);
    \coordinate (C) at (0.25,1/4);
  \coordinate (B') at (-0.35,1/4);
    \coordinate (C') at (0.35,1/4);
        \path (A''') edge[A,bend right=45] (-0.15, 0.1);
  \path (A'') edge[A,bend right=30] (-0.45, 0.1);
          \path (A') edge[B,bend left=45] (0.15, 0.1);
  \path (A) edge[B,bend left=30] (0.45, 0.1);
              \path (blob3.north east) edge[B,bend right=30] (0.45, 0.1);
                    \path (blob3.north west) edge[B,bend left=30] (-0.45, 0.1);
                      \path (0.15, 0.1) edge[B,bend left=45] (-0.15, 0.1);
  \draw[D](0.15, 0.3) -- (0.48,0.3);
    \draw[D](-0.15, 0.3) -- (-0.48,0.3);
\end{tikzpicture} 
\   +  \   \begin{tikzpicture}[baseline = - 2pt]
  \begin{pgfonlayer}{foreground}
\node[blobA, thick, draw] (blob3) at (0,-1/4) {};
\node[Triality] at (-0.45, 0.1){};
\node[Triality] at (-0.15, 0.1){};
\end{pgfonlayer}
    \path (-0.45, 0.1) edge[C] (-0.15, 0.1);
  \coordinate (A) at (0.25,1/2);
    \coordinate (A') at (0.125,1/2);
      \coordinate (A'') at (-0.25,1/2);
        \coordinate (A''') at (-1/8,1/2);
  \coordinate (B) at (-0.25,1/4);
    \coordinate (C) at (0.25,1/4);
  \coordinate (B') at (-0.35,1/4);
    \coordinate (C') at (0.35,1/4);
        \path (A''') edge[B,bend right=45] (-0.15, 0.1);
  \path (A'') edge[B,bend right=30] (-0.45, 0.1);
          \path (A') edge[A,bend left=45] (0.15, 0.1);
  \path (A) edge[A,bend left=30] (0.45, 0.1);
         \path (blob3.north east) edge[A,bend right=30] (0.45, 0.1);
                    \path (blob3.north west) edge[A,bend left=30] (-0.45, 0.1);
                      \path (0.15, 0.1) edge[A,bend left=45] (-0.15, 0.1);
  \draw[D](0.15, 0.3) -- (0.48,0.3);
    \draw[D](-0.15, 0.3) -- (-0.48,0.3);
\end{tikzpicture}   \]
\[ = \   
 \begin{tikzpicture}[baseline = - 2pt]
  \begin{pgfonlayer}{foreground}
\node[blobB, thick, draw] (blob3) at (0,-1/3) {};
\node[Triality] at (-0.45, 0.1){};
\node[Triality] at (-0.15, 0.1){};
\end{pgfonlayer}
    \path (-0.45, 0.1) edge[C] (-0.15, 0.1);
  \coordinate (A) at (0.25,1/2);
    \coordinate (A') at (0.125,1/2);
      \coordinate (A'') at (-0.25,1/2);
        \coordinate (A''') at (-1/8,1/2);
  \coordinate (B) at (-0.25,1/4);
    \coordinate (C) at (0.25,1/4);
  \coordinate (B') at (-0.35,1/4);
    \coordinate (C') at (0.35,1/4);
        \path (A''') edge[A,bend right=45] (-0.15, 0.1);
  \path (A'') edge[A,bend right=30] (-0.45, 0.1);
          \path (A') edge[B,bend left=45] (0.15, 0.1);
  \path (A) edge[B,bend left=30] (0.45, 0.1);
         \path (blob3.north east) edge[B,bend right=30] (0.45, 0.1);
                    \path (blob3.north west) edge[B,bend left=30] (-0.45, 0.1);
                      \path (0.15, 0.1) edge[B,bend left=45] (-0.15, 0.1);
  \draw[D](0.15, 0.3) -- (0.48,0.3);
    \draw[D](-0.15, 0.3) -- (-0.48,0.3);
\end{tikzpicture} 
\   + \   \begin{tikzpicture}[baseline = - 2pt]
  \begin{pgfonlayer}{foreground}
\node[blobB, thick, draw] (blob3) at (0,-1/3) {};
\node[Triality] at (-0.45, 0.1){};
\node[Triality] at (-0.15, 0.1){};
\end{pgfonlayer}
    \path (-0.45, 0.1) edge[C] (-0.15, 0.1);
  \coordinate (A) at (0.25,1/2);
    \coordinate (A') at (0.125,1/2);
      \coordinate (A'') at (-0.25,1/2);
        \coordinate (A''') at (-1/8,1/2);
  \coordinate (B) at (-0.25,1/3);
    \coordinate (C) at (0.25,1/3);
  \coordinate (B') at (-0.35,1/3);
    \coordinate (C') at (0.35,1/3);
        \path (A''') edge[B,bend right=45] (-0.15, 0.1);
          \path (A') edge[A,bend left=45] (0.15, 0.1);
                      \path (0.15, 0.1) edge[A,bend left=45] (-0.15, 0.1);
                         \draw[A](-0.45, 0.1) .. controls +(0.3, -0.5) and +(0.4, -0.3) .. (A);
                             \draw[B](blob3.north west) .. controls +(-0.6, 0.4) and +(-0.6, -0.5) .. (A'');
                                    \path (0.17, -0.26) edge[B,bend left=30] (-0.45, 0.1);
  \draw[D](0.16, 0.3) -- (0.41,0.3);
    \draw[D](-0.14, 0.3) -- (-0.57,0.3);
\end{tikzpicture} 
\   - \    
\begin{tikzpicture}[baseline = - 2pt]
  \begin{pgfonlayer}{foreground}
\node[blobC, thick, draw] (blob3) at (0,-1/3) {};
\node[Triality] at (-0.45, 0.1){};
\node[Triality] at (-0.15, 0.1){};
\end{pgfonlayer}
  \coordinate (A) at (0.25,1/2);
    \coordinate (A') at (0.125,1/2);
      \coordinate (A'') at (-0.25,1/2);
        \coordinate (A''') at (-1/8,1/2);
  \coordinate (B) at (-0.25,1/3);
    \coordinate (C) at (0.25,1/3);
  \coordinate (B') at (-0.35,1/3);
    \coordinate (C') at (0.35,1/3);
        \path (A''') edge[B,bend right=45] (-0.15, 0.1);
          \path (A') edge[A,bend left=45] (0.15, 0.1);
                    \path (-0.1, -1/3) edge[C,bend left=30] (-0.45, 0.1);
                          \path (0.2, -1/3) edge[C,bend left=30] (-0.15, 0.1);
                      \path (0.15, 0.1) edge[A,bend left=45] (-0.15, 0.1);
                          \path (-0.45, 0.1) edge[B,bend left=30] (A'');
                         \draw[A](-0.45, 0.1) .. controls +(0.3, -0.5) and +(0.4, -0.3) .. (A);
  \draw[D](0.16, 0.3) -- (0.4,0.3);
    \draw[D](-0.17, 0.3) -- (-0.49,0.3);
\end{tikzpicture} \  
  = \   0\hspace{3pt}.\]
Here  the last term vanishes identically, since there are two trialities skewed over all corresponding  indices and the other two terms cancel each other. 

\section{The consistency of the triality constraint}
We have to prove that the triality constraint commutes with each of the generators, so by color symmetry it suffices to prove that it commutes with one of the orthogonal generators and one of the tensor generators.
\begin{itemize}\item First we consider the commutator with an orthogonal generator.  We need the following commutator relation:
\end{itemize}
\[\hspace{-40pt}  0 \hspace{4pt} = \hspace{4pt} \Bigg[\qquad  \triconas \qquad   + \qquad   \triconbs  \qquad  +  \qquad  \triconcs \qquad , \qquad  \begin{tikzpicture}[baseline = - 2pt]
    \path (-0.1, 0) edge[A,bend right=30] (-0.4, 0.4);
        \path (0.1, 0) edge[A,bend left=30] (0.4, 0.4);
\node[blobA, thick, draw] (blob3) at (0,0) {};
\end{tikzpicture}\qquad \Bigg]\hspace{3pt}.\]
Computing the commutators, we see immediately that it is sufficient for this that the following algebraic condition on the triality operator holds:
\[ \hspace{-50pt} 0 \  = \begin{tikzpicture}[baseline = - 2pt]
     \draw[C, thick] (1,1) .. controls +(0.9,2) and +(0,1.3) ..  (-1.9, -0.4) ;
    \draw[A, thick] (1,1) .. controls +(0,1.3) and +(0,1.3) ..  (-1.6, -0.4) ;
 \node[triality] at (0, 0) {};
 \draw[B, thick] (1,1) .. controls +(0,-1/2) and +(0,1/2) ..  (0.1, 0.1) ;
  \draw[B, thick] (-1.2,-0.4) .. controls +(0,1/2) and +(0,1/2) ..  (-0.1, 0.1) ;
        \node[triality] at (1, 1) {};
  \begin{pgfonlayer}{foreground}
\node[blob] (blob3) at (0, 0) {};
\end{pgfonlayer}
  \coordinate (A) at (0.55,1/4);
    \coordinate (A') at (0.155,1/4);
      \coordinate (A'') at (0.05,1/4);
        \coordinate (A''') at (0.175,1/4);
  \coordinate (B) at (0.05,0);
    \coordinate (C) at (0.6,0);
  \coordinate (B') at (0,0);
    \coordinate (C') at (0.7,0);
                        \path (0.45, -0.45) edge[A,bend right=15] (1.8, 1.3);
                              \path (0.15, -0.45) edge[A,bend left=15] (-0.08, -0.1);
  \draw[A](1.6, 1.3) .. controls +(-0.5, -0.65) and +(0.3, -0.4) .. (0.09, -0.1);
 \draw[D](-0.05,-0.35) -- (0.7,-0.35);
  \draw[D](1.25,1) -- (1.8,1);
\end{tikzpicture}  
\qquad  + \qquad  
\begin{tikzpicture}[baseline = - 2pt]
    \draw[B, thick] (1,1) .. controls +(0,1.3) and +(0,1.3) ..  (-1.6, -0.4) ;
         \draw[A, thick] (1,1) .. controls +(0.9,2) and +(0,1.3) ..  (-1.9, -0.4) ;
 \node[triality] at (0, 0) {};
 \draw[C, thick] (1,1) .. controls +(0,-1/2) and +(0,1/2) ..  (0.1, 0.1) ;
  \draw[C, thick] (-1.2,-0.4) .. controls +(0,1/2) and +(0,1/2) ..  (-0.1, 0.1) ;
        \node[triality] at (1, 1) {};
  \begin{pgfonlayer}{foreground}
\node[blob] (blob3) at (0, 0) {};
\end{pgfonlayer}
  \coordinate (A) at (0.55,1/4);
    \coordinate (A') at (0.155,1/4);
      \coordinate (A'') at (0.05,1/4);
        \coordinate (A''') at (0.175,1/4);
  \coordinate (B) at (0.05,0);
    \coordinate (C) at (0.6,0);
  \coordinate (B') at (0,0);
    \coordinate (C') at (0.7,0);
                        \path (0.45, -0.45) edge[A,bend right=15] (1.8, 1.3);
                              \path (0.15, -0.45) edge[A,bend left=15] (-0.08, -0.1);
  \draw[A](1.6, 1.3) .. controls +(-0.5, -0.65) and +(0.3, -0.4) .. (0.09, -0.1);
 \draw[D](-0.05,-0.35) -- (0.7,-0.35);
  \draw[D](1.25,1) -- (1.8,1);
\end{tikzpicture}  
 \qquad  + \qquad  
 \begin{tikzpicture}[baseline = - 2pt]
    \draw[C, thick] (1,1) .. controls +(0,1.3) and +(0,1.3) ..  (-1.6, -0.4) ;
         \draw[B, thick] (1,1) .. controls +(0.9,2) and +(0,1.3) ..  (-1.9, -0.4) ;
 \node[triality] at (0, 0) {};
 \draw[A, thick] (1,1) .. controls +(0,-1/2) and +(0,1/2) ..  (0.1, 0.1) ;
  \draw[A, thick] (-1.2,-0.4) .. controls +(0,1/2) and +(0,1/2) ..  (-0.1, 0.1) ;
        \node[triality] at (1, 1) {};
  \begin{pgfonlayer}{foreground}
\node[blob] (blob3) at (0, 0) {};
\end{pgfonlayer}
  \coordinate (A) at (0.55,1/4);
    \coordinate (A') at (0.155,1/4);
      \coordinate (A'') at (0.05,1/4);
        \coordinate (A''') at (0.175,1/4);
  \coordinate (B) at (0.05,0);
    \coordinate (C) at (0.6,0);
  \coordinate (B') at (0,0);
    \coordinate (C') at (0.7,0);
                        \path (0.45, -0.45) edge[A,bend right=15] (1.8, 1.3);
                              \path (0.15, -0.45) edge[A,bend left=15] (-0.08, -0.1);
  \draw[A](1.6, 1.3) .. controls +(-0.5, -0.65) and +(0.3, -0.4) .. (0.09, -0.1);
 \draw[D](-0.05,-0.35) -- (0.7,-0.35);
  \draw[D](1.25,1) -- (1.8,1);
\end{tikzpicture}  \hspace{3pt}.
\] 
But this follows immediately from the triality preservation formula, proved earlier,  so we are done.

\begin{itemize}\item Finally, again using the colour symmetry, we need to prove the vanishing of the triality constraint with one of the tensor generators, so we need to prove that the following commutator relation:
\end{itemize}
\[ \hspace{-30pt} 0 \hspace{4pt} = \hspace{4pt} \Bigg[\qquad  \triconas \qquad   + \qquad   \triconbs  \qquad  +  \qquad  \triconcs \qquad , \qquad  \begin{tikzpicture}[baseline = - 2pt]
    \path (0, 0) edge[C, bend right = 30] (-0.4, -1.3);
        \path (0, 0) edge[A,bend right= 30] (0.4, 1.3);
\node[Circ, draw] (blob3) at (0,0) {};
\end{tikzpicture}\qquad \Bigg]\hspace{4pt}.\]
The computation is similar to the one just carried out.  Evaluating the commutators, we need the  following relation to be valid:
\[ \hspace{-40pt}0\hspace{4pt} = \hspace{4pt} \begin{tikzpicture}[baseline = - 2pt]
  \path{ (1, 1) edge[thick, bend left=20, C]  (1.2, -0.4)};
    \draw[A, thick] (1,1) .. controls +(0,1.3) and +(0,1.3) ..  (-1.6, -0.4) ;
 \draw[B, thick] (1,1) .. controls +(0,-1/2) and +(0,1/2) ..  (0.1, 0.1) ;
  \draw[B, thick] (-1.2,-0.4) .. controls +(0,1/2) and +(0,1/2) ..  (-0.1, 0.1) ;
        \node[triality] at (1, 1) {};
  \begin{pgfonlayer}{foreground}
\node[blob] (blob3) at (0, 0) {};
\end{pgfonlayer}
  \coordinate (A) at (0.55,1/4);
    \coordinate (A') at (0.155,1/4);
      \coordinate (A'') at (0.05,1/4);
        \coordinate (A''') at (0.175,1/4);
  \coordinate (B) at (0.05,0);
    \coordinate (C) at (0.6,0);
  \coordinate (B') at (0,0);
    \coordinate (C') at (0.7,0);
  \draw[A](1.6, 1.3) .. controls +(-0.5, -0.65) and +(0.3, -0.4) .. (0.09, -0.1);
    \draw[A](2.2, 1) .. controls +(-0.4, 1.3) and +(0.4, -1.3) .. (-0.09, -0.1);
  \path{ (2.2, 1) edge[thick, bend left =20, C]  (2.4, -0.4)};    
      \node[Circ] at (2.2, 1) {};
\end{tikzpicture}  
\qquad  + \qquad  
\begin{tikzpicture}[baseline = - 2pt]
  \path{ (1, 1) edge[thick, bend left=20, A]  (1.2, -0.4)};
    \draw[B, thick] (1,1) .. controls +(0,1.3) and +(0,1.3) ..  (-1.6, -0.4) ;
 \node[triality] at (0, 0) {};
 \draw[C, thick] (1,1) .. controls +(0,-1/2) and +(0,1/2) ..  (0.1, 0.1) ;
  \draw[C, thick] (-1.2,-0.4) .. controls +(0,1/2) and +(0,1/2) ..  (-0.1, 0.1) ;
        \node[triality] at (1, 1) {};
  \begin{pgfonlayer}{foreground}
\node[blob] (blob3) at (0, 0) {};
\end{pgfonlayer}
  \coordinate (A) at (0.55,1/4);
    \coordinate (A') at (0.155,1/4);
      \coordinate (A'') at (0.05,1/4);
        \coordinate (A''') at (0.175,1/4);
  \coordinate (B) at (0.05,0);
    \coordinate (C) at (0.6,0);
  \coordinate (B') at (0,0);
    \coordinate (C') at (0.7,0);
  \draw[A](1.6, 1.3) .. controls +(-0.5, -0.65) and +(0.3, -0.4) .. (0.09, -0.1);
    \draw[A](2.2, 1) .. controls +(-0.4, 1.3) and +(0.4, -1.3) .. (-0.09, -0.1);
  \path{ (2.2, 1) edge[thick, bend left =20, C]  (2.4, -0.4)};    
      \node[Circ] at (2.2, 1) {};
\end{tikzpicture}  
 \qquad  + \qquad  
 \begin{tikzpicture}[baseline = - 2pt]
  \path{ (1, 1) edge[thick, bend left=20, B]  (1.2, -0.4)};
    \draw[C, thick] (1,1) .. controls +(0,1.3) and +(0,1.3) ..  (-1.6, -0.4) ;
 \node[triality] at (0, 0) {};
 \draw[A, thick] (1,1) .. controls +(0,-1/2) and +(0,1/2) ..  (0.1, 0.1) ;
  \draw[A, thick] (-1.2,-0.4) .. controls +(0,1/2) and +(0,1/2) ..  (-0.1, 0.1) ;
        \node[triality] at (1, 1) {};
  \begin{pgfonlayer}{foreground}
\node[blob] (blob3) at (0, 0) {};
\end{pgfonlayer}
  \coordinate (A) at (0.55,1/4);
    \coordinate (A') at (0.155,1/4);
      \coordinate (A'') at (0.05,1/4);
        \coordinate (A''') at (0.175,1/4);
  \coordinate (B) at (0.05,0);
    \coordinate (C) at (0.6,0);
  \coordinate (B') at (0,0);
    \coordinate (C') at (0.7,0);
  \draw[A](1.6, 1.3) .. controls +(-0.5, -0.65) and +(0.3, -0.4) .. (0.09, -0.1);
      \draw[A](2.2, 1) .. controls +(-0.4, 1.3) and +(0.4, -1.3) .. (-0.09, -0.1);
  \path{ (2.2, 1) edge[thick, bend left =20, C]  (2.4, -0.4)};    
      \node[Circ] at (2.2, 1) {};
\end{tikzpicture}  \hspace{3pt}.
\] 
But again this follows immediately from the triality preservation identity, so we are done.  This finishes the proof that the triality constraint is consistent.  Thus the bracket is well defined and compatible with the triality constraints.  It remains only to verify the Jacobi identities, which is the task of the next section.

\section{The Jacobi identities}
In this section we verify the Jacobi identity for the bracket introduced in the previous section.  The verification is organized around two basic diagrammatic relations.  The first is the triple-fork identity, involving only the metric and Kronecker delta and encoding the Jacobi identity for the orthogonal Lie algebra.  The second is the two-colour double-fork identity, which is the corresponding relation involving the triality tensor.  Once these identities are established, the remaining Jacobi checks reduce to a finite list of diagrammatic computations.

Several of the Jacobi checks are handled in two ways.  The indirect argument, obtained by triality substitution and reduction to identities already proved, is the proof that works uniformly in all dimensions.  In a number of cases we also record the direct calculation.  That calculation gives the Jacobi identity when $n=8$ and in general yields additional relations.  The three-colour cases are colour permutations of one another, so after the first full treatment the later occurrences should be read with that symmetry in mind.

\subsection{The fork identities}
Define 
$
$ 
to be the tensor fork with three upper tines:
\[  %
\]
Then, for example, 
$%
$ 
denotes the fork
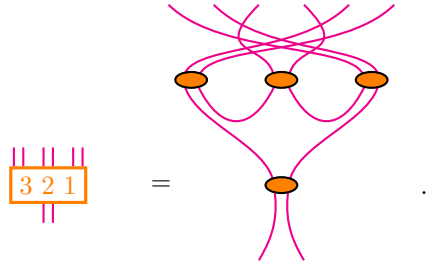
\[ %
\]
In general for a permutation of $123$, we just apply the permutation to the upper tines of the fork.\\\\  We now observe that the forks  $%
$ 
are actually equal tensors, as follows by switching the lower end of the latter fork and rearranging.  Consequently, we have the tensor equations:
\[ %
\hspace{4pt}.\]
In particular, we deduce the triple fork identity:
\[ %
\quad = \quad 0\quad.
\]

\subsection{The two-colour double fork identity}
This is the relation:
\[ %
 \]
We prove this as follows:
\[ %
 \]
The point here is that the first term in the last line vanishes, since we have the triple skewing identity, obtained by interchanging the two triality operators, picking up the product of three minus signs, so turning it into its negative:
\[ 0 \hspace{3pt} = \hspace{3pt} %
\]

We now prove the Jacobi identities, beginning with the pure tensor identities.
\begin{itemize}\item $%
$  
We need the following relation:
\end{itemize}
\[\hspace{-20pt} %
\]
The left hand side of this relation simplifies to:
\[\hspace{-20pt} %
\]
Using reflection symmetry in the horizontal and taking skew and symmetric parts, in the first index pairs, it suffices to verify the relation:
\[\hspace{-70pt} %
\]
So the required identity holds.

\begin{itemize}\item $%
$
\end{itemize}
We need the following relation:
\[ \hspace{-20pt} %
\]
Using the horizontal reflection symmetry, and decomposing into skew and symmetric parts, if suffices to prove the following identity:
\[ %
\]
\begin{itemize}\item The left-hand side of the desired relation becomes:
\end{itemize}
\[ %
 \]
\begin{itemize}\item The right-hand side of the desired relation becomes:
\end{itemize}
\[ %
 \]
So the required relation holds provided we fix the value of $\kappa$ to be $-8$, as we do henceforth, so we are done.

\begin{itemize}\item $%
$
\end{itemize}
We need the following relation:
\[\hspace{-50pt}  %
 
\]
The first term of the right-hand side of the desired relation simplifies as follows:
\[ %
\]
Putting in the appropriate colour permutations for the other two terms, we see that the desired relation follows immediately from the triality invariance of the two sets of three orthogonal generators, so we are done.  We now pass to the two-tensor identities.
\begin{itemize} \item $%
$
\end{itemize}
We need to show the following formula:
\[ \hspace{-30pt} %
 \hspace{5pt}\]

\begin{itemize}\item The left-hand side is:
\end{itemize}
\[%
 \]
\begin{itemize}\item   The right-hand side is:
\end{itemize}
\[  \hspace{5pt} %
 \]
So both sides agree and we are done.
\begin{itemize} \item $%
$
\end{itemize}
We need to show the following formula:
\[ \hspace{-30pt} %
 \hspace{5pt}\]
\begin{itemize}\item The left-hand side evaluates to the following:
\end{itemize}
\[ %
\hspace{5pt}\]

\begin{itemize}\item The terms of the right-hand side become the following:
\end{itemize}
\[ \hspace{-35pt}%
\]
The required relation now follows immediately from the invariance of the triality operator, so we are done.
\begin{itemize} \item $%
$
\end{itemize}
We need to show the following formula:
\[ \hspace{-30pt} %
 \hspace{5pt}\]
\begin{itemize}\item The left-hand side evaluates to the following:
\end{itemize}
\[ %
\hspace{5pt}\]
\begin{itemize}\item   The right-hand side becomes:
\end{itemize}
\[  \hspace{5pt} %
\]
The required identity now follows immediately from the triality preservation condition, so we are done.  We pass to the one-tensor identities.

 \begin{itemize}\item $%
$
\end{itemize} We need to prove the following relation:
 \[ \hspace{-42pt}%
\]
So the left-hand side of the required relation reduces to the following:
\[ %
\]
Next we have for the right-hand side of the required relation:
\[ \hspace{-40pt}%
 \]
The two sides of the relation are equal, so we are done.
\begin{itemize}\item  $%
$
\end{itemize} We need to prove the following relation:
 \[ \hspace{-42pt}%
\]
So the left-hand side of the required relation reduces to the following:
\[ %
\]
Next we have for the right-hand side of the required relation:
\[ \hspace{-40pt}%
 \]
The equality of the two sides of the desired relation now follows immediately from the two-color double fork identity, proved above, so we are done.
 \begin{itemize}\item  $%
$
\end{itemize} We need to prove the following relation:
 \[ \hspace{-42pt}%
 \]
 \begin{itemize}\item The first term of the left-hand-side becomes the following:
\end{itemize}
\[  %
\]

  \begin{itemize}\item The second double commutator of the left-hand-side becomes the following:
\end{itemize}
\[  %
\]
\begin{itemize}\item Putting the two terms together, the left-hand-side of the required relation becomes:
\end{itemize}
\[ %
 \]
\begin{itemize}\item Next we have for the right-hand side of the required relation:
\end{itemize}
\[ \hspace{-40pt}%
 \]
So the required identity holds and we are done.
  \begin{itemize}\item $%
$
\end{itemize}
It emerges that the direct approach to this problem (to be given later), while it gives an interesting relation and gives the required Jacobi identity when $n = 8$, will not give the requisite information directly in general.  So to prove the Jacobi identity, we instead use triality operators and the triality constraint to reduce the requisite identity to those already proved.   We need to prove the following relation:
 \[ \hspace{-50pt}%
 \]
We exploit the triality identity to eliminate the cyan orthogonal generator when tensored with a red Kronecker delta to a linear combination of the other two orthogonal generators:
\[ %
\]
Using this strategy, the required relation becomes the following:
\[ \hspace{-50pt}%
 \]
Here each of the last two lines of this equation comes to zero by virtue of the two-color Jacobi identities already proved above,  so we are done.   Next we determine the result of the direct approach.  Again we are aiming at proving the following relation:
 \[ \hspace{-50pt}%
 \]
For the first term of the left-hand side we have:
\[  %
\]
So the left-hand side of the required relation reduces to the following:
\[ %
\]
Next we have for the right-hand side of the required relation:
\[ \hspace{-40pt}%
 \]
In the last line of this calculation, we first used the double fork identity, proved above and then the (rescaled) off-diagonal box identity.   Putting the left and right-hand sides together, we deduce that the required Jacobi identity holds if and only if we have the relation:
\[ %
 \]
The preceding Jacobi computation yields the following relation.  It is automatic when $n=8$, while for $n
\neq 8$ it may be rewritten more symmetrically as an identity for the triality tensor:
\[ %
 \]
This relation is compatible with the familiar associative degenerations in the lower-dimensional cases; we record it here only as a consequence of the Jacobi computation.  If we further contract with a box and use the off-diagonal box identity, we obtain the corresponding relation for the double fork:
\[ %
 \]
This relation is likewise compatible with the commutative degenerations that occur in the lower-dimensional cases.  

Note that using the diagonal box identity, twice, the last relation can be re-written equivalently in a way that does not involve the triality at all, just the Kronecker delta tensor and the metric, as follows:
\[ %
 \]
When $(n - 4)(n - 8)$ is invertible, this relation is saying that the structure constants for the orthogonal Lie algebra vanish, so the algebra is abelian, as expected in either the real or complex case.  We turn to the Jacobi identities involving three orthogonal generators.
\begin{itemize}\item  $%
$
\end{itemize} For this, we need to prove the following relation:
 \[ \hspace{-80pt}%
 \]
The first term of the right-hand side is calculated as follows:
\[ %
\]
The other two terms of the desired relation are just given by the cyclic permutations applied to $123$.    Putting the resulting six terms together gives the required identity, as an immediate consequence of the triple fork identity, proved above, so we are done.

\begin{itemize}\item  $%
$
\end{itemize} For this, we need to prove the following relation:
\[ \hspace{-42pt}%
 \]
All three terms are calculated in a similar way.   The first term of the left-hand side is:
\[  %
\]
Putting in the other two terms, the required result follows immediately from the triple fork identity, proved above, so we are done.

  \begin{itemize}\item $%
$
\end{itemize}
It emerges that the direct approach to this problem (to be given later), while it gives an interesting relation and gives the required Jacobi identity when $n = 8$, will not give the requisite information directly in general.  So to prove the Jacobi identity, we instead use triality operators and the triality constraint to reduce the requisite identity to those already proved.   We need to prove the following relation:
 \[ \hspace{-50pt}%
 \]
We exploit the triality identity to eliminate the cyan orthogonal generator when tensored with a red Kronecker delta to a linear combination of the other two orthogonal generators:
\[ %
\]
Using this strategy, the required relation becomes the following:
\[ \hspace{-50pt}%
 \]
Here each of the last two lines of this equation comes to zero by virtue of the two-color Jacobi identities already proved above,  so we are done.   Next we determine the result of the direct approach.  Again we are aiming at proving the following relation:
 \[ \hspace{-50pt}%
 \]
For the first term of the left-hand side we have:
\[  %
\]

Similarly the second commutator is:
\[  %
\]
Putting these terms together, we get for the first two terms of the projected identity the formula:
\[ %
\]
Then the third commutator is:
\[ %
\]
Putting all the terms together, the Jacobi identity reduces to the following relation:
\[ %
\]
So we see the same structure in the identity 
$\hspace{3pt}%
\hspace{3pt}$
as in the identity $\hspace{3pt}%
\hspace{3pt}$.

So the left-hand side of the required relation reduces to the following:
\[ %
 \]
In the last line of this calculation, we first used the double fork identity, proved above and then the (rescaled) off-diagonal box identity.   Putting the left and right-hand sides together, we deduce that the required Jacobi identity holds if and only if we have the relation:
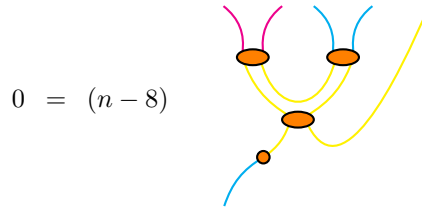
\[ %
 \]
The preceding Jacobi computation yields the following relation.  It is automatic when $n=8$, while for $n
\neq 8$ it may be rewritten more symmetrically as an identity for the triality tensor:
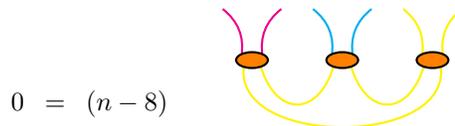
\[ %
 \]
This relation is compatible with the familiar associative degenerations in the lower-dimensional cases; we record it here only as a consequence of the Jacobi computation.  If we further contract with a box and use the off-diagonal box identity, we obtain the corresponding relation for the double fork:
\[ %
 \]
This relation is likewise compatible with the commutative degenerations that occur in the lower-dimensional cases.  

Note that using the diagonal box identity, twice, the last relation can be re-written equivalently in a way that does not involve the triality at all, just the Kronecker delta tensor and the metric, as follows:

\[ \hspace{5pt}, \hspace{10pt} \Bigg[ %
 \hspace{5 pt} \Bigg]\Bigg] \hspace{5pt}. \]
 We prove this indirectly, by reducing it to a pair of Jacobi identities (equivalent to each other, by an appropriate colour permutation) , the identities $\hspace{5pt}%
\hspace{5pt}$ already proved above, again using the following relation to eliminate one orthogonal group generator in favor of the other two:
 \[ \hspace{-25pt}    %
 \hspace{5pt}.\]

A direct computation of this Jacobi identity, without reducing to an earlier case, also yields the relation:
\[ \hspace{-60pt} \Bigg[ %
  \hspace{5pt} \]
So we need the following relation to hold:
\[ 0 = \hspace{10pt} %
  \hspace{5pt} \]
Now we have the relation, proved above:
\[ %
  \hspace{5pt} \]
So we have now, using the off-diagonal double box identity:
\[  2 %
\]
So the required Jacobi identity holds provided we have the relation:
\[ 0 = (n  - 8) \hspace{5pt} %
\]
Since we have proved the Jacobi identity by other means, above, we conclude that this relation holds.  \\\\
The same identity shows up in the consideration of the other three-colour Jacobi identity:
\begin{itemize}\item  $%
$. 
\end{itemize} For this, we need to prove the following relation:
 \[ \hspace{-80pt} 0 =  \hspace{5pt}\Bigg[ %
 \hspace{5 pt} \Bigg]\Bigg] \hspace{5pt}. \]

This is the colour-permuted companion of the preceding three-colour calculation.
The indirect proof is the same as before, with the colours permuted.  Applying the
same triality-substitution argument reduces the present Jacobi check to the same
three-colour identity obtained in the preceding case, so we do not repeat the full
string of diagrams here.  The Jacobi identity therefore follows, again, from the
triality constraint.  The corresponding direct computation is likewise just the
colour permutation of the previous one and yields the same kind of auxiliary
relation, so we omit it.

  \begin{itemize}\item  We are left with the three-colour identity: $
$
\end{itemize}
It emerges that the direct approach to this problem (to be given later), while it gives an interesting relation and gives the required Jacobi identity when $n = 8$, will not give the requisite information directly in general.  So to prove the Jacobi identity, we instead use triality operators and the triality constraint to reduce the requisite identity to those already proved.   We need to prove the following relation:
 \[ \hspace{-50pt}%
 \]
We exploit the triality identity to eliminate the cyan orthogonal generator when tensored with a red Kronecker delta to a linear combination of the other two orthogonal generators:
\[ %
\]
Using this strategy, the required relation becomes the following:
\[ \hspace{-50pt}%
 \]
Here each of the last two lines of this equation comes to zero by virtue of the two-color Jacobi identities already proved above,  so we are done.   Next we see the result of the direct approach.  Again we are aiming at proving the following relation:
 \[ \hspace{-50pt}%
 \]
For the first term of the left-hand side we have:
\[  %
\]
So the left-hand side of the required relation reduces to the following:
\[ %
\]
Next we have for the right-hand side of the required relation:
\[ \hspace{-40pt}%
 \]
In the last line of this calculation, we first used the double fork identity, proved above and then the (rescaled) off-diagonal box identity.   Putting the left and right-hand sides together, we deduce that the required Jacobi identity holds if and only if we have the relation:
\[ %
 \]

  \begin{itemize}\item  We are left with the three-colour identity: $%
$
\end{itemize}
We need to prove the following relation:
 \[ \hspace{-80pt} 0 =  \hspace{5pt}\Bigg[ %
 \hspace{5 pt} \Bigg]\Bigg] \hspace{5pt}. \]
 This is the last of the three-colour Jacobi checks.  It is again a colour permutation of the earlier argument, and we indicate the reduction in that form.  We prove this by reducing it to a pair of Jacobi identities, the two-colour identities, $\hspace{5pt}%
\hspace{5pt}$, both proved above.  To this end consider a product of two orthogonal group generators occurring in the identity:
  \[   %
 \]
 We introduce a pair of triality operators, using the basic triality identity.  Then we use triality invariance to eliminate one of the orthogonal group generators in favor of the other two:
 \[ \hspace{-50pt} 2\hspace{10pt}   %
 \hspace{5pt}.\]
If now we transform every occurrence of the pair of orthogonal generators in the projected Jacobi identity (respecting their order), and assemble corresponding terms, we reduce exactly to the pair of two-colour Jacobi identities already proved above:  $\hspace{5pt}%
\hspace{5pt}$.  So we are done.  We have finished the Jacobi identities involving at least one of the tensor generators, so we turn to the remaining identities, which all involve three orthogonal generators.

The preceding Jacobi computation yields the following relation.  It is automatic when $n=8$, while for $n
eq 8$ it may be rewritten more symmetrically as an identity for the triality tensor:
\[ %
 \]
This relation is compatible with the familiar associative degenerations in the lower-dimensional cases; we record it here only as a consequence of the Jacobi computation.  If we further contract with a box and use the off-diagonal box identity, we obtain the corresponding relation for the double fork:
\[ %
 \]
This relation is likewise compatible with the commutative degenerations that occur in the lower-dimensional cases.  Note that using the diagonal box identity, twice, the last relation can be re-written equivalently in a way that does not involve the triality at all, just the Kronecker delta tensor and the metric, as follows:
\[ %
 \]
When $(n - 4)(n - 8)$ is invertible, this relation is saying that the structure constants for the orthogonal Lie algebra vanish, so the algebra is abelian, as expected in either the real or complex case.

\section{Identification with the Magic Square}

We now identify the Lie algebras constructed in the preceding sections with the entries of the Freudenthal magic square. \cite{Freudenthal1954I,Freudenthal1954II,Tits1966,Vinberg1966,BartonSudbery2003,LandsbergManivel2001,LandsbergManivel2002,Elduque2006,Elduque2007}  Throughout this section we work first over $\mathbb C$.  Once the complex identifications have been made, the compact and split real forms follow from the compact and split trialities introduced earlier.

Let $\gamma$ and $\gamma'$ be trialities of dimensions $n$ and $n'$, respectively, with $n,n'\in\{1,2,4,8\}$.  We write
\[
\mathfrak g_{n,n'}:=\mathfrak g(\gamma,\gamma').
\]
By the preceding sections, this is a Lie algebra of dimension
\[
\dim \mathfrak g_{n,n'}=3nn'+\tau_n+\tau_{n'},
\qquad
\tau_1=0,\ \tau_2=2,\ \tau_4=9,\ \tau_8=28.
\]
The construction is symmetric in the two trialities, so $\mathfrak g_{n,n'}\cong \mathfrak g_{n',n}$.

The dimensions are therefore
\[
\begin{array}{c|cccc}
 n\backslash n' & 1 & 2 & 4 & 8 \\\hline
 1 & 3 & 8 & 21 & 52 \\
 2 & 8 & 16 & 35 & 78 \\
 4 & 21 & 35 & 66 & 133 \\
 8 & 52 & 78 & 133 & 248
\end{array}
\]
and this is already the dimension table of the Freudenthal square.

The point of the present section is that these dimensions are not accidental: the root systems arising from the triality construction are exactly those of the corresponding simple or semisimple Lie algebras.

\begin{theorem}
Over $\mathbb C$, the Lie algebras $\mathfrak g_{n,n'}$ are precisely the entries of the Freudenthal magic square:
\[
\begin{array}{c|cccc}
 n\backslash n' & 1 & 2 & 4 & 8 \\\hline
 1 & A_1 & A_2 & C_3 & F_4 \\
 2 & A_2 & A_2\oplus A_2 & A_5 & E_6 \\
 4 & C_3 & A_5 & D_6 & E_7 \\
 8 & F_4 & E_6 & E_7 & E_8.
\end{array}
\]
Equivalently,
\[
\begin{array}{c|cccc}
 n\backslash n' & 1 & 2 & 4 & 8 \\\hline
 1 & \mathfrak{sl}_2 & \mathfrak{sl}_3 & \mathfrak{sp}_6 & \mathfrak f_4 \\
 2 & \mathfrak{sl}_3 & \mathfrak{sl}_3\oplus\mathfrak{sl}_3 & \mathfrak{sl}_6 & \mathfrak e_6 \\
 4 & \mathfrak{sp}_6 & \mathfrak{sl}_6 & \mathfrak{so}_{12} & \mathfrak e_7 \\
 8 & \mathfrak f_4 & \mathfrak e_6 & \mathfrak e_7 & \mathfrak e_8.
\end{array}
\]
Here the $3$-dimensional entry may equally well be denoted $A_1$, $B_1$ or $C_1$.
\end{theorem}

\begin{proof}
We give the root-theoretic identification for the exceptional entries $F_4$, $E_6$, $E_7$ and $E_8$, since these are the genuinely distinctive cases.  The remaining classical entries are then obtained by the same weight bookkeeping in lower rank and are summarized at the end of the proof.

Fix Cartan subalgebras of the triality algebras on the two sides.  For $n=8$ this is a Cartan subalgebra of $\mathfrak{so}(8)$, with weight basis $e_1,e_2,e_3,e_4$; for $n=4$ one gets rank $3$; for $n=2$ the triality algebra is already abelian of rank $2$; and for $n=1$ the triality algebra is trivial.  The tensor generators carry the obvious product weights under the action of the two triality algebras.  The roots of $\mathfrak g_{n,n'}$ are therefore read off from the diagonal orthogonal generators together with the three tensor sectors.

\medskip
\noindent\textbf{The case $(8,8)$: the algebra $E_8$.}
Let $\pm e_i$ be the weights of the vector representation of the first $D_4$, and let $\pm f_j$ be those of the second.  The diagonal parts contribute the roots
\[
\pm e_i\pm e_j\quad (i<j),
\qquad
\pm f_i\pm f_j\quad (i<j).
\]
The tensor generator transforming as vector $\otimes$ vector contributes roots
\[
\pm e_i\pm f_j.
\]
The other two tensor sectors transform as the two half-spin representations on each side.  Their weights are therefore the half-sums
\[
\frac12(\pm e_1\pm e_2\pm e_3\pm e_4),
\qquad
\frac12(\pm f_1\pm f_2\pm f_3\pm f_4),
\]
with the usual parity conditions for the two chiralities, and the mixed tensor sectors contribute
\[
\frac12(\pm e_1\pm e_2\pm e_3\pm e_4\pm f_1\pm f_2\pm f_3\pm f_4)
\]
with an even parity condition.  Altogether the roots are
\[
\pm e_i\pm e_j,\qquad \pm f_i\pm f_j,\qquad \pm e_i\pm f_j,
\qquad
\frac12(\pm e_1\pm e_2\pm e_3\pm e_4\pm f_1\pm f_2\pm f_3\pm f_4),
\]
where the last family has even parity.  This is precisely the standard root system of type $E_8$.

\medskip
\noindent\textbf{The case $(4,8)$: the algebra $E_7$.}
On the $8$-dimensional side we again have a $D_4$ root system with weights $\pm e_i$.  On the $4$-dimensional side the triality algebra has rank $3$ and may be described by three mutually constrained $A_1$-factors.  Writing their roots as
\[
\pm L_{+},\ \pm L_{-},\ \pm L'_{+},\ \pm L'_{-},\ \pm L''_{+},\ \pm L''_{-},
\]
with the triality identifications among the three copies, the diagonal part contributes a rank-$3$ subsystem of type $A_1\times A_1\times A_1$.  The mixed tensor generators contribute three families of roots of the form
\[
\pm e_i\pm f_j,
\qquad
\pm e'_i\pm f'_j,
\qquad
\pm e''_i\pm f''_j,
\]
where the primed and double-primed $e$-weights denote the two half-spin weight systems of $D_4$.

After relabelling these weights in the standard Bourbaki coordinates $e_1,\dots,e_8$, one obtains the root system
\[
\pm e_i\pm e_j \quad (1\le i<j\le 6),
\qquad
\pm(e_7-e_8),
\qquad
\frac12\Bigl(e_8-e_7+\sum_{i=1}^6 (-1)^{\nu(i)}e_i\Bigr)
\]
with $\sum_{i=1}^6 \nu(i)$ odd.  This is the standard root system of type $E_7$.

\medskip
\noindent\textbf{The case $(2,8)$: the algebra $E_6$.}
Here the $8$-dimensional triality again contributes a $D_4$ root system
\[
\pm e_i\pm e_j.
\]
The $2$-dimensional triality algebra is abelian of rank $2$; let its three characters be $f,f',f''$ with the single relation
\[
 f+f'+f''=0.
\]
Then the three tensor sectors contribute roots
\[
\pm e_i\pm f,
\qquad
\pm e'_i\pm f',
\qquad
\pm e''_i\pm f''.
\]
After introducing Bourbaki coordinates $e_5,e_6,e_7,e_8$ so that, for example,
\[
 e_5=f,
 \qquad
 e_8-e_7-e_6 = 2f',
\]
one recovers the usual description of the $E_6$ roots:
\[
\pm e_i\pm e_j\quad (1\le i<j\le 5),
\qquad
\frac12\Bigl(\pm e_8\mp e_7\mp e_6+\sum_{i=1}^5 (\pm e_i)\Bigr)
\]
with the appropriate parity condition.  Hence the root system is of type $E_6$.

\medskip
\noindent\textbf{The case $(1,8)$: the algebra $F_4$.}
When the lower triality is $1$-dimensional there is no non-trivial diagonal contribution from that side.  Thus the diagonal roots are just those of $D_4$:
\[
\pm e_i\pm e_j.
\]
The three tensor sectors contribute the vector and two half-spin weight systems of $D_4$:
\[
\pm e_i,
\qquad
\frac12(\pm e_1\pm e_2\pm e_3\pm e_4)
\]
with the two parity choices corresponding to the two chiral spin representations.  Their union with the $D_4$ roots is exactly the root system of type $F_4$.

\medskip
\noindent\textbf{The remaining entries.}
Once the exceptional cases are in hand, the classical entries are straightforward.  The rank and dimension match those of the remaining entries of the magic square, and the same weight bookkeeping gives the corresponding classical root systems:
\begin{itemize}
\item $(1,1)$ gives a rank-$1$, dimension-$3$ algebra, hence $A_1$;
\item $(1,2)$ gives rank $2$, dimension $8$, hence $A_2$;
\item $(1,4)$ gives rank $3$, dimension $21$, and the tensor weights together with the three $A_1$-roots form the root system $C_3$;
\item $(2,2)$ gives a rank-$4$, dimension-$16$ semisimple algebra, necessarily $A_2\oplus A_2$;
\item $(2,4)$ gives rank $5$, dimension $35$, hence $A_5$;
\item $(4,4)$ gives rank $6$, dimension $66$, hence $D_6$.
\end{itemize}
This completes the identification of the complex Lie algebras.
\end{proof}

\begin{corollary}
If both trialities are compact real trialities, then the resulting real Lie algebra is the compact real form of the corresponding complex algebra.  If both trialities are split real trialities, then the resulting real Lie algebra is the split real form.
\end{corollary}

\begin{proof}
In the compact case all roots are imaginary, so the real form is compact.  In the split case the full root system is already defined over $\mathbb R$, so the real form is split.  This is immediate from the explicit root descriptions above in the exceptional cases, and the same argument applies in the classical cases.
\end{proof}

\begin{remark}
The mixed-signature cases also produce natural non-compact real forms.  For the purposes of the present paper, however, it is enough to isolate the two extremal cases---compact and split---cleanly inside the main theorem chain.  A more detailed case-by-case analysis of intermediate real forms may be added separately.
\end{remark}

\section{The Killing form of the Lie algebra}

Let
\[
\mathfrak g=\mathfrak t^{\mathrm{up}}\oplus \mathfrak t^{\mathrm{low}}\oplus W_1\oplus W_2\oplus W_3
\]
be the Lie algebra attached to a pair of trialities, where
\[
W_1=\mathbb V_1\otimes \mathbb V'_2,\qquad
W_2=\mathbb V_2\otimes \mathbb V'_3,\qquad
W_3=\mathbb V_3\otimes \mathbb V'_1.
\]
Write \(B\) for the Killing form of \(\mathfrak g\).

The bracket relations established earlier show that the adjoint action respects this
decomposition in a very rigid way.  The orthogonal generators preserve each tensor
sector, upper generators acting only on the upper tensor factor and lower generators
acting only on the lower one.  A tensor generator in \(W_i\) maps \(W_i\) to
\(\mathfrak t^{\mathrm{up}}\oplus \mathfrak t^{\mathrm{low}}\), and interchanges the other two
tensor sectors.  This is enough to determine the orthogonality pattern of the
Killing form.

\begin{proposition}
The following subspaces are pairwise orthogonal with respect to the Killing form:
\[
\mathfrak t^{\mathrm{up}}\perp \mathfrak t^{\mathrm{low}},\qquad
\mathfrak t^{\mathrm{up}}\perp W_i,\qquad
\mathfrak t^{\mathrm{low}}\perp W_i,\qquad
W_i\perp W_j\quad (i\neq j).
\]
\end{proposition}

\begin{proof}
We treat the various cases separately.

First let \(R\in \mathfrak t^{\mathrm{up}}\) and \(X\in W_i\).  Then
\[
B(R,X)=\operatorname{tr}(\operatorname{ad}_R\operatorname{ad}_X).
\]
Now \(\operatorname{ad}_X\) maps \(\mathfrak t^{\mathrm{up}}\) and
\(\mathfrak t^{\mathrm{low}}\) into \(W_i\), maps \(W_i\) into
\(\mathfrak t^{\mathrm{up}}\oplus \mathfrak t^{\mathrm{low}}\), and interchanges the
other two tensor sectors.  The operator \(\operatorname{ad}_R\) preserves each
summand of the decomposition of \(\mathfrak g\).  Hence
\(\operatorname{ad}_R\operatorname{ad}_X\) carries each summand of \(\mathfrak g\) into a
different summand.  Its trace is therefore zero.  Thus
\(B(\mathfrak t^{\mathrm{up}},W_i)=0\).  The same argument gives
\(B(\mathfrak t^{\mathrm{low}},W_i)=0\).

Next let \(X\in W_i\) and \(Y\in W_j\) with \(i\neq j\), and let \(k\) be the third
index.  Then \(\operatorname{ad}_Y\) maps
\[
\mathfrak t^{\mathrm{up}},\mathfrak t^{\mathrm{low}}\to W_j,\qquad
W_j\to \mathfrak t^{\mathrm{up}}\oplus \mathfrak t^{\mathrm{low}},\qquad
W_i\to W_k,\qquad
W_k\to W_i.
\]
Applying \(\operatorname{ad}_X\) afterwards gives
\[
\mathfrak t^{\mathrm{up}},\mathfrak t^{\mathrm{low}}\to W_k,\qquad
W_j\to W_i,\qquad
W_i\to \mathfrak t^{\mathrm{up}}\oplus \mathfrak t^{\mathrm{low}},\qquad
W_k\to W_j.
\]
Again no summand is mapped to itself, so
\(\operatorname{tr}(\operatorname{ad}_X\operatorname{ad}_Y)=0\).  Hence
\(B(W_i,W_j)=0\) for \(i\neq j\).

Finally let \(R\in\mathfrak t^{\mathrm{up}}\) and \(S\in\mathfrak t^{\mathrm{low}}\).  Then
\[
B(R,S)=\operatorname{tr}(\operatorname{ad}_R\operatorname{ad}_S).
\]
The operator \(\operatorname{ad}_S\) vanishes on \(\mathfrak t^{\mathrm{up}}\) and preserves
\(\mathfrak t^{\mathrm{low}}\), while \(\operatorname{ad}_R\) vanishes on
\(\mathfrak t^{\mathrm{low}}\) and preserves \(\mathfrak t^{\mathrm{up}}\).  Thus
\(\operatorname{ad}_R\operatorname{ad}_S\) has zero trace on the two orthogonal
subalgebras.  On each tensor sector \(W_i=\mathbb V_i\otimes \mathbb V'_{i+1}\) it has
the form
\[
A_i\otimes B_{i+1},
\]
with \(A_i\in \mathfrak{so}(\mathbb V_i)\) and
\(B_{i+1}\in \mathfrak{so}(\mathbb V'_{i+1})\).  Hence
\[
\operatorname{tr}(A_i\otimes B_{i+1})
=\operatorname{tr}(A_i)\operatorname{tr}(B_{i+1})=0,
\]
because skew endomorphisms are traceless.  Therefore
\(B(\mathfrak t^{\mathrm{up}},\mathfrak t^{\mathrm{low}})=0\).
\end{proof}

It follows that the Killing form is block-diagonal with respect to the decomposition
\[
\mathfrak g=\mathfrak t^{\mathrm{up}}\oplus \mathfrak t^{\mathrm{low}}\oplus W_1\oplus W_2\oplus W_3.
\]
On each tensor sector \(W_i\), the restriction of \(B\) is invariant under the action
of \(\mathfrak t^{\mathrm{up}}\oplus \mathfrak t^{\mathrm{low}}\).  By triality symmetry, the same
scalar occurs on all three tensor sectors, so there is a constant \(\lambda\) such that
\[
B|_{W_i}=\lambda\,(\text{product metric on }W_i),\qquad i=1,2,3.
\]
Likewise the restrictions of \(B\) to \(\mathfrak t^{\mathrm{up}}\) and
\(\mathfrak t^{\mathrm{low}}\) are invariant symmetric bilinear forms on those Lie
algebras.  For the purposes of this paper one does not need their numerical
normalizations.  What matters is non-degeneracy.

\begin{theorem}
For \(n,n'\in\{1,2,4,8\}\), the Killing form of \(\mathfrak g\) is non-degenerate.
More precisely, in every case except \((n,n')=(2,2)\), the Lie algebra \(\mathfrak g\)
is simple, hence its Killing form is non-degenerate.  In the case \((2,2)\), one has
\[
\mathfrak g \cong A_2\oplus A_2,
\]
so the Killing form is the orthogonal direct sum of the two \(A_2\) Killing forms
and is again non-degenerate.
\end{theorem}

\begin{proof}
The preceding section identifies the complex Lie algebra obtained from each pair
\((n,n')\in\{1,2,4,8\}^2\) with the corresponding entry of the Freudenthal magic
square.  All of those entries are simple except \(A_2\oplus A_2\), which occurs for
\((n,n')=(2,2)\).  For a simple Lie algebra the Killing form is non-degenerate; for a
direct sum it is the direct sum of the Killing forms of the summands.  The claim
follows.
\end{proof}

Since the decomposition of \(\mathfrak g\) is orthogonal, the theorem implies in
particular that the restriction of the Killing form to each summand appearing above
is non-degenerate.  Thus \(\lambda\neq 0\) on the tensor sectors, and the restricted
forms on \(\mathfrak t^{\mathrm{up}}\) and \(\mathfrak t^{\mathrm{low}}\) are non-degenerate as
well.  This is the point needed later.  If one wants explicit numerical constants,
they can be recovered from the standard normalization of the Killing form in the
identified simple Lie algebra, but that plays no role in the present paper.

\section{The universal Clifford algebra}

This section is auxiliary to the main line of the paper.  It is not used in the later identification of the magic square entries, and a first reading can skip it without loss.  We include it because the triality construction naturally suggests a universal Clifford-algebra package, and because that package seems likely to be useful in subsequent work.
\noindent Let $\mathbb{V}$ be a vector space of dimension a non-negative integer $n$ over a field $\mathbb{F}$, whose characteristic is not $2$.   Denote by $\mathcal{T}(\mathbb{V})$ the (associative) tensor algebra of $\mathbb{V}$.   This algebra may be described as follows. For $k$ a non-negative integer, a word $w$ of length $k$ in $\mathcal{T}(\mathbb{V})$ is an ordered set $w = (f, v_1, v_2, \dots,  v_k)$, where $f \in \mathbb{F}$ and $v_i \in \mathbb{V}$, for $i = 1, 2, \dots, k$.  The vectors $v_i, i = 1, 2, \dots, k$ are called the letters of the word $w$.  The multiplication of the word $w$ with another word $x = (g, x_1, x_2, \dots ,x_m)$ of length a non-negative integer $m$, with $g \in \mathbb{F}$ and   $x_j \in \mathbb{V}$, for $j = 1, 2, \dots, m$ is the word $y = (h, y_1, y_2,  \dots, y_p)$, where $p = k + m$, $h = fg$, $y_r = w_r$, for $r = 1, 2, \dots, k$ and $y_r = x_{r - k}$, for $r = k + 1, k + 2, \dots, p$.  Then $\mathcal{T}(\mathbb{V})$ is spanned by all words, subject to the requirement that each word is linear in each of its letters and that the multiplication of words is linear in each argument and in each letter of each word.   We usually abbreviate a word, for example writing the word $w$ above as $w = fw_1w_2\dots w_k$.  For each non-negative integer $k$, the vector space spanned by  all words of length $k$ is denoted $\mathcal{T}^k(\mathbb{V})$.   When convenient, for $m$ a negative integer, we put $\mathcal{T}^m(\mathbb{V}) = \{0\}$.  Then the span of all the $\mathcal{T}^k(\mathbb{V})$, as the nonnegative integer $k$ varies, gives the algebra $\mathcal{T}(\mathbb{V})$.   Then the subspace $\mathcal{T}^k(\mathbb{V})$ is said to have grade $k$ and we have $\mathcal{T}^k(\mathbb{V})\mathcal{T}^m(\mathbb{V})\subset \mathcal{T}^{k + m}(\mathbb{V})$ for any non-negative integers $k$ and $m$.  In particular $\mathcal{T}(\mathbb{V})$ is generated as an algebra over $\mathbb{F}$, by its subspaces $\mathcal{T}^0(\mathbb{V}) = \mathbb{F}$ and $\mathcal{T}^1(\mathbb{V}) = \mathbb{V}$.    Define:
\[ \mathcal{T}^+(\mathbb{V}) = \Sigma_{j = 0}^\infty \mathcal{T}^{2j} (\mathbb{V}), \hspace{10pt} \mathcal{T}^-(\mathbb{V}) = \Sigma_{k = 0}^\infty \mathcal{T}^{2k+1} (\mathbb{V}).\]  Then we have:
\[ \mathcal{T}(\mathbb{V})  = \mathcal{T}^+(\mathbb{V}) +\mathcal{T}^-(\mathbb{V}), \]
\[ \mathcal{T}^+(\mathbb{V})  \mathcal{T}^+(\mathbb{V})  +   \mathcal{T}^-(\mathbb{V})  \mathcal{T}^-(\mathbb{V})   \subset \mathcal{T}^+(\mathbb{V}),\]
\[  \mathcal{T}^+(\mathbb{V})  \mathcal{T}^-(\mathbb{V})  +   \mathcal{T}^-(\mathbb{V})  \mathcal{T}^+(\mathbb{V})   \subset \mathcal{T}^-(\mathbb{V}).\]   The sub-algebra  $\mathcal{T}^+(\mathbb{V})$ is called the even sub-algebra and the subspace   $\mathcal{T}^-(\mathbb{V})$, which is a two sided $\mathcal{T}^+(\mathbb{V})$-module, is called the odd subspace of $\mathcal{T}(\mathbb{V})$. 
\begin{itemize}\item By definition, the  universal Clifford algebra of $\mathbb{V}$, denoted $\mathcal{C}(\mathbb{V})$, is the quotient of the tensor algebra $\mathcal{T}(\mathbb{V})$, by the relations $v^2 \alpha = \alpha v^2$, for any $v\in \mathbb{V}$ and any $\alpha\in \mathcal{T}(\mathbb{V})$.  
\end{itemize}
Denote by $\mathcal{Z}$ the center of the Clifford algebra.   So $\mathcal{Z}$ is a commutative $\mathbb{Z}$-graded algebra over $\mathbb{F}$ and for each $v \in \mathbb{V}$, contains in particular the element $v^2$, of grade two.   The defining  relations of $\mathcal{C}(\mathbb{V})$ preserve grading,  so the Clifford algebra is naturally graded over the integers.   Denote by $\mathcal{C}^k(\mathbb{V})$ the elements of $\mathcal{C}(\mathbb{V})$ of grade $k$ and by $\mathcal{C}^+(\mathbb{V})$ and $\mathcal{C}^-(\mathbb{V})$, its even sub-algebra and odd subspace, respectively.  For any $v \in \mathbb{V}$, we have that $v^2$ lies in $\mathcal{Z}$ and for any $v$, $w$ and $x$ in $\mathbb{V}$, we have the relation:
\[ vwx + wvx = xvw + xwv.\]
For $v$ and $w$ in $\mathbb{C}^1(\mathbb{V}) = \mathbb{V}$, put:
\[ g(v, w) = 2^{-1}(vw + wv).\]
So $g$ is a non-degenerate symmetric bilinear form with values in $\mathcal{Z}$. Then the Clifford algebra is generated over $\mathcal{Z}$ by $\Omega(\mathbb{V})$, the exterior algebra of $\mathbb{V}$.     The action of $v \in \mathbb{V}$ on the Clifford algebra is represented  by the formula, valid for any $\beta \in \Omega(\mathbb{V})$:
\[ \beta \rightarrow v\beta + g(v, \delta)\beta.\] 
Here $\delta$ denotes the $\mathbb{V}$-valued derivation of $\Omega(\mathbb{V})$, that has degree minus one and whose action on $v \in \Omega^1(\mathbb{V}) = \mathbb{V}$ gives $v\in \mathbb{V} \otimes \Omega^0(\mathbb{V})$.  Note that the center $\mathcal{Z}$ contains $\Omega^n(\mathbb{V})$ in the case that $n$ is odd.\\\\
If $\{e_i, i = 1, 2, \dots n\}$ is a basis for $\mathbb{V}$, then the center $\mathcal{Z}$ contains the elements $g_{ij} = 2^{-1}(e_i e_j + e_j e_i)$, for $1\le i, j \le n$.\\\\
If $\alpha$ and $\beta$ lie in $\Omega(\mathbb{V})$, their product in the Clifford algebra is given by the formula:
\[ \alpha\beta = \mu(\alpha e^{D_{12}}\otimes \beta).\]
Here $\mu: \Omega(\mathbb{V}) \otimes \Omega(\mathbb{V})\rightarrow \Omega(\mathbb{V})$, $\alpha\otimes \beta\rightarrow \alpha \beta$ is the multiplication map of $\Omega(\mathbb{V})$.
Here the right-hand side is computed in the algebra $\Omega(\mathbb{V})$ and the operator $D_{12} = g(\delta', \delta)$ where $\delta$ is the usual $\mathbb{V}$-valued derivation acting on $\beta$ from the left and $\delta' \alpha$ is the same derivation but acting on $\alpha$ from the right, so for any $w \in \mathbb{V}^*$, we have $\delta'_w(\alpha) = (-1)^{k - 1} \delta_w\alpha$, where $\alpha$ has degree $k$.\\\\
\textbf{Example}\\
In the case that $\mathbb{V}$ is one-dimensional,   a  basis for the Clifford algebra is the set $\{1, x\}$, where $x \in \mathbb{V} - \{0\}$, with $1$ even and grade zero and $x$ odd and grade one, with the Clifford multiplication:  $x(x) = a$.  Here $a$ is a non-zero central element of the algebra of grade two.  In this case, the algebra is abelian, so the center $\mathcal{Z}$ is the whole algebra.    The algebra is just the graded ring of polynomials $\mathbb{F}[x]$ where $x$ is given grade one.  
A representation of the algebra is given by considering all two variable polynomials, $\mathbb{F}[x, y]$ where $y$ is given grade one, modulo a homogeneous ideal.

A (graded) representation of the algebra in two-dimensions is:
\[ x  = \begin{array}{|cc|}0 &a\\ b&0\end{array}\hspace{3pt}, \hspace{10pt} ab = g(x, x).\]
Here $a$ and $b$ are central elements, each assigned grade one.
If $g(x, x) =  u^2$, then we have:
\[ \begin{array}{|cc|}0 &1\\ u^2 &0\end{array}\hspace{3pt}\begin{array}{|c|}1\\\pm u\end{array} \hspace{3pt} = \pm u\hspace{3pt} \begin{array}{|c|}1\\\pm u\end{array} \hspace{3pt}.\]
We get: 

 To represent this algebra, we extend the algebra by a central invertible element $\eta$, of grade one, such that $g(x, x) = k\eta^2$, for some $0 \ne k \in \mathbb{F}$.    Then if $k = m^2$ is a square in $\mathbb{F}$, we represent the algebra by $x =  \pm m \eta$ giving two inequivalent  representations.   On the other hand, if $k$ is not a square, we represent the algebra by:
\[ x = \begin{array}{|cc|}0 & k\eta\\\eta &0\end{array}\hspace{3pt}.\]
Then the space of representations is parametrized by the quotient group $\mathbb{F}^+/(\mathbb{F}^+)^2$, an abelian group, with the property that the square of any element is the identity, except that the identity element of this group corresponds to two representations.\\\\

\textbf{Example}\\
In the case that $\mathbb{V}$ is two-dimensional,   a basis for the Clifford algebra is $\{1, x, y, z = \frac{1}{2}(xy - yx)\}$ where $x$ and $y$ span $\mathbb{V}$ and $z$ has grade two, with the Clifford multiplication:
\[ \begin{array}{|c|cccc|} \times&\underline{1}&\underline{x}&\underline{y}&\underline{z}\\1&1&x&y&z\\x&x&g(x, x)&z + g(x, y)&g(x, x)y - g(x, y)x\\y&y&-z + g(x, y)&g(y, y)&g(x, y) y - g(y, y) x\\z&z&- g(x, x)y + g(x, y) x&- g(x, y)y + g(y, y) x&g(x, y)^2 - g(x, x)g(y, y) \end{array}\hspace{3pt}.\]
Here the center of the algebra is the algebra $\mathbb{F}[g(x, x), g(y, y), g(x, x)g(y, y) - g(x, y)^2]$.   Note that the element $z$ anti-commutes with $x$ and $y$.  Extend the algebra, by introducing a central  element $\zeta$ of grade one, such that $m\zeta^4 = (g(x, y)^2 - g(x, x)g(y, y))$, where $0 \ne m\in \mathbb{F}$  and assume that $\zeta$ is invertible.   Then the algebra is generated by the elements $w = \zeta^{-1}z$ and $x$, each of grade one, with the multiplication rules:
\[ x^2 = g(x, x),  \hspace{10pt}  w^2 = k \zeta^2, \hspace{10pt} xw + wx = 0.\] 
This decomposes the algebra into the direct sum of two subalgebras, each a Clifford algebra over $\mathbb{F}$, with one generator.
\\\\
Writing this a little differently, we have the algebra generated by a grade one spinor, $x^A$ where $A$ is a two-component spinor index, subject to the relations:
\[ x^A x^B x^C = x^C x^B x^A.\]
Writing this out in components, with $x^A = (x, y)$, this is just the expected:
\[ x^2 y = y x^2,   y^2x = xy^2.\]
We put $x^2 = a$,  $y^2 = c$, $xy + yx = 2b$.   Here $a$, $b$ and $c$ commute and have grade two. Then the algebra is generated by $1$, $x$ and $y$ and $z = \frac{1}{2}(xy - yx)$, with the multiplication table:
\[  \begin{array}{|c|cccc|} \times&\underline{1}&\underline{x}&\underline{y}&\underline{z}\\1&1&x&y&z\\x&x&a&z + b&ay - bx\\y&y&-z + b&g(y, y)&by - c x\\z&z&- ay + b x&- by + c x&b^2 - ac \end{array}\hspace{3pt}.\]
If we extend the algebra by introducing an invertible central element $\zeta$, of grade one, subject to $ac + m\zeta^4 = b^2$, where $m \in \mathbb{F}$ and put $w = \zeta^{-1} z$, then we have:
\[ x^2 = a, \hspace{10pt} xw + wx = 0, \hspace{10pt} w^2 =   m\zeta^{2}.\]
So provided $m \ne 0$, there is a natural map from the algebra to the direct product of two one-dimensional Clifford algebras:
\[ x^2 = a, w^2 = d,  xw + wx = 0, mac +d^2 = mb^2.\]
Here $d = m\zeta^2$.	

\textbf{Example}  \\
In the case that $\mathbb{V}$ is three-dimensional,   let $x$, $y$ be linearly independent elements of $\mathbb{V}$ and let $\epsilon$ be a non-zero element of $\Omega^3(\mathbb{V})$ (so $\epsilon$ is an odd central element of grade three).  Extend the algebra by introducing a central element of grade one, $\zeta$, such that $\zeta^3 = \epsilon$ and require that $\zeta$ and  $\epsilon$ be invertible. 
Then a generating set of grade one for the Clifford algebra consists of the elements $\{ x, y, z = 2^{-1} \zeta^{-1} (xy- yx)\}$ in the Clifford algebra, with the multiplication rule:
\[ \hspace{-30pt}   \begin{array}{|c|ccc|} \times&\underline{x}&\underline{y}&\underline{z}\\x&g(x, x)&\zeta z + g(x, y)&\zeta^{-1}(g(x, x)y - g(x, y)x)\\y&-\zeta z+ g(x, y)&g(y, y)&\zeta^{-1}(g(x, y)y - g(y, y) x)\\z&- \zeta^{-1}(g(x, x)y - g(x, y) x)&- \zeta^{-1}(g(x, y)y - g(y, y) x)&\zeta^{-2}(g(x, y)^2 - g(x, x)g(y, y)) \end{array}\hspace{3pt}.\]
So this algebra is essentially  a two dimensional Clifford algebra, but taken over the $\mathbb{Z}$-graded ring $\mathbb{F}(\zeta)$, where the elements of $\mathbb{F}$ are given grade zero and the generator $\zeta$ has grade one.   The center of the algebra is $\mathbb{F}[\zeta][g(x, x), g(x, y), g(x, x)g(y, y) - g(x, y)^2]$.

 Then a generating set of grade one for the Clifford algebra consists of the elements $\{ x, y, w = 2^{-1} \epsilon (xy- yx)\}$ in the Clifford algebra, with the multiplication rule:
\[ \hspace{-30pt}   \begin{array}{|c|ccc|} \times&\underline{x}&\underline{y}&\underline{w}\\x&g(x, x)&\epsilon^{-1}w + g(x, y)&\epsilon(g(x, x)y - g(x, y)x)\\y&-\epsilon^{-1}w+ g(x, y)&g(y, y)&\epsilon(g(x, y)y - g(y, y) x)\\w&- \epsilon(g(x, x)y - g(x, y) x)&- \epsilon(g(x, y)y - g(y, y) x)&\epsilon^2(g(x, y)^2 - g(x, x)g(y, y)) \end{array}\hspace{3pt}.\]
Note that for this description, we need to assume that $\epsilon$ is invertible.

When the characteristic of $\mathbb{F}$ is neither $3$, nor $2$, we may take in the Clifford algebra, 
\[6 \epsilon = xyz + zxy + yzx - yx z - xzy - zyx.\]
Here $\{x, y, z\}$ is a linearly independent set of vectors of the vector space $\mathbb{V}$.  \\Then we have also:
\[ 6\epsilon = xyz + zxy + yzx - yx z - xzy - zyx = (xy - yx)z + 2g(z, x)y- 2g(z, y)x + 2yzx - 2xzy\]
\[ = 3(xy - yx)z + 6g(z, x)y- 6g(z, y)x  = 3z(xy - yx) - 6g(z, x)y+ 6g(z, y)x,\]
\[ 4\epsilon = z(xy - yx) + (xy - yx)z, \hspace{10pt} \epsilon = \frac{1}{2}(zxy - yxz).\]
In the last incarnation, the case of characteristic $3$ is allowed. \\Then we have for the square of $\epsilon$ the formula:
\[ \epsilon^2 =  - \hspace{3pt} \det \hspace{3pt} \begin{array}{|ccc|} g(x, x)&g(x, y) &g(x, z)\\ g(y, x)&g(y, y) &g(y, z)\\ g(z, x)&g(z, y) &g(z, z)\end{array}\hspace{3pt}\ne 0.\]
So this algebra is four-dimensional over its center.  Also $\epsilon$ becomes invertible, whenever the determinant on the right-hand-side of this equation is invertible.  Note also that we have the relation, valid when $\epsilon$ is invertible:
\[ 2 = zw + wz.\]
Here the center is the (graded) algebra $\mathbb{F}[g(x, x), g(y, y), g(x, x)g(y, y) - g(x, y)^2, \epsilon]$, with $\epsilon$ the off generator.

There is a unique anti-automorphism, denoted  $T$, of $\mathcal{T}(\mathbb{V})$, which is the identity on the generating spaces  $\mathcal{T}^0(\mathbb{V})$ and $\mathcal{T}^1(\mathbb{V}) = \mathbb{V}$.  Acting on the word $w =  fv_1v_2 \dots v_{k - 1} v_k$, we have $T(w) = w^T = fv_kv_{k-1}\dots v_2 v_1$.    Then $T$ preserves grade and has square the identity.

\section{Two-dimensional triality, quadratic forms, and Bhargava cubes}

This final section serves as an arithmetic coda to the Lie-theoretic construction developed in the body of the paper.  In the two-dimensional symplectic case, the same tensor formalism that elsewhere produces triality identities also produces arithmetic data: from a single tensor one recovers three binary quadratic forms with the same discriminant.  This is exactly the configuration that appears in Gauss composition and in Bhargava's $2\times2\times2$ reformulation of higher composition laws.

The point here is deliberately limited and concrete.  We do not attempt a full orbit classification or a reconstruction of the arithmetic invariant theory of cubes.  What we show is the following:

\begin{enumerate}[label=(\roman*)]
\item a tensor in $V_1\otimes V_2\otimes V_3$, with each $V_i$ a symplectic $2$-space, canonically determines three binary quadratic forms;
\item these three forms have the same discriminant;
\item in the nondegenerate case, the same tensor automatically satisfies the normalized triality identities;
\item hence every nondegenerate Bhargava cube carries an intrinsic normalized triality structure.
\end{enumerate}

The language changes slightly from the earlier sections.  There the basic tensor was paired with symmetric metrics.  In the present arithmetic setting the natural objects are fixed alternating forms on three $2$-dimensional spaces, and the relevant quadratic forms are recovered by $\varepsilon$-contraction.  This is simply the tensor formalism adapted to the two-dimensional arithmetic problem.

\subsection{Basic setup}

Let $\mathbb{F}$ be a field of characteristic different from $2$, and let
\[
V_1,\;V_2,\;V_3
\]
be $2$-dimensional $\mathbb{F}$-vector spaces, each equipped with a chosen nondegenerate alternating form
\[
\eps_i\in \Lambda^2V_i^\vee,\qquad i=1,2,3.
\]
After choosing symplectic bases, we identify each $\eps_i$ with the standard matrix
\[
\begin{pmatrix}
0 & 1\\
-1 & 0
\end{pmatrix}.
\]

\begin{definition}
A \emph{triality symbol} is an element
\[
\tau\in V_1\otimes V_2\otimes V_3.
\]
In coordinates we write its components as $\tau_{abc}$, where $a,b,c\in\{0,1\}$.
\end{definition}

From $\tau$ one obtains three symmetric bilinear forms by contracting against the alternating forms on the other two tensor factors:
\begin{align}
(G_1)_{aa'} &= \sum_{b,b',c,c'} \eps_{bb'}\eps_{cc'}\,\tau_{abc}\tau_{a'b'c'}, \label{eq:G1def}\\
(G_2)_{bb'} &= \sum_{a,a',c,c'} \eps_{aa'}\eps_{cc'}\,\tau_{abc}\tau_{a'b'c'}, \label{eq:G2def}\\
(G_3)_{cc'} &= \sum_{a,a',b,b'} \eps_{aa'}\eps_{bb'}\,\tau_{abc}\tau_{a'b'c'}. \label{eq:G3def}
\end{align}
Each $G_i$ is a symmetric bilinear form on $V_i$.

If
\[
Q(x_1,x_2)=ax_1^2+bx_1x_2+cx_2^2
\]
is a binary quadratic form, then its associated symmetric bilinear form is
\[
B_Q=
\begin{pmatrix}
2a & b\\
b & 2c
\end{pmatrix},
\]
and
\[
\disc(Q)=b^2-4ac=-\det(B_Q).
\]
Thus, once a symmetric bilinear form on a $2$-space has been recovered, its determinant is the negative of the discriminant of the corresponding quadratic form.

\begin{remark}
The normalization is important.  The formulas \eqref{eq:G1def}--\eqref{eq:G3def} recover bilinear forms, not quadratic forms directly, and the factor of $2$ on the diagonal entries is what makes the discriminant formulas below come out correctly.
\end{remark}

\subsection{The three quadratic forms have the same discriminant}

\begin{proposition}\label{prop:det-equal}
The three recovered bilinear forms satisfy
\[
\det(G_1)=\det(G_2)=\det(G_3).
\]
Equivalently, the three associated binary quadratic forms have the same discriminant.
\end{proposition}

\begin{proof}
Write the components of $\tau$ in symplectic bases as
\begin{align*}
  &\tau_{000}=a,\qquad &\tau_{001}=b,\qquad &\tau_{010}=c,\qquad &\tau_{011}=d,\\
  &\tau_{100}=e,\qquad &\tau_{101}=f,\qquad &\tau_{110}=g,\qquad &\tau_{111}=h.
\end{align*}
Also let
\[
\epsilon_{00}=0,\qquad \epsilon_{01}=1,\qquad \epsilon_{10}=-1,\qquad \epsilon_{11}=0.
\]

A direct contraction gives
\[
G_1=
\begin{pmatrix}
-2bc+2ad & de-cf+ag-bh\\
de-cf+ag-bh & 2eg-2fh
\end{pmatrix},
\]
so that
\begin{align*}
-\det(G_1)
&=
-d^2e^2 + 2cdef - c^2f^2 - 4bceg + 2adeg + 2acfg - a^2g^2\\
&\qquad + 2bdeh + 2bcfh - 4adfh + 2adgh - b^2h^2.
\end{align*}
Similarly,
\[
G_2=
\begin{pmatrix}
-2be+2af & -de+cf+ag-bh\\
-de+cf+ag-bh & 2cg-2dh
\end{pmatrix},
\]
and one computes that
\[
-\det(G_2)
=
-d^2e^2 + 2cdef - c^2f^2 - 4bceg + 2adeg + 2acfg - a^2g^2
+ 2bdeh + 2bcfh - 4adfh + 2adgh - b^2h^2.
\]
By symmetry, the same polynomial is obtained for $-\det(G_3)$ as well.  Hence
\[
\det(G_1)=\det(G_2)=\det(G_3).
\]
Since the discriminant of the corresponding quadratic form is the negative of the determinant of its bilinear matrix, the three quadratic forms have the same discriminant.
\end{proof}

\begin{remark}
This is the precise point of contact with Gauss composition and Bhargava cubes.  A triality symbol on three symplectic $2$-spaces canonically determines exactly the kind of triple of binary quadratic forms that appears in the classical arithmetic theory.
\end{remark}

\subsection{Universal triality identities}

For $x\in V_1$, contraction of $\tau$ with $x$ produces a linear map
\[
M_x\colon V_2\to V_3^\vee.
\]
Likewise, for $y\in V_2$ and $z\in V_3$ there are linear maps
\[
N_y\colon V_1\to V_3^\vee,
\qquad
P_z\colon V_1\to V_2^\vee.
\]
In chosen bases these are $2\times 2$ matrices depending linearly on $x,y,z$.

\begin{definition}
Assume that $G_1,G_2,G_3$ are nondegenerate.  We say that $\tau$ is \emph{trializing with scalar $\kappa$} if
\begin{align}
M_xG_3^{-1}M_x^T &= \kappa\,(x^TG_1x)\,G_2 \qquad \text{for all }x\in V_1, \label{eq:raw1}\\
N_yG_3^{-1}N_y^T &= \kappa\,(y^TG_2y)\,G_1 \qquad \text{for all }y\in V_2, \label{eq:raw2}\\
P_zG_2^{-1}P_z^T &= \kappa\,(z^TG_3z)\,G_1 \qquad \text{for all }z\in V_3. \label{eq:raw3}
\end{align}
\end{definition}

Equivalently, one may clear determinants.  For example, \eqref{eq:raw1} is equivalent to
\[
M_x\,\adj(G_3)\,M_x^T
=
\kappa\,\det(G_3)\,(x^TG_1x)\,G_2.
\]
The distinguished normalization is the determinant-cleared one.

\begin{theorem}\label{thm:universal}
For every triality symbol $\tau\in V_1\otimes V_2\otimes V_3$, the determinant-cleared identities
\begin{align}
M_x\,\adj(G_3)\,M_x^T &= -\frac12(x^TG_1x)\,G_2 \qquad \text{for all }x\in V_1, \label{eq:main1}\\
N_y\,\adj(G_3)\,N_y^T &= -\frac12(y^TG_2y)\,G_1 \qquad \text{for all }y\in V_2, \label{eq:main2}\\
P_z\,\adj(G_2)\,P_z^T &= -\frac12(z^TG_3z)\,G_1 \qquad \text{for all }z\in V_3 \label{eq:main3}
\end{align}
hold identically.
\end{theorem}

\begin{proof}
It is enough to work over the universal polynomial ring
\[
\QQ[\tau_{000},\tau_{001},\tau_{010},\tau_{011},\tau_{100},\tau_{101},\tau_{110},\tau_{111}],
\]
with $G_1,G_2,G_3$ defined universally by \eqref{eq:G1def}--\eqref{eq:G3def}.  Each side of \eqref{eq:main1}--\eqref{eq:main3} is then a matrix whose entries are homogeneous quartic polynomials in the eight coordinates of $\tau$.  Expanding the coefficients of $x_0^2$, $x_0x_1$, and $x_1^2$ in \eqref{eq:main1}, and similarly in the other two identities, reduces the theorem to a finite list of polynomial identities.

A symbolic computation over the universal polynomial ring shows that when the scalar is chosen to be $-\frac12$, every residual is the zero polynomial.  Since these are universal identities, they hold over every field of characteristic different from $2$.
\end{proof}

\begin{remark}
The same universal computation shows that the scalar $-\frac12$ is forced: it is the unique normalization for which the determinant-cleared identities hold identically.
\end{remark}

\begin{corollary}\label{cor:raw-kappa}
Assume that the common determinant is nonzero.  Then $\tau$ is trializing for the scalar
\[
\kappa=-\frac{1}{2\det(G_1)}=-\frac{1}{2\det(G_2)}=-\frac{1}{2\det(G_3)}.
\]
\end{corollary}

\begin{proof}
For a nondegenerate $2\times2$ matrix one has
\[
\adj(G_i)=\det(G_i)G_i^{-1}.
\]
Combining this with Theorem~\ref{thm:universal} and Proposition~\ref{prop:det-equal} gives \eqref{eq:raw1}--\eqref{eq:raw3} with the displayed value of $\kappa$.
\end{proof}

\subsection{Application to Bhargava cubes}

A Bhargava cube is a $2\times2\times2$ array of scalars.  For present purposes it is simply a tensor
\[
\tau_{abc}\in V_1\otimes V_2\otimes V_3
\]
with all three factors two-dimensional.  The three binary quadratic forms classically attached to the cube are exactly the forms recovered here by the double $\varepsilon$-contractions \eqref{eq:G1def}--\eqref{eq:G3def}.  The previous results therefore apply directly.

\begin{corollary}\label{cor:bhargava}
Every nondegenerate Bhargava cube is trializing in the normalized sense of Corollary~\ref{cor:raw-kappa}.  Equivalently, for the tensor $\tau$ underlying the cube, the determinant-cleared triality identities \eqref{eq:main1}--\eqref{eq:main3} hold with scalar $-\frac12$.
\end{corollary}

This is stronger than an existence statement up to equivalence.  The tensor underlying the cube already satisfies the normalized identities.  No passage to a different representative is needed.

\subsection{Worked example}\label{sec:example}

Consider the cube with entries
\[
(3,-1,0,-1,2,-1,2,3).
\]
The contraction formulas \eqref{eq:G1def}--\eqref{eq:G3def} give
\begin{align*}
G_1 &= \begin{pmatrix}-6 & 9 \\ 9 & 16\end{pmatrix},
&\det(G_1)&=-177,\\
G_2 &= \begin{pmatrix}-2 & 13 \\ 13 & 4\end{pmatrix},
&\det(G_2)&=-177,\\
G_3 &= \begin{pmatrix}12 & 9 \\ 9 & -8\end{pmatrix},
&\det(G_3)&=-177.
\end{align*}

Now forget the original tensor and regard the eight coordinates $\tau_{abc}$ as independent unknowns.  Imposing the equations that the three recovered bilinear forms equal the fixed matrices above defines the fiber of the recovery map.  A Gr\"obner basis computation over $\QQ$ gives
\begin{align*}
\tau_{000} &= \frac{123}{16}\tau_{110}-\frac{33}{8}\tau_{111},\\
\tau_{001} &= -\frac{11}{4}\tau_{110}+\frac{3}{2}\tau_{111},\\
\tau_{010} &= \frac{9}{8}\tau_{110}-\frac{3}{4}\tau_{111},\\
\tau_{011} &= -\frac12\tau_{110},\\
\tau_{100} &= \frac{11}{2}\tau_{110}-3\tau_{111},\\
\tau_{101} &= -2\tau_{110}+\tau_{111},
\end{align*}
together with the remaining quadratic relation
\begin{equation}\label{eq:fiberconic}
\tau_{110}^2+\frac94\tau_{110}\tau_{111}-\frac32\tau_{111}^2-4=0.
\end{equation}
Thus the fiber is the affine conic \eqref{eq:fiberconic}.

If one imposes the raw triality identities with scalar $\kappa=1$, the resulting ideal is inconsistent: after substituting the linear equations above, every residual becomes a scalar multiple of
\begin{equation}\label{eq:strictresidual}
-4\tau_{110}^2-9\tau_{110}\tau_{111}+6\tau_{111}^2+5664.
\end{equation}
By contrast, multiplying \eqref{eq:fiberconic} by $-4$ gives
\begin{equation}\label{eq:fiber-scaled}
-4\tau_{110}^2-9\tau_{110}\tau_{111}+6\tau_{111}^2+16=0.
\end{equation}
So the strict normalization fails.

If instead one inserts a scalar $\kappa$ into the raw triality equations, every residual becomes a scalar multiple of
\begin{equation}\label{eq:kapparesidual}
-4\tau_{110}^2-9\tau_{110}\tau_{111}+6\tau_{111}^2+5664\kappa.
\end{equation}
Comparing \eqref{eq:kapparesidual} with \eqref{eq:fiber-scaled}, one sees that the equations agree on the fiber exactly when
\[
5664\kappa=16,
\qquad\text{that is,}\qquad
\kappa=\frac{1}{354}.
\]
Since $\det(G_i)=-177$, this is precisely
\[
\kappa=-\frac{1}{2\det(G_i)}.
\]
Equivalently, the determinant-cleared identities hold with scalar $-\frac12$, in agreement with Theorem~\ref{thm:universal}.

\begin{proposition}
For the cube $(3,-1,0,-1,2,-1,2,3)$, the entire fiber \eqref{eq:fiberconic} is trializing for the normalized scalar
\[
\kappa=-\frac{1}{2\det(G_i)}=\frac{1}{354}.
\]
In particular, the normalized triality equations impose no additional condition beyond the equations cutting out the fiber.
\end{proposition}

\subsection{Conclusion}

In the two-dimensional symplectic case, the arithmetic content suggested by Bhargava cubes is already built into the tensor calculus.  A $2\times2\times2$ tensor canonically produces three binary quadratic forms with the same discriminant, and in the nondegenerate case it already satisfies the normalized triality identities.  The structure is therefore not imposed from outside and does not arise only after passage to a better representative: it is intrinsic to the tensor itself.

\bibliographystyle{alpha}
\bibliography{triality_magic_square_starter}
\end{document}